\documentclass[]{article}

\usepackage{bm} 
\usepackage{tabularx} 
\usepackage{algorithm2e}
\usepackage{dsfont}
\usepackage{amsmath}  
\usepackage{amsthm}
\usepackage{multicol}
\usepackage{changepage}
\usepackage{amssymb}
\usepackage{tikz}
\usetikzlibrary{decorations.pathreplacing,calligraphy}
\usepackage{graphicx} 
\usepackage[margin=1.3in,letterpaper]{geometry} 
\usepackage{cite} 
\usepackage[colorlinks,allcolors=blue]{hyperref} 
\usepackage{esint}
\usepackage{caption}
\usepackage{subcaption}

\newtheorem{theorem}{Theorem}[section]
\newtheorem{defi}[theorem]{Definition}
\newtheorem{prop}[theorem]{Proposition}
\newtheorem{lemma}[theorem]{Lemma}
\newtheorem{cor}[theorem]{Corollary}
\newtheorem{remark}[theorem]{Remark}
\newtheorem{example}[theorem]{Example}

\newcommand*{\clos}{\overline}

\newcommand*{\R}{\mathbb{R}}
\newcommand*{\N}{\mathbb{N}}

\newcommand{\cP}{\mathcal{P}}
\newcommand*{\cX}{\mathcal{X}}

\newcommand*{\cM}{\mathcal{M}}

\newcommand*{\cL}{\mathcal{L}}
\newcommand*{\Omb}{\clos{\Omega}}

\renewcommand{\epsilon}{\varepsilon}
\DeclareMathOperator*{\argmin}{arg\,min}

\DeclareMathOperator{\Shrink}{Shrink}
\DeclareMathOperator{\spn}{span}
\DeclareMathOperator{\Img}{Im}

\newcommand{\id}{\mathrm{id}}

\renewcommand{\div}{\mathrm{div}}
\renewcommand{\d}{\mathrm{d}}
\newcommand{\spt}{\mathrm{spt}}
\newcommand{\loc}{\mathrm{loc}}

\newcommand{\peps}{p_\epsilon}
\newcommand{\Jeps}{J_\epsilon}
\newcommand{\ueps}{u^\epsilon}
\newcommand{\neps}{\nu^\epsilon}
\newcommand{\sigeps}{\sigma^\epsilon}
\newcommand{\bnu} {\boldsymbol{\nu}}
\newcommand{\bl} {\boldsymbol{\lambda}}
\newcommand{\bx} {\boldsymbol{x}}
\newcommand{\by} {\boldsymbol{y}}
\newcommand{\be} {\boldsymbol{e}}
\newcommand{\bdel} {\boldsymbol{\delta}}
\newcommand{\LL} {\mathcal{L}}
\newcommand{\tnu}{\widetilde{\nu}}
\newcommand{\btnu}{\widetilde{\bnu}}

\newcommand{\eps}{\varepsilon}

\DeclareMathOperator{\Mop}{\mathsf{M}}
\newcommand{\Ml}{\Mop_{\boldsymbol{\lambda}}}
\DeclareMathOperator{\Med}{\mathsf{Med}}
\newcommand{\Medln}{\Med_{\boldsymbol{\lambda}^n}}
\newcommand{\Medl}{\Med_{\boldsymbol{\lambda}}}
\DeclareMathOperator{\HMed}{\mathsf{HMed}}
\newcommand{\HMedl}{\HMed_{\boldsymbol{\lambda}}}
\DeclareMathOperator{\VMed}{\mathsf{VMed}}
\newcommand{\VMedl}{\VMed_{\boldsymbol{\lambda}}}
\DeclareMathOperator{\Lip}{Lip}

\usepackage{xcolor}

\usepackage{tcolorbox}

\title{Wasserstein medians: robustness, PDE characterization and numerics}
\author{G. Carlier\thanks{CEREMADE, Université Paris Dauphine, PSL, and INRIA-Paris, MOKAPLAN, \newline email: \texttt{carlier@ceremade.dauphine.fr}} \and E. Chenchene\thanks{Department of Mathematics and Scientific Computing, University of Graz, Graz, Austria. \newline email: \texttt{enis.chenchene@uni-graz.at}} \and K. Eichinger\thanks{CMAP, École Polytechnique \newline email: \texttt{katharina.eichinger@polytechnique.edu}}}

\begin{document}

\maketitle
\begin{abstract}
We investigate the notion of Wasserstein median as an alternative to the Wasserstein barycenter, which has become popular but may be sensitive to outliers. In terms of robustness to corrupted data, we indeed show that Wasserstein medians have a breakdown point of  approximately $\frac{1}{2}$.  We give explicit constructions of Wasserstein medians in dimension one which enable us to obtain $L^p$ estimates (which do not hold in higher dimensions).  We also address dual and multimarginal reformulations. In convex subsets of $\R^d$, we  connect Wasserstein medians to a minimal (multi)  flow problem \`a la Beckmann and  a system of PDEs of Monge--Kantorovich-type, for which we propose a $p$-Laplacian approximation.  Our analysis eventually leads to a new numerical method to compute Wasserstein medians, which is based on a Douglas--Rachford scheme applied to the minimal flow formulation of the problem.  

\end{abstract}


\section{Introduction}

	The notions of mean and median are well-known to be of variational nature. For instance, the arithmetic mean of a sample composed by $N$ points in $\mathbb{R}^d$ is the minimizer of the sum of the \emph{squared} Euclidean distances  to the sample points. Minimizing a weighted sum of distances to the sample points, one gets a notion of weighted medians, which in the literature is commonly referred to as  Torricelli--Fermat--Weber points or geometric medians. As  pointed out by Maurice Fréchet in his seminal work  \cite{Frechet1948}, these definitions can be generalized to any metric space $\left(\cX,d\right)$, yielding the notion of Fréchet mean and Fréchet median (or in general typical element). 

    The concept of Wasserstein barycenter, which corresponds to Fr\'echet means over the Wasserstein space of probability measures with finite second moments and equipped with the quadratic Wasserstein distance, was introduced and extensively studied in \cite{agueh_barycenters_2011}. Since then, research on Wasserstein barycenters has expanded in various directions. For instance, investigations have been conducted on Riemannian manifolds \cite{KimPass}, population barycenters involving possibly infinitely many measures \cite{LeGouicLoubes}, and Radon spaces \cite{kroshnin_frechet_2018}. The concept has gained popularity as a valuable tool for meaningful geometric interpolation between probability measures, finding applications in diverse fields such as image synthesis \cite{RabinPeyreDelon}, template estimation \cite{boissard_distributions_2015}, bayesian learning \cite{Fontbona}, and statistics \cite{Zemel}. Despite the inherent complexity of computing Wasserstein barycenters \cite{Altschuler}, numerical methods based on entropic regularization and the Sinkhorn algorithm have demonstrated their efficiency in calculating these interpolations \cite{cuturi_fast_2014, peyrecuturi, bccmnp15}.

    \begin{figure}[!t]
	    \centering
	    \includegraphics[width=\linewidth]{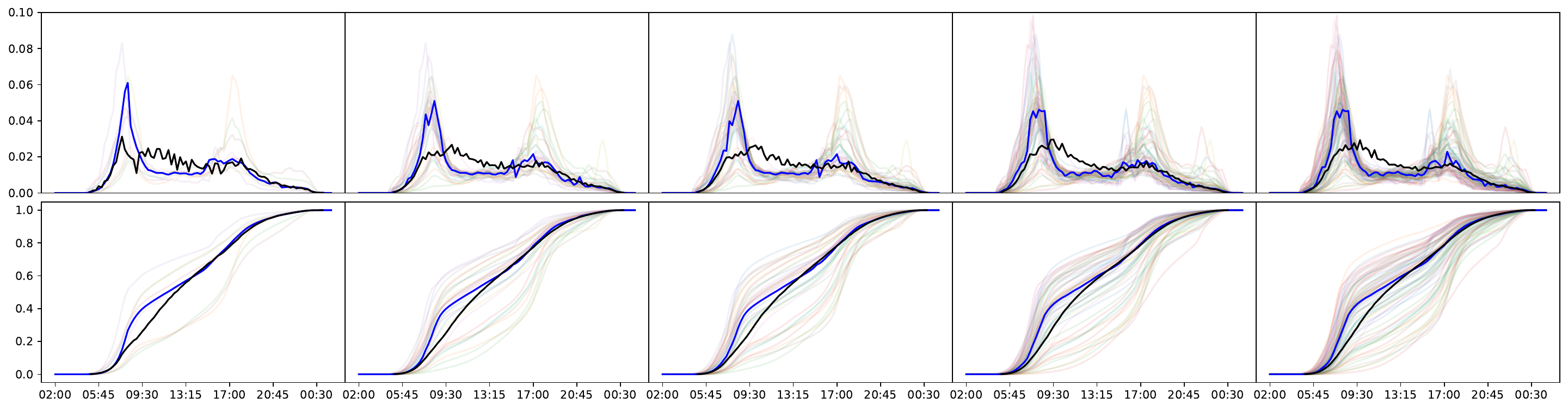}
	    \caption{Superposition of a Wasserstein median (blue), a Wasserstein barycenter (black) and the corresponding sample of $N = 9, \ 29, \ 39, \ 59,\ 81$ one-dimensional histograms. Each histogram represents the daily attendance frequency of some London underground stations\protect\footnotemark. Second row:~the corresponding cumulative distribution functions.}
	    \label{fig:one_dimensional_wm}
	\end{figure}

Following Fr\'echet's metric viewpoint, medians in Wasserstein spaces can be constructed as follows.
Given $p\geq 1$, positive weights $\lambda_1, \dots \lambda_N$ and  probability measures $\nu_1, \dots, \nu_N$ with finite $p$-moments, over a metric space $\cX$, the corresponding medians  are obtained by minimizing $\sum_{i=1}^N \lambda_i W_p(\nu_i,\cdot)$ where $W_p$ denotes the Wasserstein distance of order $p$. In the following, we will restrict ourselves to the case $p=1$. Indeed, even though the case $p>1$ might be natural it leads to delicate non-convex problems (see \cite{altschuler2021}) which are beyond the scope of the paper. On the contrary, the case $p=1$ is a special instance of the matching for teams problem \cite{carlier_matching_2010} and thus admits a linear programming formulation, which makes a clear connection to the Torricelli--Fermat--Weber points on the ambient space (see Section \ref{sec:dual_multimarginal}). Therefore in the present paper we will  investigate in details minimizers of  $\sum_{i=1}^N \lambda_i W_1(\nu_i,\cdot)$, which from now on we call \emph{Wasserstein medians}.  Our primary motivation for studying these objects comes from the following question: does the well-known \textit{robustness} of geometric medians extend to Wasserstein medians? Consider for instance the problem of averaging the daily attendance frequency of some London underground stations\footnotetext{Tfl open data \url{https://tfl.gov.uk/info-for/open-data-users}, accessed on June 29, 2023.} as in Figure \ref{fig:one_dimensional_wm} or the five pictures on the left of Figure \ref{fig:robustness}. It is pretty clear in these examples that Wasserstein medians show some sort of robustness, and that in general they should behave quite differently from the barycenter. 
	
	Our objective is to further explore the notion of Wasserstein median with a first focus on stability and robustness. We also investigate in depth the one-dimensional case where special constructions (which we call vertical and horizontal selections) select medians which inherit properties of the sample measures, in particular, we  show  that if all the sample measures $\nu_i$ are absolutely continuous with  densities bounded by some $M_i$, then there exists a Wasserstein median with a density bounded by $\max_i M_i$, which, as we will show later (Example \ref{ex:counterex-reg}),  cannot be true in higher dimensions. For more general situations,  we  present some  general tools to study Wasserstein medians, such as multi-marginal and dual formulations for the initial convex minimization problem. To the best of our knowldege,  Wasserstein medians have not been very much investigated even in the Euclidean setting with more than two sample measures, however related optimal matching problems (with two sample measures and additional constraints) have been studied in \cite{MazonRossiToledo}, \cite{igbida} and \cite{buttazzo2022wasserstein}. In the Euclidean setting in several dimensions, we  also characterize medians by a minimal flow problem \`a la Beckmann \cite{Beckmann} and a system of PDEs of Monge--Kantorovich type. This analysis leads to a new numerical method to compute Wasserstein medians, which is based on a Douglas--Rachford scheme applied to the minimal flow formulation of the problem. 

	\begin{figure}[!t]
    \centering
	\begin{subfigure}{0.2\textwidth}
	\centering
	\includegraphics[width=1\textwidth]{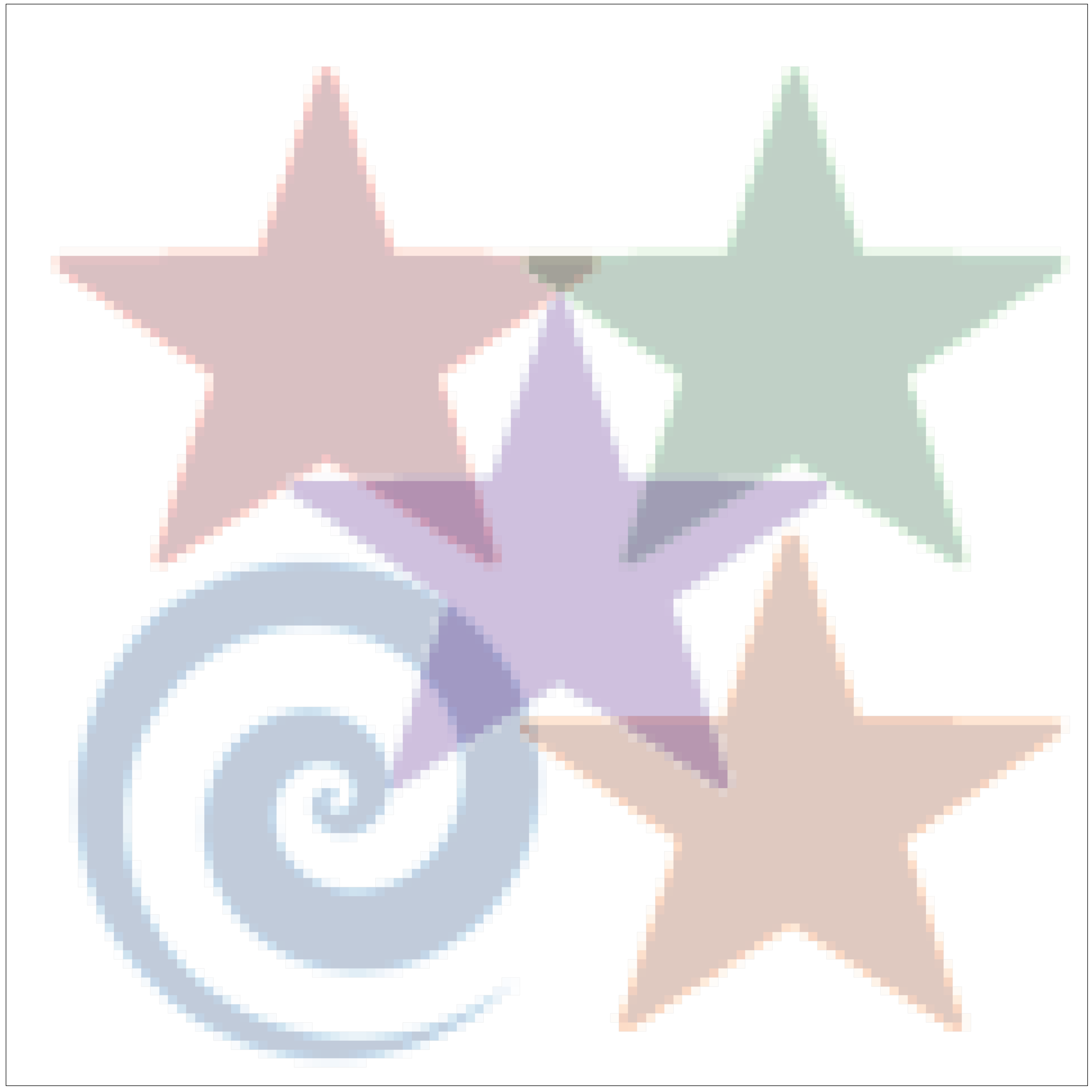}
	\caption{Sample measures.}
	\end{subfigure}
	\begin{subfigure}{0.2\textwidth}
		\centering
		\includegraphics[width=1\linewidth]{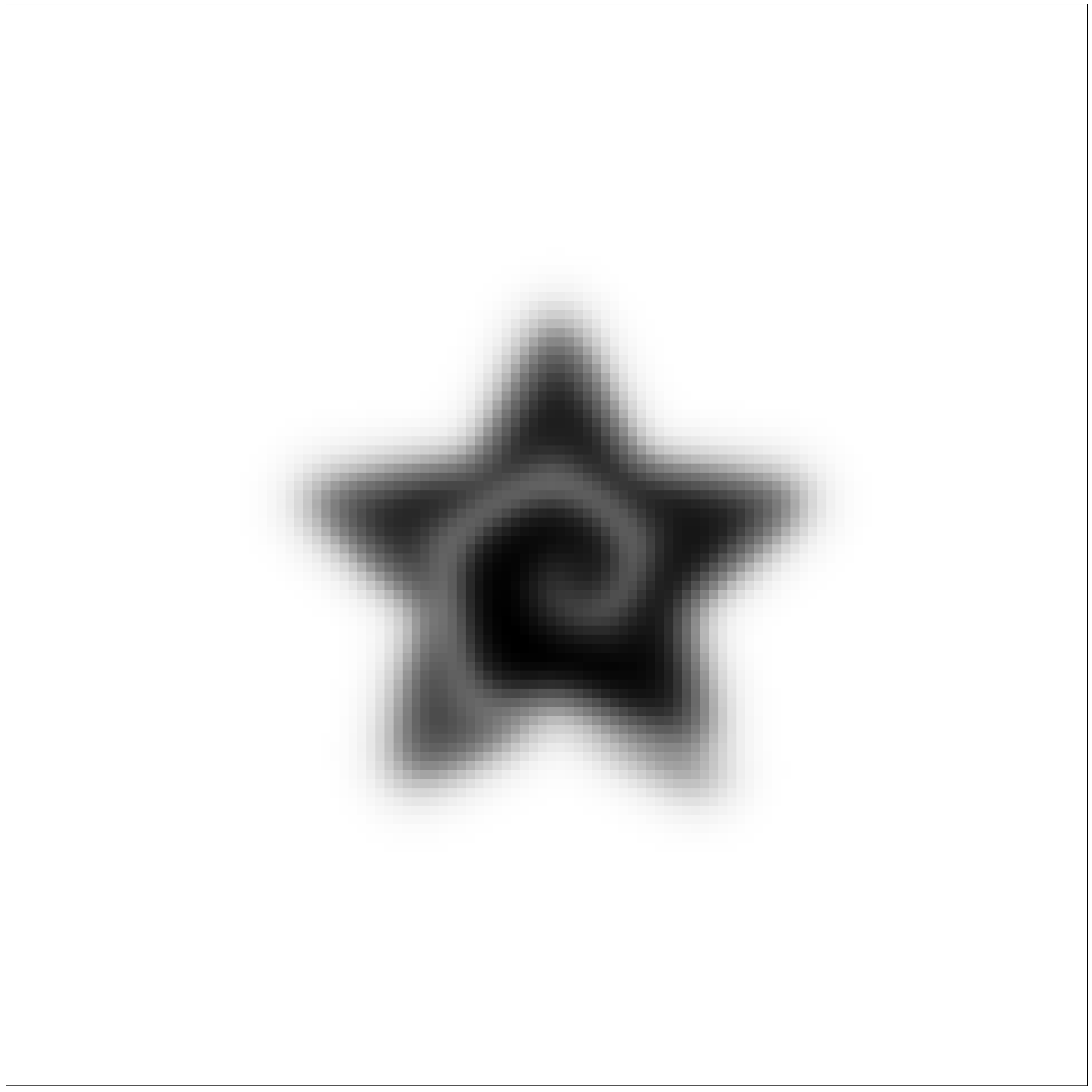}
	\caption{Barycenter.}
	\end{subfigure}
	\begin{subfigure}{0.2\textwidth}
		\centering
		\includegraphics[width=1\linewidth]{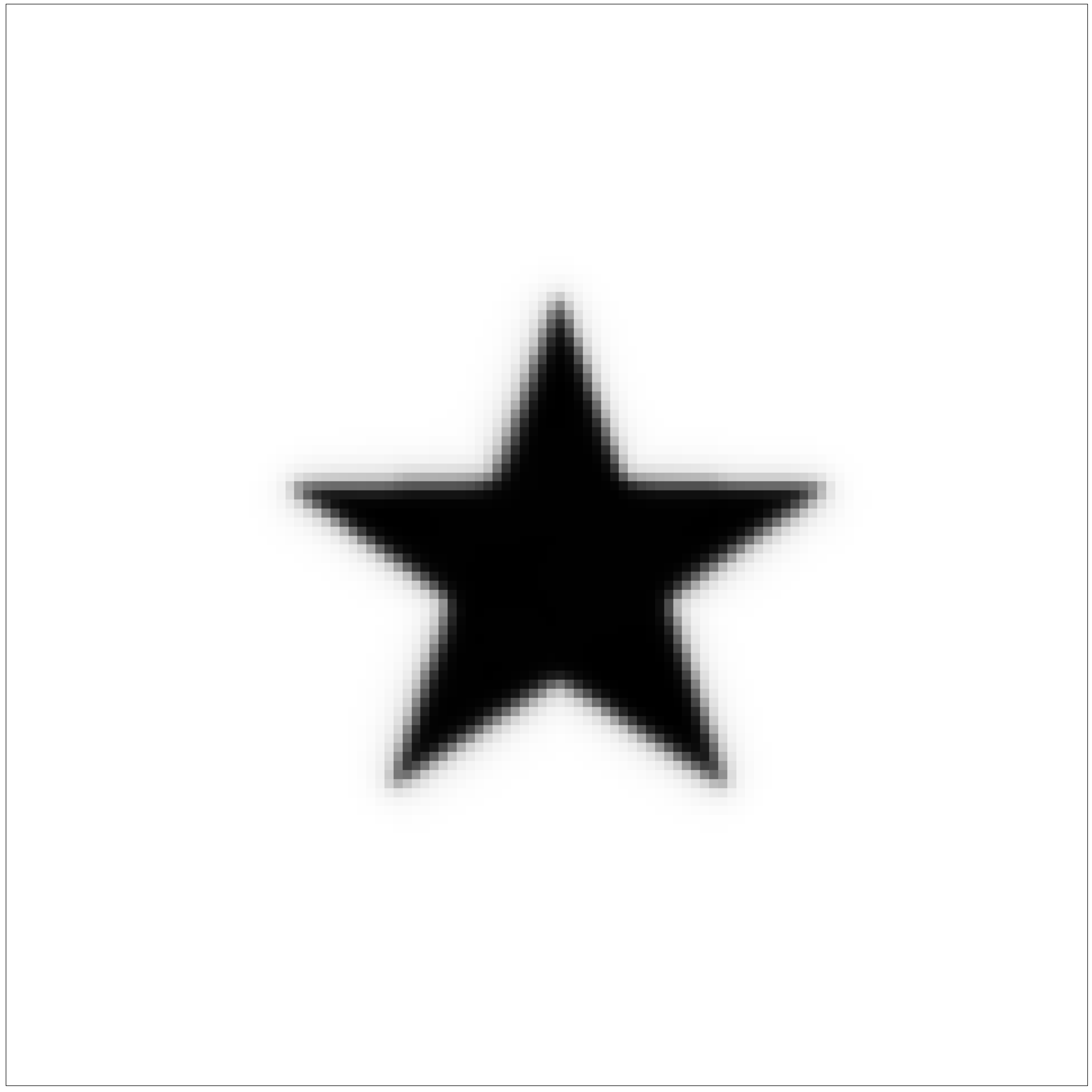}
	\caption{Median.}
	\end{subfigure}
	\caption{Comparison between a Wasserstein barycenter and a Wasserstein median for a sample of five measures computed with Sinkhorn (cf., Section \ref{sec:numerics}) in $1000$ iterations.}
	\label{fig:robustness}
    \end{figure}
	
	The  paper is organized as follows: in Section \ref{sec:formulation_existence}, we introduce the problem, show existence of Wasserstein medians and consider some basic examples. In Section \ref{sec:stability_robustness}, we discuss the stability of the notion subject to perturbations of the sample measures and prove that the \emph{break-down point} of the Wasserstein median problem with uniform weights is at least $1/2$, i.e.~to \emph{drastically corrupt} the estimation of the Wasserstein median one has to modify at least half of the sample measures. In Section \ref{sec:one_dim}, we focus on the one-dimensional case and emphasize the properties of medians which we call vertical and horizontal median selections. In Section \ref{sec:dual_multimarginal}, we present dual and multi-marginal formulations of the problem. In Section \ref{sec:beckmann}, we use a minimal flow formulation of the Wasserstein median problem to derive a system of Monge--Kantorovich type PDEs that characterizes medians. We also describe an approximation by a system of $p$-Laplace equations. We conclude in Section \ref{sec:numerics} with a brief description of the numerical methods we implemented to obtain the various figures in this paper and present a new one based on a Douglas--Rachford scheme on the flow formulation.

	\begin{figure}[!t]
	
    \centering
	\begin{subfigure}{0.45\textwidth}
	\centering
	\includegraphics[width=1\textwidth]{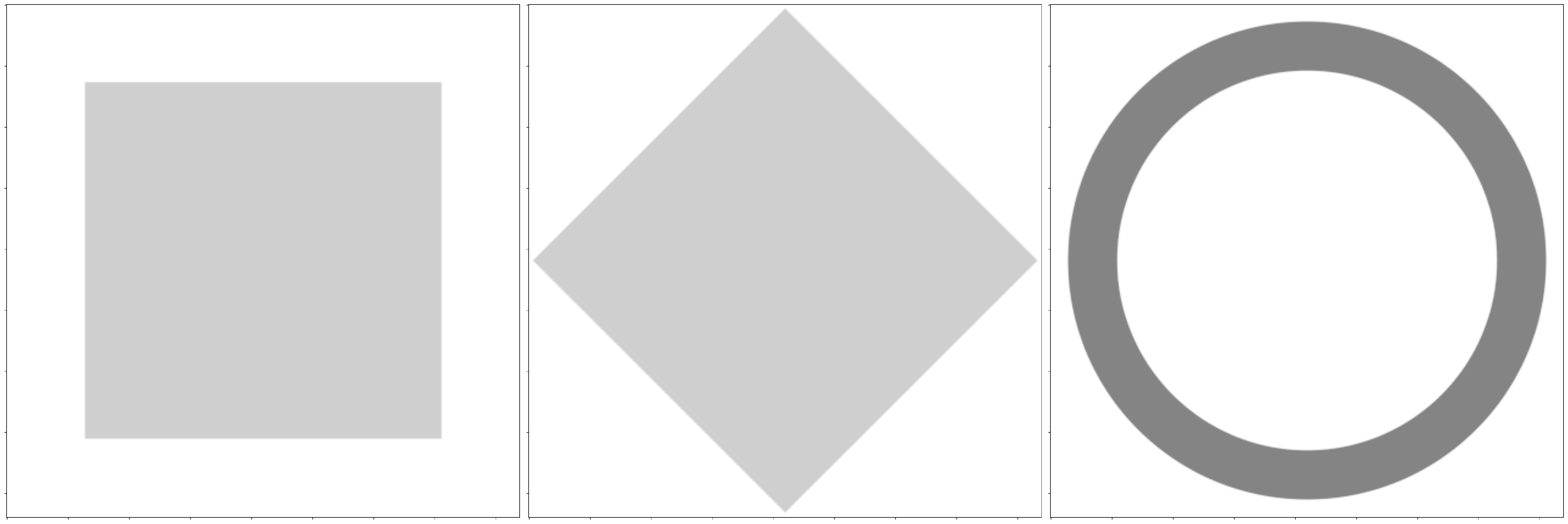}
	\end{subfigure}
	\hspace{1cm}
	\begin{subfigure}{0.25\textwidth}
	\centering
	\includegraphics[width=0.7\textwidth]{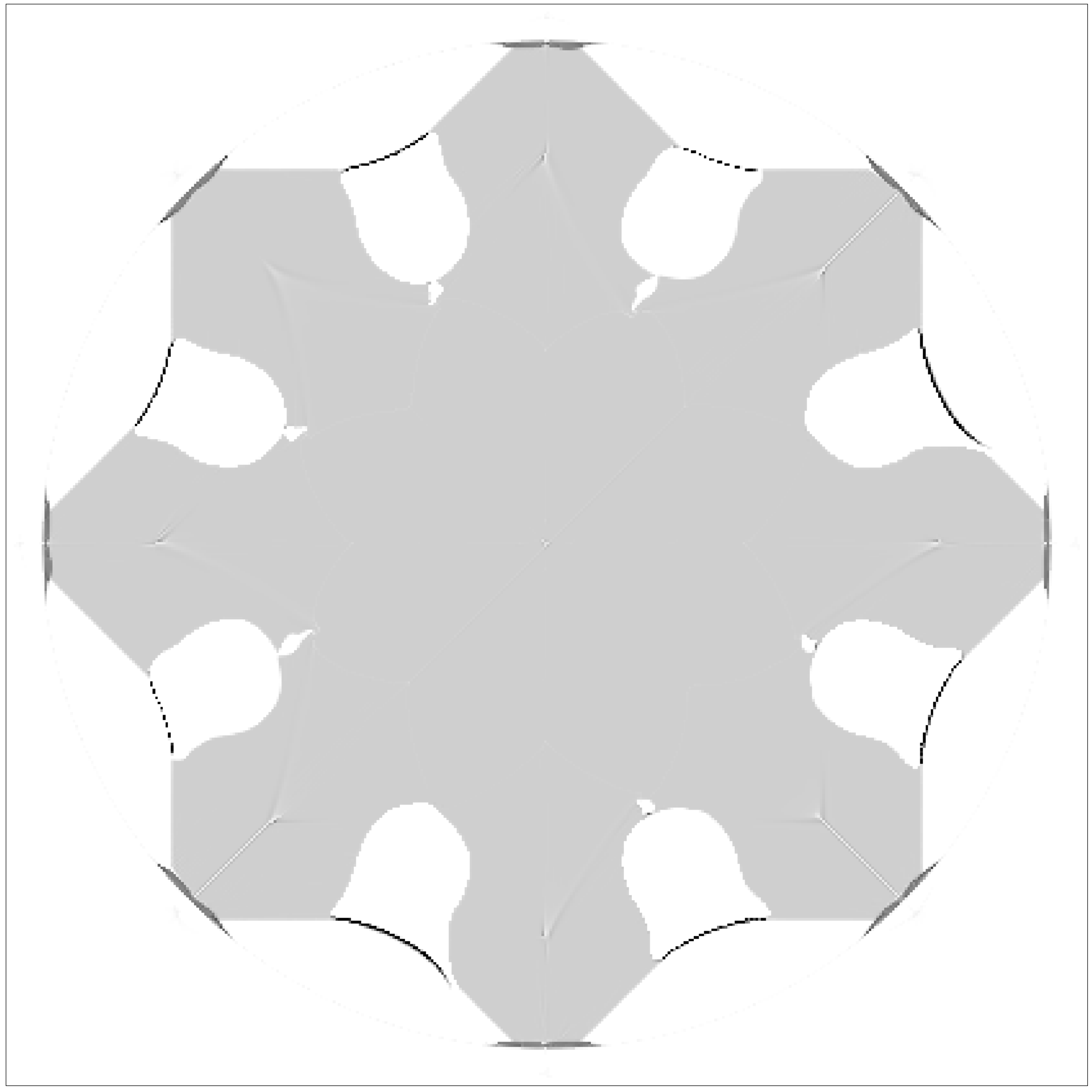}
	\end{subfigure}

    \centering
	\begin{subfigure}{0.45\textwidth}
	\centering
	\includegraphics[width=1\textwidth]{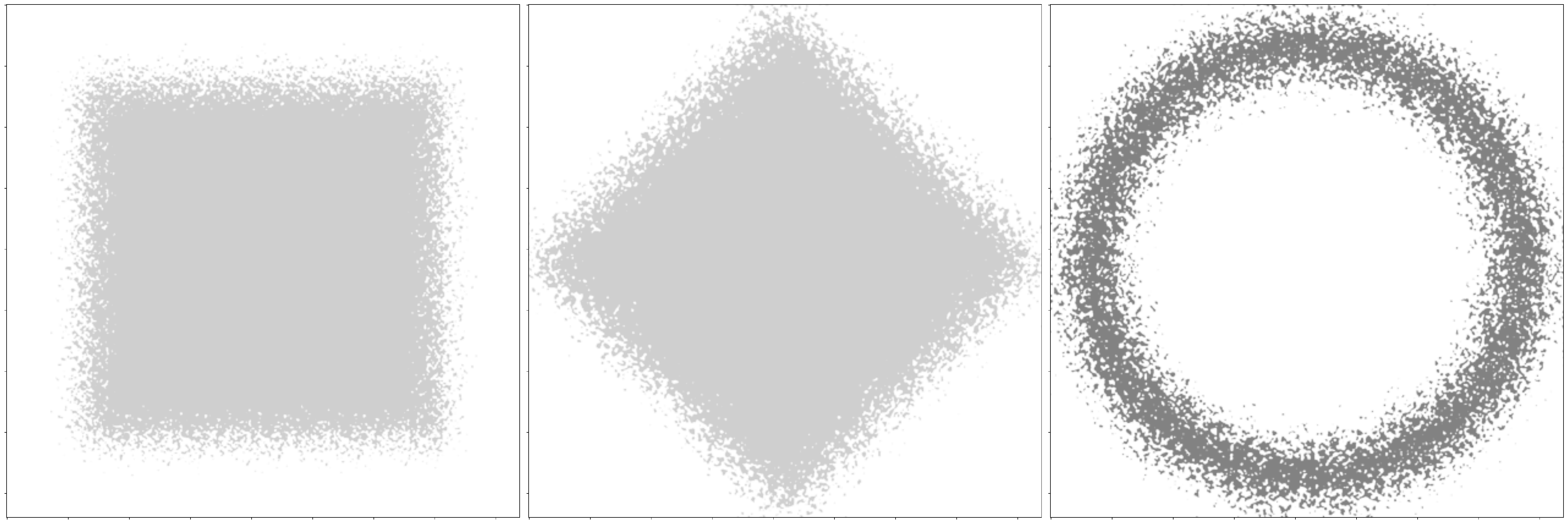}
	\end{subfigure}
	\hspace{1cm}
	\begin{subfigure}{0.25\textwidth}
	\centering
	\includegraphics[width=0.7\textwidth]{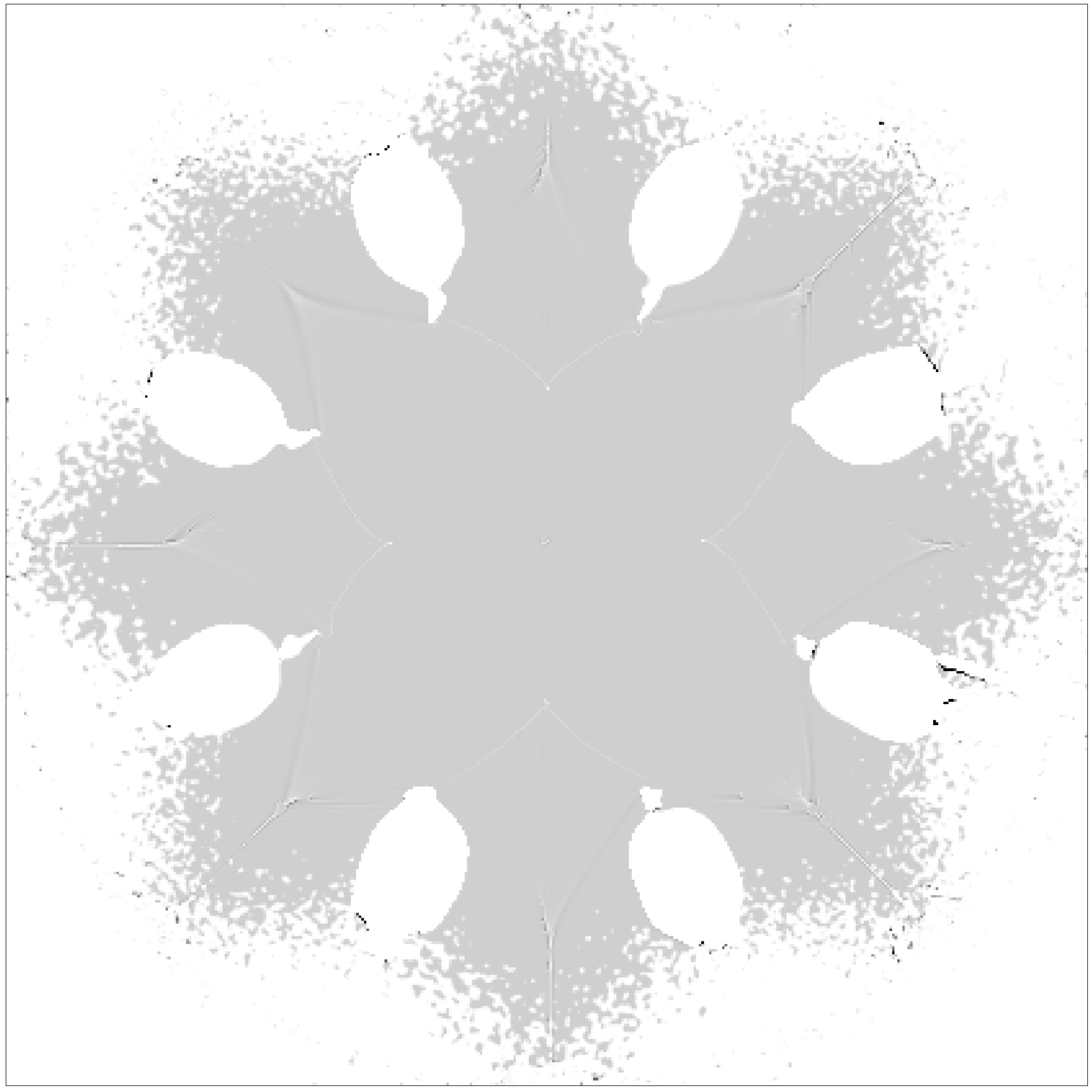}
	\end{subfigure}

    \centering
	\begin{subfigure}{0.45\textwidth}
	\centering
	\includegraphics[width=1\textwidth]{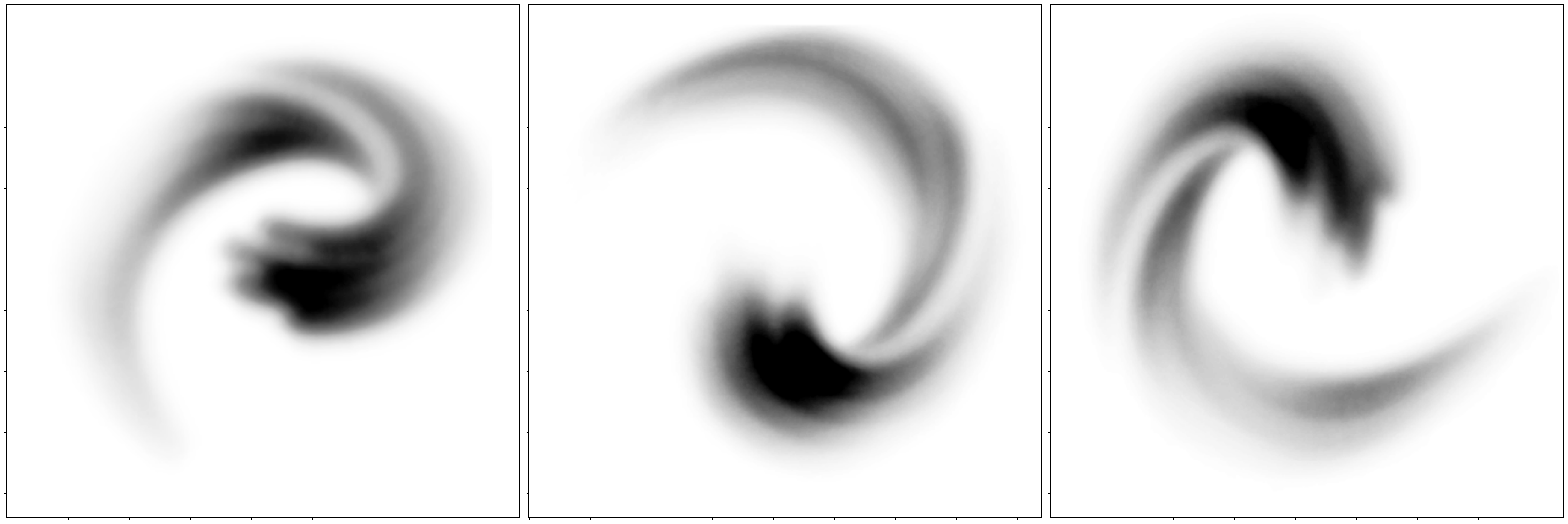}
        \caption{Sample measures.}
	\end{subfigure}
	\hspace{1cm}
	\begin{subfigure}{0.25\textwidth}
	\centering
	\includegraphics[width=0.7\textwidth]{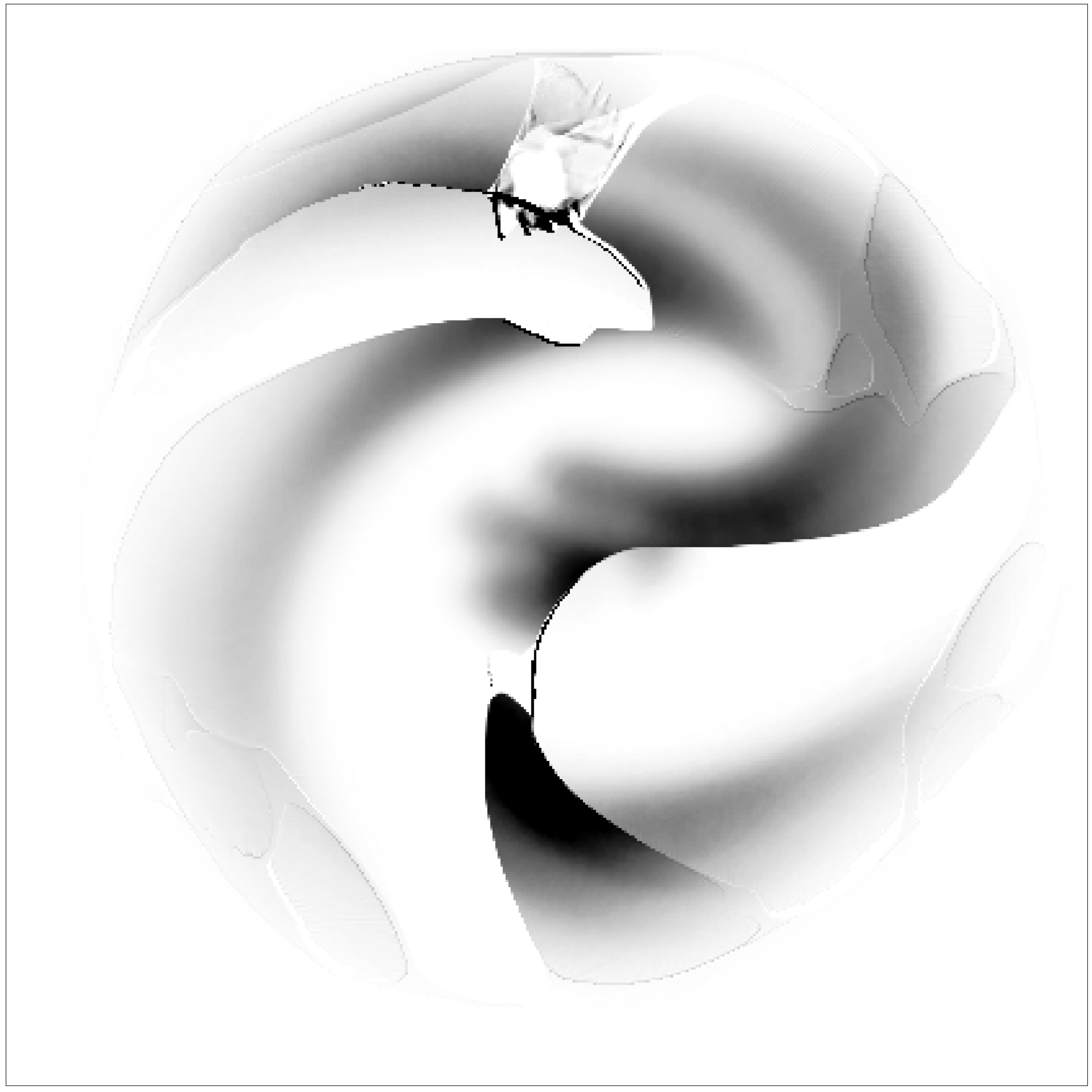}
        \caption{Wasserstein medians.}
	\end{subfigure}
	
	\caption{Some Wasserstein medians on a $420\times420$ grid computed with Douglas--Rachford up to $2000$ iterations, with a final residual of about $\sim 10^{-7}$, cf., Section \ref{sec:numerics}.}
	\label{fig:intro2}
    \end{figure}

\section{Definition, existence and basic examples}\label{sec:formulation_existence}

\paragraph{Setting.} Let $(\mathcal{X},d)$ be a \textit{proper} metric space, i.e.~a metric space in which closed balls are compact. This implies in particular that $(\cX, d)$ is \textit{Polish}, i.e.~separable and complete. Note that $(\mathcal{X},d)$ being proper is a natural assumption to define medians by minimization of weighted sums of distances; indeed this implies that for every integer $N\geq 1$, every $(x_1, \dots, x_N)\in \cX^N$ and every $\boldsymbol{\lambda}:=(\lambda_1,\dots,\lambda_N)$ in the simplex $\Delta_N$:
\[\Delta_N:= \Big\{ (\lambda_1,\dots,\lambda_N) \in \R_+^N \; : \; \sum_{i=1}^N \lambda_i =1\Big\},\] 
the set of medians of $(x_1, \dots, x_N)$ with weights $\boldsymbol{\lambda}$, defined by
\begin{equation}\label{defgroudmed}
\Ml(x_1, \dots, x_N):=\argmin_{x \in \mathcal{X}} \sum_{i=1}^N \lambda_i d(x_i, x)
\end{equation}
is a nonempty (and compact) subset of $\cX$.  

\begin{example}[\textbf{Medians on the real line}]
For $\cX=\R$ equipped with the distance associated with the absolute value, $N\ge 1$, $\bl=(\lambda_1, \dots, \lambda_N)\in \Delta_N$ and $\bx:=(x_1, \dots, x_N)\in \R^N$, $\Ml(\bx)$ is the set of minimizers of the convex, piecewise affine function  $x\mapsto  f(x):=\sum_{i=1}^N \lambda_i \vert x -x_i\vert$, this function being right and left differentiable at each point with corresponding one-sided derivative given by
\[f'(x^-)=\sum_{i \; : \; x_i <x} \lambda_i- \sum_{i \; : \; x_i \ge x} \lambda_i=2 \sum_{i \; : \; x_i <x} \lambda_i-1,  \;  f'(x^+) =2 \sum_{i \; : \; x_i \le x} \lambda_i-1.\]
We see that $x$ belongs to the median interval $\Ml(\bx)$ if and only if $f'(x^-) \le 0 \le f'(x^+)$, i.e.
\[\sum_{i \; : \; x_i <x} \lambda_i \le \frac{1}{2} \le \sum_{i \; : \;  x_i \le x} \lambda_i, \]
that is, $\Ml(\bx)= [\Ml^-(\bx), \Ml^+(\bx)]$, where $\Ml^-(\bx)$ and $\Ml^+(\bx)$ stand for the lower and upper medians respectively, which are given by:
\begin{equation}\label{defm+-}
\Ml^-(\bx):=\inf \Big\{ y \in \R \; : \;  \sum_{i \; : \;  x_i \le y} \lambda_i \ge \frac{1}{2} \Big\}, \;
\Ml^+(\bx):=\sup \Big\{ y \in \R \; : \;  \sum_{i \; : \;  x_i < y} \lambda_i \le \frac{1}{2} \Big\}.
\end{equation}
We shall use extensively properties of lower and upper medians when studying Wasserstein medians on $\R$ in Section \ref{sec:one_dim}. Obviously, since $f$ is affine in the neighbourhood of each point of $\R\setminus\{x_1, \dots, x_N\}$, both $\Ml^-(\bx)$ and $\Ml^+(\bx)$ belong to the sample $\{x_1, \dots, x_N\}$: 
\begin{equation}\label{attainmentpm}
I_\pm(\bx):=\{i =1, \dots, N \; : \;  \Ml^\pm(\bx) =x_i \} \neq \emptyset.
\end{equation}
Note also  both $\Ml^-$ and $\Ml^+$ are positively homogeneous and that setting $\be=(1, \dots, 1)$,
\[\Ml^\pm(\be)=1, \, \Ml^\pm(\alpha\bx)=\alpha  \Ml^\pm(\bx), \; \ \text{for all} \  \alpha \in \R_+.\]

Of course, in general, medians are highly non-unique. For instance if $N=2k$ is even, $\lambda_i=1/N$ and  $x_1<\dots< x_N$, the median interval is $[x_k, x_{k+1}]$. A mild condition guaranteeing uniqueness i.e.~$\Ml^-(\bx)= \Ml^+(\bx)$    for every $\bx$ is: 
\begin{equation}\label{eq:uniqueness_medians_1d}
	\text{There is no subset $I\subset \{1,\dots,N\}$ such that:} \ \sum_{i\in I} \lambda_{i}=\frac{1}{2}.
\end{equation}
Despite non-uniqueness, both selections $\Ml^+$ and $\Ml^-$ enjoy nice properties: obviously they are monotone in each of their arguments and invariant by translation, that is, for all $\bx \ge \by$ (i.e.~$\bx-\by\in \R_+^N$) we have $\Ml^\pm(\bx)\ge \Ml^\pm(\by)$, and for $\alpha \in \R$ it holds $\Ml^\pm(\bx + \alpha \be) = \Ml^\pm(\bx)+\alpha$. This implies in particular that for every $\bx$ and $\by$ one has
\begin{equation}\label{variationmpm}
 \Ml^\pm(\bx) + \min_{i=1, \dots, N} (y_i-x_i)  \le     \Ml^\pm(\by) \le  \Ml^\pm(\bx) + \max_{i=1, \dots, N} (y_i-x_i),
\end{equation}
so that  $\Ml^\pm$ are Lipschitz continuous. Inequality \eqref{variationmpm} will be very useful for studying horizontal and vertical Wasserstein median selections on the real line in Section \ref{sec:one_dim}. In fact, we will also need to use a slightly refined form of \eqref{variationmpm}, namely: for all $\bx$ there exists $\varepsilon>0$ such that for all $\by$ with $\Vert \bx-\by\Vert_{\infty} \leq \varepsilon$ it holds
\begin{equation}\label{variationmpmrefined}
  \Ml^\pm(\bx)  + \min_{i\in I_\pm(\bx)} (y_i-x_i) \le  \Ml^{\pm}(\by) \le  \Ml^\pm(\bx) + \max_{i\in I_\pm(\bx)} (y_i-x_i),
 \end{equation}
where we recall that $I_\pm(\bx)$ are given by \eqref{attainmentpm}. The proof of \eqref{variationmpmrefined} is postponed to the appendix. 
\end{example}\label{1dmed}

\begin{example}[\textbf{Torricelli--Fermat--Weber points}] Consider now $\cX=\R^d$ equipped with the  distance associated with the Euclidean norm $\vert  \cdot \vert$, $\bl\in \Delta_N$ and $(x_1, \dots, x_N)\in \cX^N$, by definition, $x\in \Ml(x_1, \dots, x_N)$ if and only if $x$ minimizes the convex function $\sum_{i=1}^N \lambda_i \vert \cdot -x_i\vert$ i.e.~satisfies the optimality condition
\[0 \in \sum_{i=1}^N    \lambda_i \partial \vert \cdot \vert (x-x_i),\]
where $\partial \vert \cdot \vert (x-x_i)$ is the subdifferential of the Euclidean norm at $x-x_i$:
\[\partial \vert \cdot \vert (x-x_i)=\{p\in \R^d \; : \; \vert p \vert \leq 1, \;  \langle p, x-x_i\rangle= \vert x-x_i\vert\}=\begin{cases}  B(0, 1)  & \mbox{ if $x=x_i$} \\ \frac{x-x_i}{\vert x-x_i\vert} & \mbox{ otherwise}\end{cases}\]
where $B(0,1)$ stands for the closed unit Euclidean ball. Therefore $x\in \Ml(x_1, \dots, x_N)$ if and only if there exist $p_1, \dots, p_N$ such that
\begin{equation}\label{euclidmedcns}
\vert p_i \vert \leq 1, \;  \langle p_i, x-x_i\rangle= \vert x-x_i\vert, \; i=1, \ldots, N, \mbox{ and }  \sum_{i=1}^N    \lambda_i p_i=0.
\end{equation}
Note that for $x\in \Ml(x_1, \dots, x_N)$ either $x=x_i$ for some $i$ or 
\[\sum_{i=1}^N \lambda_i \frac{x-x_i}{\vert x-x_i\vert}=0\]
so that in any case $x$ is a convex combination of  $x_1, \dots, x_N$, we thus have
\[ \Ml(x_1, \dots, x_N) \subset \mathrm{co} \{ x_1, \dots, x_N\}.\]

\end{example}\label{Torricelli}

Denote by $\cP(\cX)$ the set of Borel probability measures on $\cX$. Recall that on $\cP(\cX)$ the narrow topology is the coarsest topology making $\mu\in \cP(\cX) \mapsto \int_{\cX } f \d \mu$ continuous
for every continuous and bounded function $f$ on $\cX$ and that this topology is metrizable on $\cP(\cX)$ (so that there is no need to distinguish between narrow compactness and narrow sequential compactness). We denote by $\cP_1(\cX)$ the set of Borel probability measures with finite first moment i.e.~the set of $\mu \in \cP(\cX)$ for which for some (equivalently for all) $x_0 \in \cX$, $d(x_0, \cdot)\in L^1(\cX, \mu)$. We endow $\cP_1(\cX)$ with the \textit{Wasserstein} distance of order one:
\begin{equation*}
	W_1(\mu,\nu):=\inf_{\gamma \in \Pi(\mu,\nu)} \  \int_{\cX^2}d(x,y) \d\gamma(x,y),
\end{equation*} 
where $\Pi(\mu,\nu)$ is the set of transport plans between $\mu$ and $\nu$ i.e.~the set of Borel probability measures on $\cX^2$ with marginals $\mu$ and $\nu$. 
With this choice, the metric space $\left(\cP_1(\cX), W_1\right)$ is a Polish (but not necessarily proper) space.  Let us recall the Kantorovich--Rubinstein duality formula which expresses $W_1(\mu,\nu)$ as
\begin{equation}\label{kantorub}
	W_1(\mu,\nu):=\sup   \Big\{ \int_{\cX} u \d \mu- \int_{\cX} u \d \nu \; : \; \mbox{ $u:\cX \to \R$, $1$-Lipschitz} \Big\},
\end{equation} 
in particular, $W_1(\mu, \nu)$ is the dual Lipschitz semi-norm of $\mu-\nu$ and the linear interpolation $\mu_t:= (1-t)\mu+ t\nu$ for $t\in [0,1]$, is obviously a geodesic between $\mu$ and $\nu$ i.e.:
\begin{equation}\label{geodW1}
W_1(\mu_t,  \mu_s)=\vert t-s \vert W_1(\mu, \nu), \quad \ \text{for all} \  (t,s)\in [0,1]^2.
\end{equation}
Note that convergence in $\cP_1(\cX)$ for $W_1$ implies convergence for the narrow topology but is stronger unless $\cX$ is compact.
For proofs of these classical facts and more on Wasserstein distances, we refer to the textbooks \cite{ambrosio_gradient_2005, santambrogio_optimal_2015}.

\paragraph{Wasserstein medians.}  As mentioned in the introduction, on $\left(\cP_1(\cX), W_1\right)$, one can naturally define medians in the Fr\'echet sense as follows. Given $N\geq 1$,  $\bnu=(\nu_1, \dots, \nu_N) \in \cP_1(\cX)^N$ and $\boldsymbol{\lambda}:=(\lambda_1,\dots,\lambda_N)\in \Delta_N$, consider the weighted dispersion functional
\begin{equation}\label{disperfunc}
\mathcal{F}_{\bl, \bnu}(\mu):=  \sum_{i=1}^N \lambda_i W_1(\nu_i,\mu), \; \ \text{for all} \  \mu \in \cP_1(\cX)
\end{equation}
then Wasserstein medians are defined as minimizers of this dispersion functional:

\begin{defi}[\textbf{Wasserstein medians}] For $N\geq 1$, let $\bnu=(\nu_1,\dots,\nu_N) \in \cP_1(\cX)^N$ and $\boldsymbol{\lambda}=(\lambda_1,\dots,\lambda_N)\in \Delta_N$, defining $\mathcal{F}_{\bl, \bnu}(\mu)$  by \eqref{disperfunc}, 
we call Wasserstein median of $(\nu_1,\dots,\nu_N)$ with weights $\boldsymbol{\lambda}$  any solution of the following (convex) problem
	\begin{equation}\label{eq:Medianprob}
		v(\bl, \bnu):=\min_{\mu \in \cP_1(\cX)}   \; \mathcal{F}_{\bl, \bnu}(\mu)
	\end{equation}
	We denote by  $\Med_{\boldsymbol{\lambda}}(\nu_1,\dots,\nu_N)$ the set of all Wasserstein medians of $\bnu$ with weights $\boldsymbol{\lambda}$.
\end{defi}

The existence of a solution of \eqref{eq:Medianprob} follows from the direct method:  
\begin{lemma}[\textbf{Existence of Wasserstein medians}]\label{existence} Let $N\geq1$, $\bnu=(\nu_1,\dots,\nu_N) \in \mathcal{P}_1(\cX)^N$ and $\bl\in \Delta_N$, then there exists a minimizer of \eqref{eq:Medianprob} and the set $\Med_{\boldsymbol{\lambda}}(\nu_1,\dots,\nu_N)$ is a convex and narrowly compact  subset of $\mathcal{P}_1(\cX)$. 
\end{lemma}
\begin{proof}
The functional $\mathcal{F}_{\bl, \bnu}$ is l.s.c. for the narrow topology (this follows at once from \eqref{kantorub}) and by the triangle inequality for every $x_0\in \cX$ and every $\mu\in  \mathcal{P}_1(\cX)$ one has 
\[ W_1(\delta_{x_0}, \mu)=\int_{\cX} d(x_0, x) \d \mu(x) \leq  \mathcal{F}_{\bl, \bnu}(\mu) + \mathcal{F}_{\bl, \bnu}(\delta_{x_0}),\]
which implies that  the first moment is  uniformly bounded on  sublevel sets of $\mathcal{F}_{\bl, \bnu}$. Since $(\cX, d)$ is proper, this implies that sublevel sets of $\mathcal{F}_{\bl, \bnu}$  are tight hence narrowly relatively compact by Prokhorov's theorem. This implies nonemptiness and narrow compactness of $\Med_{\boldsymbol{\lambda}}(\nu_1,\dots,\nu_N)$, convexity follows from the convexity of $\cal{F}_{\bl, \bnu}$.
\end{proof}

Let us end this section with some simple explicit examples.

\begin{example}[\textbf{Medians of Dirac masses}] If $\nu_i=\delta_{x_i}$ is a Dirac mass for all $i =1,\dots, N$, then $\Med_{\boldsymbol{\lambda}}(\nu_1,\dots,\nu_N)$ is nothing but the set of probability measures supported on $\Ml(x_1, \dots, x_N)$. In particular, if $N=2$, $\cX=\R$,  $x_1 \leq x_2$ and $\bl=(1/2, 1/2)$, then $\Ml(x_1, x_2)=[x_1, x_2]$ so that $\Medl(\delta_{x_1}, \delta_{x_2})$ is the set of \emph{all} probability measures supported on $[x_1, x_2]$.
\end{example}

\begin{example}[\textbf{Threshold effect}]\label{rem:thresholdMed} Suppose that there is $J \subset\{1,\dots,N\}$ with $\sum_{j \in J} \lambda_j \geq \frac{1}{2}$ and  $\boldsymbol{\nu}:=(\nu_1,\dots,\nu_N)$ with $\nu_j = \rho$ for $j \in J$ for some $\rho \in \cP_1(\cX)$. Then a Wasserstein median of $\boldsymbol{\nu}$ is given by $\rho$ since for any $\tilde \rho \in \cP_1(\cX)$, denoting $J^c := \{1,\dots,N \} \setminus J$
\begin{align*}
    \sum_{i=1}^N \lambda_i W_1(\nu_i, \rho) & = \sum_{i \in J^c}  \lambda_i W_1(\nu_i, \rho) \\
    & \leq \sum_{i=1}^N \lambda_i W_1(\nu_i, \tilde \rho) + \sum_{i \in J^c}  \lambda_i W_1(\tilde \rho , \rho)- \sum_{i \in J} \lambda_i W_1(\nu_i, \tilde \rho) \\
    & = \sum_{i=1}^N \lambda_i W_1(\nu_i, \tilde \rho) + \underbrace{\left( \sum_{i \in J^c} \lambda_i -  \sum_{i \in J} \lambda_i \right)}_{\leq 0} W_1(\tilde \rho, \rho).
\end{align*}
Note that if $\sum_{j \in J} \lambda_j > \frac{1}{2}$, this also proves that the Wasserstein median is unique and equal to $\rho$. Note that this threshold effect is not specific to Wasserstein medians but holds for Fr\'echet medians in any metric space.
\end{example}

\begin{example}[\textbf{Medians of two measures}]
If $N=2$, $\nu_1 \neq \nu_2$, it follows from the previous example that when $\lambda_1 \in (1/2,1)$ (respectively $\lambda_1 \in (0,1/2)$) the median of $(\nu_1, \nu_2)$ with weights $(\lambda_1, 1-\lambda_1)$ is $\nu_1$ (respectively $\nu_2)$, when $\lambda_1=\lambda_2=1/2$, by the triangle inequality any interpolate $(1-t) \nu_1+t \nu_2$, $t\in [0,1]$ belongs to $\Med_{1/2, 1/2}(\nu_1,\nu_2)$. 

\end{example}

\begin{example}[\textbf{Medians of translated measures}]
Consider $\cX=\R^d$ endowed with the Euclidean distance, $\mu\in \cP_1(\cX)$, $(x_1, \dots, x_N)\in \cX^N$ and let $\tau_{x_i\#} \mu$ be the translation of $\mu$ by $x_i$ (i.e.~${\tau_{x_i}}_{\#} \mu(A)=\mu(A-x_i)$, for every Borel subset $A$ of $\R^d$). We claim that whenever $x \in \Ml(x_1, \dots, x_N)$ one has ${\tau_x}_{\#} \mu \in  \Medl({\tau_{x_1}}_{\#} \mu, \dots, {\tau_{x_N}}_{\#} \mu)$. To see this, let $(p_1, \dots, p_N)$ satisfy the optimality condition \eqref{euclidmedcns}, then we first have
\[\sum_{i=1}^N \lambda_i W_1( {\tau_{x_i}}_{\#}\mu, {\tau_x}_{\#} \mu) \leq \sum_{i=1}^N \lambda_i \vert x-x_i\vert=\sum_{i=1}^N \lambda_i \langle p_i,  x-x_i\rangle =-\sum_{i=1}^N \lambda_i \langle p_i , x_i\rangle.\]
Let now $\nu \in \cP_1(\cX)$, since $p_i \in B(0,1)$ the affine function $u_i(y):= \langle p_i, y+x-x_i\rangle$ is $1$-Lipschitz so that by the Kantorovich--Rubinstein formula
\begin{align*}
W_1({\tau_{x_i}}_{\#} \mu, \nu) &\geq \big\langle p_i,  \int_{\R^d} (y-x_i+x) \d \nu(y)\big\rangle- \big\langle p_i, \int_{\R^d} (y+x) \d\mu(y) \big\rangle \\
&= \big\langle  p_i, \int_{\R^d} (y-x_i)\d \nu(y)\big\rangle  - \big\langle  p_i , \int_{\R^d} y \d\mu(y) \big\rangle.
\end{align*}
Multiplying by $\lambda_i$, summing and using \eqref{euclidmedcns}, we obtain
\[ \sum_{i=1}^N \lambda_i W_1({\tau_{x_i}}_{\#} \mu, \nu)  \geq -\sum_{i=1}^N \lambda_i \langle p_i, x_i\rangle \geq \sum_{i=1}^N \lambda_i W_1( {\tau_{x_i}}_{\#}, {\tau_x}_{\#} \mu) \]
which shows that ${\tau_x}_{\#} \mu \in  \Medl({\tau_{x_1}}_{\#} \mu, \dots, {\tau_{x_N}}_{\#} \mu)$. 
\end{example}

\section{Stability and robustness}\label{sec:stability_robustness}

The stability with respect to perturbations of the sample measures is a crucial property for any location estimator especially when the underlying space $\cX$ is unbounded. This is why,  in this section, we will first investigate some stability properties of Wasserstein medians (improving the easy narrow stability to the stability in $W_1$-distance),  note that Theorem 5.5  in \cite{kroshnin_frechet_2018} establishes strong consistency results in a much more general framework. We will then show robustness to outliers by showing that the breakdown point of Wasserstein medians on an unbounded $\cX$ is at least $1/2$, the proof will be an easy adaptation of \cite{lopuhaa_breakdown_1991} revealing that the argument is in fact quite general and actually carries over to Fr\'echet medians on geodesic metric spaces. 

\subsection{Compactness in $W_1$ distance and stability with respect to data} 
 
 Let $N\geq 1$, $(\bl, \bnu)=(\lambda_1, \dots, \lambda_N, \nu_1, \dots, \nu_N)$ and $(\bl', \bnu')=(\lambda'_1, \dots, \lambda'_N, \nu'_1, \dots, \nu'_N)$ in $\Delta_N\times \cP_1(\cX)^N$, an obvious consequence of the triangle inequality is the fact that for any $\mu\in  \cP_1(\cX)$, one has
 \begin{equation}\label{ineggFF'}
 \mathcal{F}_{\bl, \bnu}(\mu) \leq \mathcal{F}_{\bl', \bnu'}(\mu) + \max_{i=1, \dots, N} W_1(\nu_i, \nu'_i)+  \sum_{i=1}^N \vert \lambda_i-\lambda'_i\vert  \max_{i=1,\dots, N} W_1(\nu'_i, \mu).
 \end{equation}
  This  pointwise inequality for the dispersions corresponding to $(\bl, \bnu)$ and $(\bl', \bnu')$, implies in particular that $ \mathcal{F}_{\bl', \bnu'}$ converges to $\ \mathcal{F}_{\bl, \bnu}$ uniformly on $W_1$ balls as
  \[  \sum_{i=1}^N \vert \lambda_i-\lambda'_i\vert  +  \max_{i=1, \dots, N} W_1(\nu_i, \nu'_i) \to 0.\]
  Let us also observe that for every $x_0\in \cX$, again by the triangle inequality, one also has the moment bound
  \begin{equation}\label{mb0}
\sup_{\mu \in \Medl(\bnu)} \int_{\cX} d(x_0, x) \d \mu(x) \leq  2 \max_{i=1, \dots, N} \int_{\cX} d(x_0, x_i) \d \nu_i(x).
  \end{equation}
Recalling the definition of $v(\bl, \bnu)$ from \eqref{eq:Medianprob}, \eqref{ineggFF'} and \eqref{mb0} show that $v$ is locally Lipschitz continuous for $W_1$. Combining the previous pointwise convergence with the narrow lower semicontinuity of $W_1$, \eqref{mb0} and the narrow compactness of measures with bounded first moments, we straightforwardly get:

\begin{lemma}\label{gammaeasy}
Let $N\geq 1$, $(\bl, \bnu)=(\lambda_1, \dots, \lambda_N, \nu_1, \dots, \nu_N) \in \Delta_N \times \cP_1(\cX)^N$ and  $(\bl^n, \bnu^n)_{n \in \mathbb{N}}= (\lambda_1^n, \dots, \lambda^n_N, \nu_1^n, \dots, \nu^n_N)_{n \in \mathbb{N}}$ be a sequence in  $\Delta_N \times \cP_1(\cX)^N$ such that 
\begin{equation}\label{datacv}
\sum_{i=1}^N \vert \lambda_i^n-\lambda_i\vert  +  \max_{i=1, \dots, N} W_1(\nu_i^n,  \nu_i) \to 0 \mbox{ as $n\to \infty$.}
\end{equation}
Then, $ \mathcal{F}_{\bl^n, \bnu^n}$ $\Gamma$-converges to  $\mathcal{F}_{\bl, \bnu}$ for the narrow topology, in particular if $\mu^n \in \Medln(\bnu^n)$ for all $n \in \mathbb{N}$, narrow cluster points of $(\mu^n)_{n \in \mathbb{N}}$ belong to $\Medl(\bnu)$.

\end{lemma}

One can improve the  previous (elementary and expected) result by stability in $W_1$ distance as follows (for more general results of this type, we refer the reader to \cite{kroshnin_frechet_2018}):

\begin{theorem}[\textbf{Stronger stability of Wasserstein medians}]\label{stability}
Let $N\geq 1$, $(\bl, \bnu) \in \Delta_N \times \cP_1(\cX)^N$, $(\bl^n, \bnu^n)_{n \in \mathbb{N}}$ be a sequence in  $\Delta_N \times \cP_1(\cX)^N$ such that \eqref{datacv} holds and let $\mu^n \in \Medln(\bnu^n)$ for all $n \in \mathbb{N}$, then $(\mu^n)_{n \in \mathbb{N}}$ admits a subsequence that converges for $W_1$ to some $\mu \in  \Medl(\bnu)$. In particular $\Medl(\bnu)$ is compact and  the set-valued map $(\bl, \bnu) \in \Delta_N \times \cP_1(\cX)^N \mapsto \Medl(\bnu) \subset \cP_1(\cX)$ has a closed graph for the $W_1$ distance.
\end{theorem}

\begin{proof}
We already know from Lemma \ref{gammaeasy} that $(\mu^n)_{n\in \mathbb{N}}$ has a (not relabeled) subsequence that converges narrowly to some $\mu$ which belongs to  $\Medl(\bnu)$. To improve narrow to $W_1$ convergence, it follows from Proposition 7.1.5 of \cite{ambrosio_gradient_2005}, that it is enough to show that (some subsequence of) $(\mu^n)_{n\in \mathbb{N}}$ has uniformly integrable moments. More precisely, fixing $x_0 \in \cX$ and for $R>0$ denoting by $B(x_0, R)$ the open ball of radius $R$, we have to show that (passing to a subsequence if necessary)
\begin{equation}\label{uimomentmu}
\lim_{R \to +\infty} \sup_n \int_{\cX\setminus B(x_0, R)} d(x_0, x) \d \mu^n(x) =0.
\end{equation}
Let $\gamma_i^n \in \Pi(\nu_i^n, \mu^n)$ such that $\int_{\cX^2} d(x_i, x) \d \gamma_i^n(x_i, x)=W_1(\nu_i^n, \mu^n)$, since both sequences $(\nu_i^n)_{n\in \mathbb{N}}$ and $(\mu^n)_{n\in \mathbb{N}}$ are tight so is $(\gamma_i^n)_{n\in \mathbb{N}}$, passing to subsequences if necessary, we may thus assume that $(\gamma_i^n)_{n\in \mathbb{N}}$ converges narrowly to some $\gamma_i\in \cP(\cX\times \cX)$. Of course $\gamma_i \in \Pi(\nu_i, \mu)$ and then
\[W_1(\nu_i, \mu) \leq \int_{\cX^2} d(x_i, x) \d \gamma_i(x_i, x) \leq \liminf_n  \int_{\cX^2} d(x_i, x) \d \gamma_i^n(x_i, x)= \liminf_n W_1(\nu_i^n, \mu^n).  \]
We deduce from Lemma \ref{gammaeasy} and the fact that $(\nu_i^n)_{n\in \mathbb{N}}$ and $(\mu^n)_{n\in \mathbb{N}}$ have uniformly bounded moments
\[\sum_{i=1}^N \lambda_i  W_1(\nu_i, \mu)=\lim_n \sum_{i=1}^N \lambda_i^n  W_1(\nu_i^n, \mu^n)  \geq \sum_{i=1}^N \liminf_n  \lambda_i^n  W_1(\nu_i^n, \mu^n)=\sum_{i=1}^N \lambda_i \liminf_n   W_1(\nu_i^n, \mu^n). \]
Hence, for every $i$ for which $\lambda_i>0$, one has $W_1(\nu_i, \mu)=\liminf_n   W_1(\nu_i^n, \mu^n)$. Assuming without loss of generality that $\lambda_1>0$, we thus have
\[W_1(\nu_1, \mu)=\int_{\cX^2} d(x_1, x) \d \gamma_1(x_1, x) = \liminf_n  \int_{\cX^2} d(x_1, x) \d \gamma_1^n(x_1, x).\]
Passing to a subsequence if necessary, we may assume that the liminf of the right hand side above is a true limit and then, using Lemma 5.1.7 of \cite{ambrosio_gradient_2005}, we deduce that 
\begin{equation}\label{uimomentgamma}
\lim_{R \to +\infty} \sup_n \int_{\{(x_1, x) \in \cX^2 \; : \; d(x_1, x) \geq R\}} d(x_1, x) \d \gamma_1^n(x_1,x) =0.
\end{equation}
Note also that since $(\nu_1^n)_{n\in \mathbb{N}}$ converges in $W_1$ we also have 
\begin{equation}\label{uimomentnu1n}
\lim_{R \to +\infty} \sup_n \int_{\cX\setminus B(x_0, R)} d(x_0, x_1) \d \nu_1^n(x_1) =0.
\end{equation}
Defining for $R>0$ and $t\geq 0$, 
\[\Phi_R(t):=\begin{cases} t & \mbox{ if $t \ge R$,}\\ 0 & \mbox{ else,}\end{cases}\]
note that $\Phi_R$ is non decreasing and 
\[\Phi_R(t+s) \leq 2 \Big(\Phi_{\frac{R}{2}} (t)+ \Phi_{\frac{R}{2}} (s) \Big),\]
so by the triangle inequality for every $(x, x_1) \in \cX^2$, we have
\[\Phi_R( d(x_0, x)) \leq 2 \Big(\Phi_{\frac{R}{2}} (d(x_0, x_1))+ \Phi_{\frac{R}{2}} (d(x_1, x)) \Big). \]
Integrating with respect to $\gamma_1^n$ which has marginals $\nu_1^n$ and $\mu^n$ yields 
\[ \begin{split} \int_{\cX\setminus B(x_0, R)} d(x_0, x) \d \mu^n(x)&= \int_{\cX}  \Phi_R( d(x_0, x)) \d \mu^n(x)= \int_{\cX}  \Phi_R( d(x_0, x)) \d \gamma_1^n(x_1, x) \\
 &\leq 2  \int_{\cX}  \Phi_{\frac{R}{2}}( d(x_0, x_1)) \d \nu_1^n(x_1) + 2 \int_{\cX^2}  \Phi_{\frac{R}{2}}( d(x_1, x)) \d \gamma_1^n(x_1,x)\\
 &= 2  \int_{\cX\setminus B(x_0, \frac{R}{2})} d(x_0, x_1) \d \nu_1^n(x_1)\\
& \quad +  2 \int_{\{(x_1, x) \in \cX^2 \; : \; d(x_1, x) \geq \frac{R}{2}\}} d(x_1, x) \d \gamma_1^n(x_1,x).
 \end{split}\]
Then, \eqref{uimomentmu} readily follows from \eqref{uimomentgamma} and \eqref{uimomentnu1n}. 

\end{proof}

\subsection{Robustness of Wasserstein medians} 

In statistics, a popular robustness index is the so-called \textit{break-down point}. Roughly speaking, it is the largest fraction of the input data which could be corrupted (i.e.~changed arbitrarily) without moving the estimation too far from the original estimation for the non-corrupted data. It is well known that the break-down point of geometric medians with uniform weights is approximately $\tfrac{1}{2}$, see, e.g.~Theorem 2.1 and 2.2 in \cite{lopuhaa_breakdown_1991}, so that even corrupting about half  of the data, we can stay rather confident on the output. In this section, we prove a similar result for Wasserstein medians. To do so, we first recall some basic facts about break-down points, starting with a definition of the break-down point adapted to the case of a non-unique estimator.
\begin{defi}[\textbf{Break-down point}]
Let $(\cX,d)$ be a metric space. Let $N\geq 2$ and $\boldsymbol{\lambda}=(\lambda_1,\dots,\lambda_N) \in \Delta_N$. For a set-valued map $t_{\boldsymbol{\lambda}}:   \cX^N \rightarrow 2^\cX$ with nonempty values, we define its \textit{break-down point} associated to the weights $\boldsymbol{\lambda}$ at $\bx=(x_1,\dots,x_N)\in \mathcal{X}^N$ by
\begin{equation*}
    b(t_{\boldsymbol{\lambda}}(\bx)) := \min \bigg\{ \sum_{i \in I}\lambda_i: \ I \subset \{1,\dots,N\}, \sup_{\substack{\by^I \in \cX^N \\ \by^I_j = \bx_j \ \forall  j \notin I}}  \left\{ d(y,x)\; : \; y \in t_{\boldsymbol{\lambda}}(\by^I), \ x \in t_{\boldsymbol{\lambda}}(\bx) \right\}=+\infty \bigg\}.
\end{equation*}

\end{defi}

We now state the main theorem for Wasserstein medians, where the reference metric space is $\mathcal{P}_1(\cX)$ equipped with the $W_1$ distance. The proof is a slight generalization of Theorem 2.2.~in \cite{lopuhaa_breakdown_1991}.

\begin{theorem}[\textbf{Break-down point of Wasserstein medians}]\label{thm:breakdownpointofwassersteinmedians}
Suppose the  metric space $ (\cX, d)$ is proper and unbounded.   Let $N\geq 2,$ $\boldsymbol{\nu}:=(\nu_1,\dots,\nu_N)\in \mathcal{P}_1(\cX)^N$ and $\boldsymbol{\lambda} := (\lambda_1,\dots,\lambda_N) \in \Delta_N$.
Then the break-down point of $\Med_{\boldsymbol{\lambda}}(\boldsymbol{\nu})$ is given by
	\begin{equation}\label{eq:break_down_point}
	b(\Med_{\boldsymbol{\lambda}}(\boldsymbol{\nu}))= \min \bigg\{ \sum_{j \in J}\lambda_j\; : \; J \subset \{1,\dots,N \}, \ \sum_{j \in J}\lambda_j \geq \frac{1}{2}\bigg\}.
	\end{equation}
\end{theorem}
\begin{proof}
For future reference, let us denote by $B$ the right hand-side of \eqref{eq:break_down_point}. Let us take $\nu \in \Med_{\boldsymbol{\lambda}}(\nu_1,\dots,\nu_N)$ and $I \subset \{1,\dots,N\}$ such that $\sum_{i \in I} \lambda_i < \frac{1}{2}$. Denote by $\boldsymbol{\mu}:=(\mu_1,\dots,\mu_N) \in \cP_1(\cX)^N$ a corrupted collection of $\boldsymbol{\nu}:=(\nu_1,\dots,\nu_N)$, i.e.~such that $\mu_j = \nu_j$ for all $j \notin I$. Let 
\begin{equation}\label{Cetdelta}
C:=\max\limits_{\rho \in \Med_{\boldsymbol{\lambda}}(\nu_1,\dots,\nu_N)}\max\limits_{1 \leq i \leq N} W_1(\rho,\nu_i), \; 
\delta := \max \bigg\{ \sum_{j \in J}\lambda_j\; : \; J \subset \{1,\dots,N \}, \ \sum_{j \in J}\lambda_j < \frac{1}{2}\bigg\}.
\end{equation}
Let $\mu \in \Med_{\boldsymbol{\lambda}}(\boldsymbol{\mu})$,  let us  first prove by contradiction that
\[W_1(\nu,\mu) \leq  \frac{2C\delta}{1-2\delta} + 2C.\]

In order to do so,  let $\boldsymbol{\mathcal{B}}=B_{2C}(\nu)$ be the ball with center $\nu$ and radius $2C$ with respect to the $W_1$ distance. Further, let 
$$\xi := \mathsf{Dist}(\mu,\boldsymbol{\mathcal{B}}) := \inf_{\rho \in \boldsymbol{\mathcal{B}}} W_1(\mu,\rho).$$
Then by the triangle inequality $W_1(\mu,\nu) \leq \xi + 2C$, so that for all $j=1,\dots,N$
\begin{equation}\label{eq:est-corr}
    W_1(\mu_j,\mu) \geq W_1(\mu_j,\nu) - W_1(\nu,\mu) \geq W_1(\mu_j,\nu) - (\xi + 2C).
\end{equation}
Now suppose by contradiction that $\xi > 2C\delta/(1-2\delta)$, which in particular implies that $W_1(\mu, \nu) > 2C$. 
Using the fact  that in ($\cP_1(\cX), W_1)$, line segments are  geodesics (recall \eqref{geodW1}),  defining for $j=1,\dots,N$ the interpolation $\nu_j^t := (1-t) \nu_j + t \mu$, $t \in [0,1]$, we have 
$$W_1(\nu_j,\mu) = W_1(\nu_j, \nu_j^t) + W_1(\nu_j^t,\mu) \text{ for } t \in [0,1].$$
Since $W_1(\nu,\nu_j^0)= W_1(\nu,\nu_j) \leq C$ and $W_1(\nu,\nu_j^1)=W_1(\nu, \mu)  > 2 C$, there exists $\bar{t}\in [0,1]$ such that $W_1(\nu,\nu_j^{\bar{t}}) = 2C$. In particular $W_1(\nu_j^{\bar{t}},\mu)\geq \xi$ 
and $W_1(\nu_j,\nu_j^{\bar{t}}) \geq  W_1(\nu,\nu_j^{\bar{t}})-W_1(\nu_j, \nu)= 2C - W_1(\nu_j, \nu) \geq W_1(\nu_j, \nu)$ so that
\begin{equation}\label{eq:est-noncorr}
    W_1(\nu_j,\mu) = W_1(\nu_j,\nu_j^{\bar{t}}) + W_1(\nu_j^{\bar{t}},\mu)    \geq W_1(\nu_j,\nu) + \xi. 
\end{equation}
Putting together \eqref{eq:est-corr} and \eqref{eq:est-noncorr} yields
\begin{align*}
\sum_{j=1}^N \lambda_j W_1(\mu_j,\mu) & \geq  \sum_{j \in I} \lambda_j(W_1(\mu_j,\nu) - (\xi + 2C)) + \sum_{j \notin I} \lambda_j(W_1(\mu_j,\nu)+ \xi) \\
& =  \sum_{j =1}^N \lambda_j W_1(\mu_j,\nu) + \xi (\sum_{j \notin I} \lambda_j - \sum_{j \in I} \lambda_j) - 2C \sum_{j \in I} \lambda_j \\
& \geq  \sum_{j =1}^N \lambda_j W_1(\mu_j,\nu) + \xi (1 - 2 \delta) - 2C\delta \\
& >  \sum_{j =1}^N \lambda_j W_1(\mu_j,\nu),
\end{align*}
which contradicts $\mu$ being a Wasserstein median for the corrupted collection $\boldsymbol{\mu}$. We thus have
\begin{equation}\label{inge6788}
\xi \leq \frac{2C\delta}{1-\delta} \mbox{ and }  W_1(\nu,\mu) \leq \xi+ 2 C \leq  \frac{2C\delta}{1-2\delta} + 2C,
\end{equation}
and  $b(\Med_{\boldsymbol{\lambda}}(\boldsymbol{\nu})) >\delta$, yielding $b(\Med_{\boldsymbol{\lambda}}(\boldsymbol{\nu}))\geq B$. To obtain the exact value of the breakdown point, take now $J \subset\{1,\dots,N\}$ with $\sum_{j \in J} \lambda_j \geq \frac{1}{2}$
and consider a sequence $(x_n)_{n \in \N}$ in $\cX$ such that $d(x_n,x_0) \rightarrow +\infty$ as $n \rightarrow \infty$. Then, by using the sequence of corrupted collections defined by $\boldsymbol{\mu}^n:=(\mu^n_1,\dots,\mu^n_N)$ with $\mu^n_j = \delta_{x_n}$ for $j \in J$ and $\mu^n_k = \nu_k$ for $k \notin J$ we have $\delta_{x_n} \in \Med_{\boldsymbol{\lambda}}(\boldsymbol{\mu}^n)$ as we have observed in Example \ref{rem:thresholdMed} and
\begin{equation*}
    W_1(\delta_{x_n},\nu) \geq d(x_n,x_0) - W_1(\delta_{x_0},\nu) \rightarrow + \infty \quad \text{as } n \rightarrow \infty,
\end{equation*}
implying $b(\Med_{\boldsymbol{\lambda}}(\boldsymbol{\nu}))\leq B$ and concluding the proof.
\end{proof}
Note that in the case of uniform weights, i.e.~with $\boldsymbol{\lambda}:=(1/N, \dots, 1/N)$, \eqref{eq:break_down_point} turns into the classical estimate $b(\Med_{\boldsymbol{\lambda}}(\boldsymbol{\nu})) = \left\lfloor{\tfrac{N+1}{2}}\right\rfloor/N$. Let us finally emphasize that the proof of Theorem \ref{thm:breakdownpointofwassersteinmedians} actually works for Fr\'echet medians on any geodesic metric space.

\section{One dimensional Wasserstein medians}\label{sec:one_dim}
	
	In this section, we study the case of Wasserstein medians on $\cX=\mathbb{R}$ with distance $d$ induced by the absolute value. Since the Wasserstein distance of order 1 is equal to the $L^1$ distance between cumulative or quantile distribution functions, the problem becomes more explicit. This will in particular enable us to  find  different explicit constructions of Wasserstein medians. In this section for all $\nu \in \cP_1(\mathbb{R})$ we denote by $F_\nu$ its associated cumulative distribution function (cdf), which is defined by $F_\nu(x)=\nu((-\infty,x])$ for all $x \in \mathbb{R}$. We also denote by $Q_\nu: [0,1]\to \bar{\R}$ its \textit{pseudo-inverse} or quantile distribution function (qdf), which is defined by
	\begin{equation*}
	Q_\nu(t):=\inf \{x \in \mathbb{R} \; : \;  \ F_\nu(x)\geq t \}.
	\end{equation*}
	Denoting by $\LL$ the Lebesgue measure on $[0,1]$, it is well-known that one recovers $\nu$ from its qdf $Q_\nu$ through ${Q_{\nu}}_{\#} \LL=\nu$, that is $Q_\nu$ is the \emph{monotone transport} between $\LL$ and $\nu$. We first recall that in one dimension, both maps $\nu\in  \mathcal{P}_1(\R) \mapsto F_\nu$ and  $\nu\in  \mathcal{P}_1(\R) \mapsto Q_\nu$ map isometrically, for the $L^1$ distance, the Wasserstein space  $(\mathcal{P}_1(\R), W_1)$ to the set of cdf's of probabilities in $ \mathcal{P}_1(\R)$ (i.e.~the set of nondecreasing, right-continuous functions $F:\R \to [0,1]$ such that $(1-F)\in L^1((0, +\infty))$, $F\in L^1((-\infty, 0))$, $F(+\infty)=1$ and $F(-\infty)=0$) and the set of qdf's (i.e.~the set of $L^1((0,1), \LL)$  non-decreasing left-continuous functions) respectively. More precisely, for $(\mu,\nu)\in \mathcal{P}_1(\mathbb{R})^2$ we have the following convenient expressions  for the $1$-Wasserstein distance between $\mu$ and $\nu$ (see Theorem 2.9 in \cite{santambrogio_optimal_2015}):
		\begin{align}
		\label{w1Q}
		W_1(\mu,\nu) &=\int_0^1 | Q_\nu(t)-Q_\mu(t) | \d t=\|Q_\nu-Q_\mu\|_{L^1([0,1])}\\
		 \label{w1F}
		&=\int_{\mathbb{R}} |F_\mu(x)-F_\nu(x)| \d x=\|F_\nu-F_\mu\|_{L^1(\R)}.
		\end{align}
	This enables us to reformulate  the Wasserstein median problem  as
	\begin{align}
	\min\eqref{eq:Medianprob}&=\min_{\nu \in \mathcal{P}_1(\mathbb{R})} \  \int_{\mathbb{R}} \sum_{i=1}^N \lambda_i |F_\nu(t)-F_{\nu_i}(t)| \d t \label{eq:one_dim_v}\\
	&=\min_{\nu \in \mathcal{P}_1(\mathbb{R})} \  \int_0^1 \sum_{i=1}^{N}\lambda_i |Q_\nu(t)-Q_{\nu_i}(t)| \d t,\label{eq:one_dim_h}
	\end{align}
	which will be referred as \textit{vertical} \eqref{eq:one_dim_v} and \textit{horizontal} \eqref{eq:one_dim_h} \textit{formulations}. The terminology will become clear in the sequel. Note that, in this way, the problem is equivalent to performing a proper selection of a  weighted median of all cumulative or quantile distribution functions, the lower and upper median maps $\Ml^+$ and $\Ml^-$ defined in \eqref{defm+-}  in Example \ref{1dmed} and their regularity properties will be particularly useful in this setting. 
		
	\begin{prop}\label{minimed1dqf}
	Let $\boldsymbol{\lambda} \in \Delta_N$, $\bnu:=(\nu_1, \dots, \nu_N) \in  \mathcal{P}_1(\R)^N$ and $\nu \in  \mathcal{P}_1(\R)$, then the following statements are equivalent
	\begin{enumerate}
	\item $\nu \in \Medl(\bnu)$,\label{item:minimed1dqf_1}
	\item $F_\nu(x) \in \Ml(F_{\nu_1}(x), \dots, F_{\nu_N}(x))$ for all $x\in \R$,\label{item:minimed1dqf_2}
	\item $Q_\nu(t) \in \Ml(Q_{\nu_1}(t), \dots, Q_{\nu_N}(t))$ for all $t\in [0,1]$.\label{item:minimed1dqf_3}
	\end{enumerate}
    In particular, if \eqref{eq:uniqueness_medians_1d} holds, there exists a unique Wasserstein median.
	\end{prop}%
	\begin{proof}
	The fact that \ref{item:minimed1dqf_2} implies \ref{item:minimed1dqf_1} obviously follows from the definition of $\Ml$ and the expression in \eqref{w1F} for the Wasserstein distance. Assume now that $\nu \in \Medl(\bnu)$, then we should have for a.e.~$x\in \R$, 
	\begin{equation}\label{inegg12121}
	\Ml^-(F_{\nu_1}(x), \dots, F_{\nu_N}(x)) \le  F_\nu(x) \le \Ml^+(F_{\nu_1}(x), \dots, F_{\nu_N}(x)).
	\end{equation}
	Hence for $x\in \R$, there exists a sequence $\varepsilon_n >0$, $\varepsilon_n \to 0$ such that the previous inequality holds at $x+\varepsilon_n$,
	by the right continuity of $F_\nu$, $(F_{\nu_1}(\cdot), \dots, F_{\nu_N}(\cdot))$ at $x$ and the continuity of $\Ml^\pm$, we easily get that \eqref{inegg12121} actually holds at $x$ hence everywhere, proving the equivalence between \ref{item:minimed1dqf_1} and \ref{item:minimed1dqf_2}. The equivalence between \ref{item:minimed1dqf_1} and \ref{item:minimed1dqf_3} follows the same lines (using left-continuity of qdf's). 
	\end{proof}

This suggests to define
\[F^-(x):=\Ml^-(F_{\nu_1}(x), \dots, F_{\nu_N}(x)), \; F^+(x):=\Ml^+(F_{\nu_1}(x), \dots, F_{\nu_N}(x)),  \quad \text{for all} \ x\in \R,\]
as well as for $\theta \in [0,1]$, 
\begin{equation}\label{interpolF}
F_\theta(x):=(1-\theta) F^-(x)+ \theta F^+(x), \quad \text{for all} \ x\in \R.
\end{equation}
Thanks to  the properties of $\Ml^+$ and $\Ml^-$ we saw in Example \ref{1dmed} and the fact that the $F_{\nu_i}$'s are the cdf's of probability measures with finite first moments, $F^+$ and $F^-$ are also the cdf's of measures with  finite first moments and then so is $F_\theta$. Thanks to Proposition \ref{minimed1dqf}, $F_\theta$ is the cdf of $\nu^\theta$ which belongs to the set of Wasserstein medians $\Medl(\bnu)$, we call these measures $\nu^\theta$ vertical median selections:

\begin{defi}[\textbf{Vertical median selections}]
For every $\theta \in [0,1]$, the measure $\nu^\theta$ whose cdf is $F_\theta$ given by \eqref{interpolF} is called the vertical median selection of $\bnu$ with weights $\bl$ and interpolation parameter $\theta$ and simply denoted $\VMedl(\theta, \bnu)$.
\end{defi}

Let us also define
\[Q^-(t):=\Ml^-(Q_{\nu_1}(t), \dots, Q_{\nu_N}(t)), \; Q^+(t):=\Ml^+(Q_{\nu_1}(x), \dots, Q_{\nu_N}(t)), \; \ \text{for all} \  t\in (0,1),\]
as well as for $\theta \in [0,1]$, 
\begin{equation}\label{interpolQ}
Q_\theta(t):=(1-\theta) Q^-(t)+ \theta Q^+(t), \; \ \text{for all} \  t\in (0,1).
\end{equation}
It is easy to see that $Q_\theta$ is nondecreasing, left-continuous and in $L^1((0,1), \LL)$; it is therefore the qdf of a median $\mu^\theta \in \Medl(\bnu)$ which we call an horizontal median selection:

\begin{defi}[\textbf{Horizontal median selections}]
For every $\theta \in [0,1]$, the measure $\mu^\theta$ whose qdf is $Q_\theta$ given by \eqref{interpolF} is called the horizontal median selection of $\bnu$ with weights $\bl$ and interpolation parameter $\theta$ and simply denoted $\HMedl(\theta, \bnu)$.
\end{defi}

A first nice feature of both vertical and horizontal median selections is that it selects medians in a   Lipschitz  continuous way with respect to the sample measures:

\begin{lemma}\label{vhselarelip}
Let $\bl\in \Delta_N$,  $\bnu=(\nu_1, \ldots, \nu_N) \in  \mathcal{P}_1(\R)^N$, $\btnu=(\tnu_1, \ldots, \tnu_N)\in  \mathcal{P}_1(\R)^N$, $\theta\in [0,1]$ then 
\begin{equation}\label{vsellip}
W_1( \VMedl(\theta, \bnu) , \VMedl(\theta, \btnu)) \le \sum_{i=1}^N W_1(\nu_i, \tnu_i),
\end{equation}
and 
\begin{equation}\label{hsellip}
 W_1( \HMedl(\theta, \bnu) , \HMedl(\theta, \btnu))\le  \sum_{i=1}^N W_1(\nu_i, \tnu_i)  .
 \end{equation}
\end{lemma}

\begin{proof}
From \eqref{variationmpm}, we have for every $x\in \R$:
\[\begin{split} \vert \Ml^\pm (F_{\nu_1}(x), \dots, F_{\nu_N}(x))- \Ml^\pm (F_{\tnu_1}(x), \dots, F_{\tnu_N}(x)) \vert \\ \leq \max_{i=1,\ldots, N} \vert F_{\nu_i}(x)- F_{\tnu_i}(x)\vert 
 \leq \sum_{i=1}^{N} \vert F_{\nu_i}(x)- F_{\tnu_i}(x)\vert.
 \end{split}\]
Integrating and recalling the cdf expression \eqref{w1F} for the Wasserstein distance, we readily get \eqref{vsellip} for $\theta=0$ and $\theta=1$, the general case $\theta\in [0,1]$ follows by the triangle inequality. The proof of  \eqref{hsellip} is similar using the expression of $W_1$  in terms of quantiles as in \eqref{w1Q}.%
\end{proof}

One may wonder whether some medians inherit properties of the sample measures  and in particular whether samples consisting of probabilities with an $L^p$ density with respect to the Lebesgue measure have medians with the same property. As we will shortly see, vertical and horizontal medians will enable us to answer these questions by the positive.

\begin{lemma}
Let $\bl\in \Delta_N$,  $\bnu=(\nu_1, \ldots, \nu_N)\in  \mathcal{P}_1(\R)^N$, $\theta\in [0,1]$, and $\nu^\theta:=\VMedl(\theta, \bnu)$,  $\mu^\theta:=\HMedl(\theta, \bnu)$, then 
\begin{enumerate}
\item if $\nu_1, \ldots, \nu_N$ are atomless, then so are $\nu^\theta$ and $\mu^\theta$,
\item if $\nu_1, \ldots, \nu_N$ have connected supports, then so does $\mu^\theta$.
\end{enumerate}
\end{lemma}%
\begin{proof}
Recall that for a probability measure, being  atomless is equivalent to having a continuous cdf as well as to having a strictly increasing qdf, see, e.g., Proposition 1 in \cite{Embrechts2013}. Let us denote by $F_\theta$ the cdf (see \eqref{interpolF}) of $\nu^\theta$ and by $Q_\theta$ the qdf of $\mu^\theta$ (see \eqref{interpolQ}). If  $\nu_1, \ldots, \nu_N$ are atomless  then $F_\theta$ is continuous by continuity of $F_{\nu_i}$ so that  $\nu^\theta$ is atomless. On the other hand, \eqref{variationmpm} entails
\[Q_\theta(t)-Q_\theta(s) \geq \min_{1 \leq i \leq N} (Q_{\nu_i}(t)-Q_{\nu_i}(s)), \; \text{for all} \ (t,s)\in (0,1)^2,\]
so that $Q_\theta$ is strictly increasing whenever each $Q_{\nu_i}$ is.  Let us assume now that $\nu_1, \ldots, \nu_N$ have connected supports then each $Q_{\nu_i}$ is continuous and so is $Q_\theta$. Thus, $\mu^\theta$ has a connected support (again by Proposition 1 in \cite{Embrechts2013}). 
\end{proof}

	\begin{figure}[!t]
    \centering
	\begin{subfigure}{0.35\textwidth}
	\centering
	\includegraphics[width=1\textwidth]{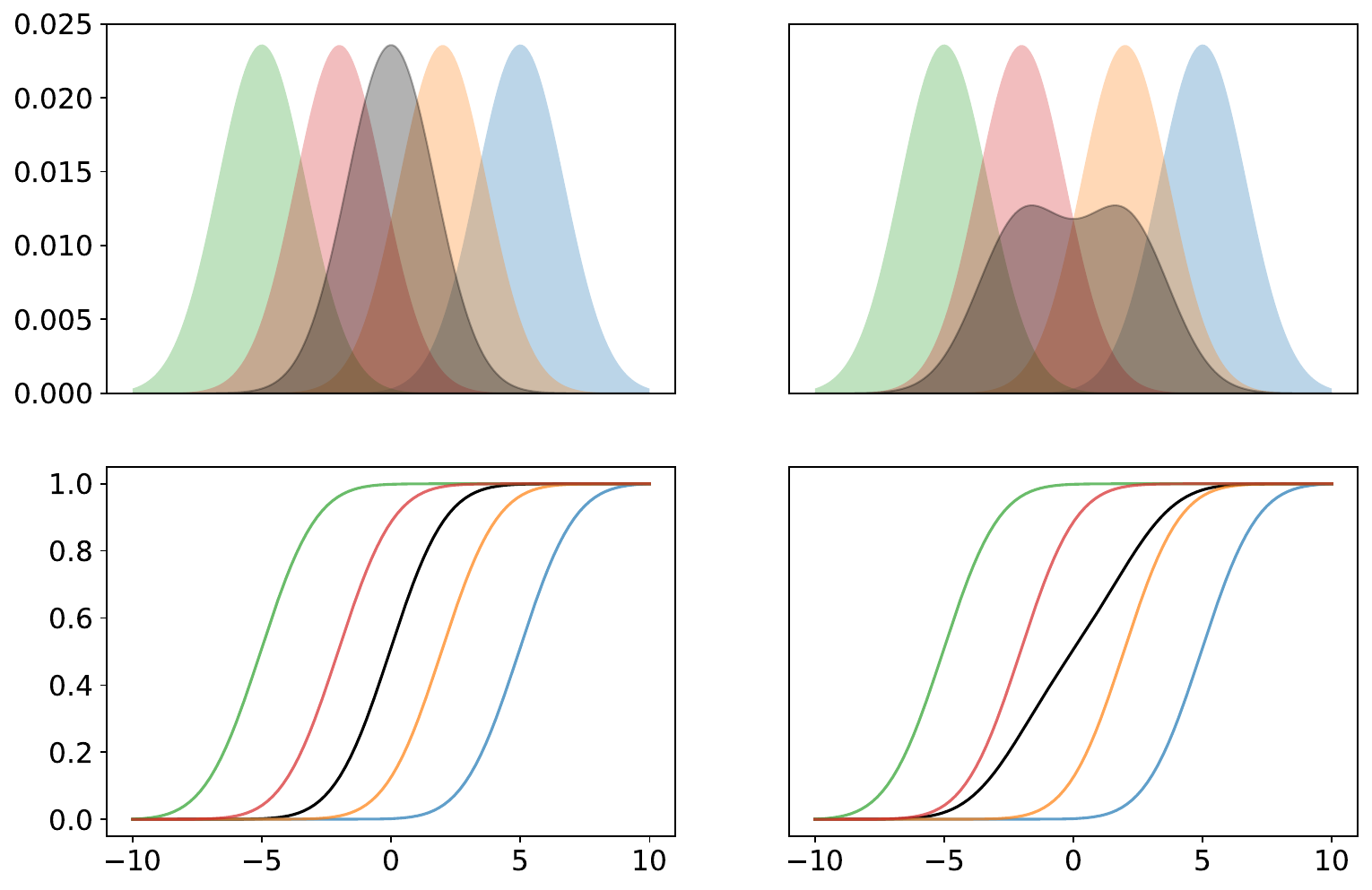}
	\caption{Horizontal selection compared to the vertical selection with $\theta=1/2$.}
	\end{subfigure}
	\hspace{0.5cm}
	\begin{subfigure}{0.46\textwidth}
		\centering
		\includegraphics[width=1\linewidth]{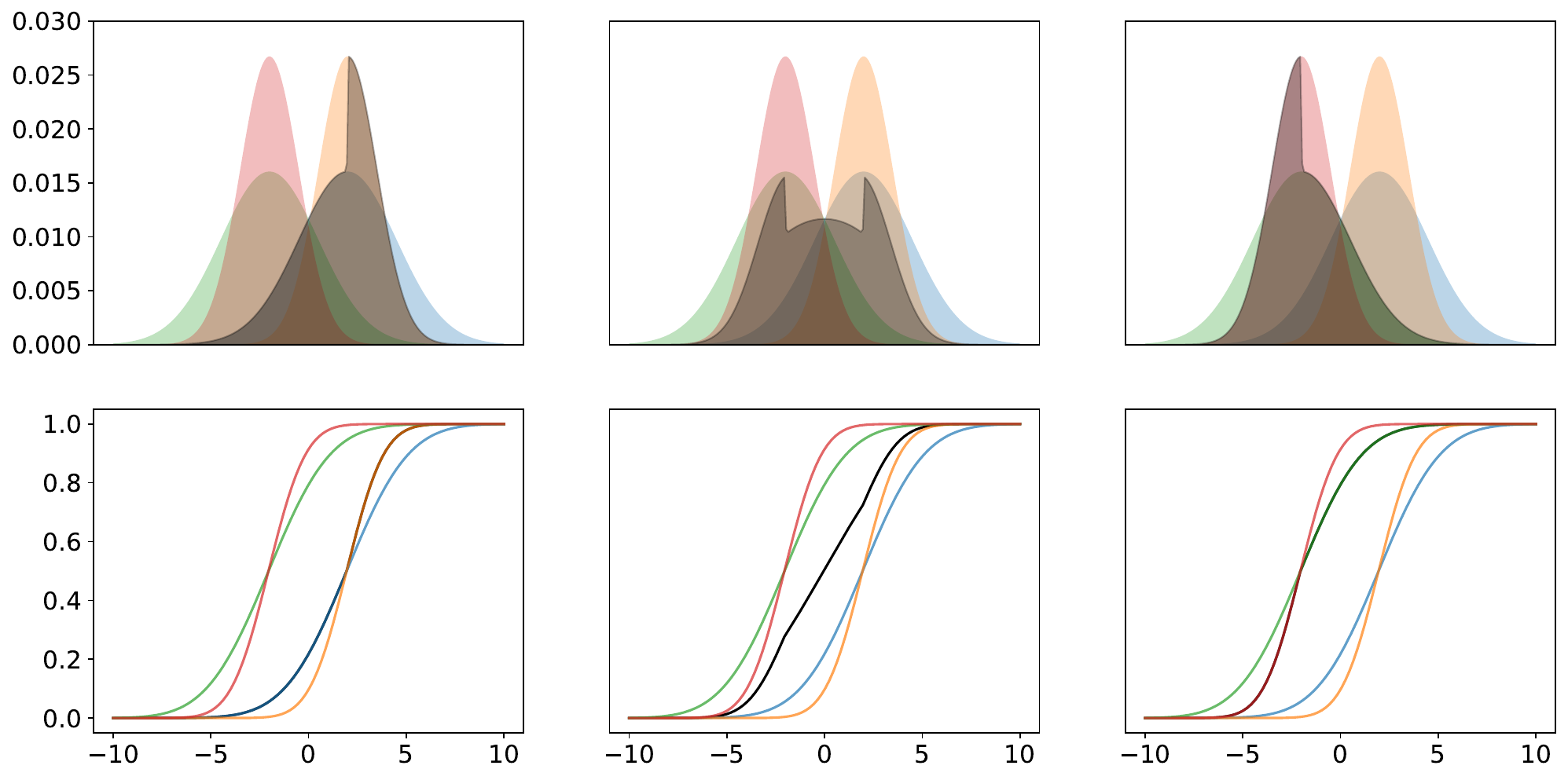}
	\caption{Three vertical selections for $\theta=0, \ 1/2, \ 1$.\\ \textcolor{white}{.}}
	\end{subfigure}
	\caption{Comparison between different Wasserstein median selections. In the second line we displayed the corresponding cumulative distribution functions.}
	\label{fig:intro}
    \end{figure}

Considering absolute continuity of medians, we first discuss the easier case of vertical selections:
	\begin{theorem}[\textbf{Vertical selections: absolute continuity}]\label{MainTheoremOneDimensionalVertical}
	Let $\bl\in \Delta_N$,  $\bnu=(\nu_1, \ldots, \nu_N)\in  \mathcal{P}_1(\R)^N$, $\theta\in [0,1]$, and $\nu^\theta:=\VMedl(\theta, \bnu)$. If $\nu_1,\dots,\nu_N$ are all absolutely continuous (with respect to the Lebesgue measure on $\R$) with densities $f_1,\dots,f_N\in L^1(\R)$  then $\nu^\theta$ is absolutely continuous with a density $f_{\nu^\theta} \in L^1(\R)$ which satisfies
			\begin{equation}\label{quant1}
			\min_{1 \leq i \leq N} f_i \le f_{\nu^\theta}  \leq \max_{1 \leq i \leq N} f_i, \quad \text{a.e.~on $\R$}.
			\end{equation}
	In particular, if,  for some $p\in [1, \infty]$, $f_i\in L^p(\R)$ for  $i=1, \dots, N$, then $f_{\nu^\theta} \in L^p(\R)$ and
	\begin{equation}\label{eq:1d_lp_bound_ver}
	 \| \min_{1 \leq i \leq N} f_i\|_{L^p(\R)}    \le \| f_{\nu^\theta}\|_{L^p(\R)} \leq \| \max_{1 \leq i \leq N} f_i\|_{L^p(\R)}\le \sum_{i=1}^N  \|f_i\|_{L^p(\R)}.
	\end{equation}
	\end{theorem}
\begin{proof}
Let $x\in \R$ and $h\geq 0$, it follows from \eqref{variationmpm} and the definition of the cdf $F_\theta$ that
\[ 0\leq F_\theta(x+h)-F_\theta(x)\leq \max_{1 \leq i \leq N} \{ F_{\nu_i}(x+h)-F_{\nu_i}(x)\} =  \max_{1 \leq i \leq N} \int_{x}^{x+h} f_i \leq  \int_{x}^{x+h}  \max_{1 \leq i \leq N} f_i,\]
which yields absolute continuity of $F_\theta$, i.e.~$\nu^\theta$ is absolutely continuous with respect to Lebesgue's measure, and the upper bound in \eqref{quant1}. In a similar fashion,
\[F_\theta(x+h)-F_\theta(x) \geq \int_{x}^{x+h} \min_{1 \leq i \leq N} f_i,\]
which shows shows the lower bound in \eqref{quant1} and concludes the proof.
\end{proof}
\noindent In particular, in dimension one, vertical medians  automatically select medians which inherit integrability properties of the sample measures, with simple explicit pointwise bounds. 

Let us now turn our attention to the case of horizontal selections which is slightly more involved.

	\begin{theorem}[\textbf{Horizontal selections: absolute continuity}] \label{MainTheoremOneDimensionalHorizontal}
Let $\bl\in \Delta_N$,  $\bnu=(\nu_1, \ldots, \nu_N)\in  \mathcal{P}_1(\R)^N$, $\theta\in [0,1]$, and $\mu^\theta:=\HMedl(\theta, \bnu)$. If $\nu_1,\dots,\nu_N$ are all absolutely continuous (with respect to the Lebesgue measure on $\R$) with densities $f_1,\dots,f_N\in L^1(\R)$  then:
\begin{itemize}
\item

 $\mu^0$ and $\mu^1$ are  absolutely continuous with densities $f_{\mu^0}$, $f_{\mu^1}$ which satisfy
			\begin{equation}\label{quanth1}
			\min_{1 \leq i \leq N} f_i \le  \min(f_{\mu^0}, f_{\mu^1}) \leq  \max(f_{\mu^0}, f_{\mu^1}) \leq \max_{1 \leq i \leq N} f_i, \ \text{ a.e.~on $\R$},
			\end{equation}

\item for every $\theta\in [0,1]$, 	$\mu^\theta$ is  absolutely continuous, we denote its density $f_{\mu^\theta}$,		
\item if,  for some $p\in [1, \infty]$, $f_i\in L^p(\R)$ for $i=1, \dots, N$, then $f_{\mu^\theta} \in L^p(\R)$ and
	\begin{equation}\label{eq:1d_lp_bound_hor}
	 \| f_{\mu^\theta}\|_{L^p(\R)} \leq \| \max_{1 \leq i \leq N} f_i\|_{L^p(\R)}\le \sum_{i=1}^N  \|f_i\|_{L^p(\R)}.
	\end{equation}
	
\end{itemize}	
	\end{theorem}
\begin{proof}
We shall proceed in three steps.

{\textbf{Step 1:}} Let us show \eqref{quanth1}, under the extra assumption that each $f_i$ satisfies
\begin{equation}\label{extra}
f_i \in L^{\infty}(\R), \; \frac{1}{f_i} \in L^{\infty}_\loc(\R).
\end{equation}
Recall that by construction
\[\mu^0:=Q^-_{\#} \LL, \; \mu^1:=Q^+_{\#} \LL \mbox{ with } Q^\pm:=\Ml^\pm(Q_{\nu_1}, \dots, Q_{\nu_N}),\]
and \eqref{extra} ensures that each $F_{\nu_i}$ is Lipschitz (with Lipschitz constant $\Vert f_i\Vert_{L^\infty(\R)}$) with inverse $Q_{\nu_i}$, which is locally Lipschitz on $(0,1)$, with $Q_{\nu_i}'$ satisfying 
\begin{equation}\label{easymongeampere0}
 f_i(Q_{\nu_i}) Q_{\nu_i}'  =1, \; \mbox{hence} \ Q_{\nu_i}' \geq \frac{1}{M} \mbox{ with } M:= \max_j  \Vert f_j\Vert_{L^\infty(\R)} \mbox{ a.e.~on $(0,1)$}.
 \end{equation}
Hence, $Q^\pm$ are locally Lipschitz on $(0,1)$, and it follows from \eqref{variationmpm} that for $ 0< t< s <1$, one has
\[Q^\pm(s)-Q^\pm(t) \ge \frac{1}{M} (s-t),\]
which implies that $Q^-=Q_{\mu^0}$ and $Q^+=Q_{\mu^1}$ have $M$-Lipschitz inverses which are the cdf's $F_{\mu^0}$ and $F_{\mu^1}$. Thus, $\mu^0$ and $\mu^1$ are absolutely continuous with bounded positive densities $f_{\mu^0}$,  $f_{\mu^1}$, and 
\begin{equation}\label{easymongeampere}
f_{\mu^0} (Q_{\mu^0}) Q_{\mu^0}'  =1, \; f_{\mu^1} (Q_{\mu^1}) Q_{\mu^1}' =1 \mbox{ a.e.~on $(0,1)$}.
\end{equation}
Using \eqref{variationmpmrefined}, we also have for  $ 0< t < s < 1$ with $|t-s|$ small enough
\[ \max_{i \in I_-(t) } (Q_{\nu_i}(s)-Q_{\nu_i}(t) )\ge Q_{\mu^0}(s)-Q_{\mu^0}(t) \geq \min_{i \in I_-(t) } (Q_{\nu_i}(s)-Q_{\nu_i}(t)),\]
 where  $I_-(t):=\{i \; : \;  Q^-(t)=Q_{\nu_i}(t)\}$. If we choose $t$ a point where all qdf's $Q_{\mu^0}$, \ $Q_{\nu_i}$ are differentiable and the change of variable formulas \eqref{easymongeampere0} and \eqref{easymongeampere} hold dividing the previous inequality by $(s-t)$ and letting $s\to t^+$ yields
 \[\begin{split}
 \max_{i \in I_-(t) } Q_{\nu_i}'(t)=\max_{i \in I_-(t)} \frac{1}{f_i(Q_{\mu^0}(t))}   \ge Q'_{\mu^0}(t)=\frac{1}{ f_{\mu^0}(Q_{\mu^0}(t))} \\
 \geq \min_{i \in I_-(t) } Q_{\nu_i}'(t)=\min_{i \in I_-(t) } \frac{1}{f_i(Q_{\mu^0}(t)),}
 \end{split}\]
  so that
 \[ \min_{ 1\leq i \leq N} f_i ( Q_{\mu^0}(t)) \le
 f_{\mu^0} ( Q_{\mu^0}(t)) \leq \max_{ 1\leq i \leq N} f_i ( Q_{\mu^0}(t))  \mbox{ for a.e.~$t\in(0,1)$}.\]
But since $\mu^0= {Q_{\mu^0}}_{\#} \LL$ has a positive and bounded density, it has the same null sets as $\LL$ hence the previous inequality can be simply reformulated as 
\[\min_{ 1\leq i \leq N} f_i  \le  f_{\mu^0}  \leq \max_{ 1\leq i \leq N} f_i   \mbox{ a.e.~on $\R$}.\]
The fact that $f_{\mu^1}$ obeys the same inequality can be proved in a similar way using \eqref{variationmpmrefined} for $\Ml^+$ instead of $\Ml^-$, we thus have shown \eqref{quanth1} under \eqref{extra}. Note also that the $L^p$ bound \eqref{eq:1d_lp_bound_hor} follows for $\theta\in\{0,1\}$. 
\smallskip

\textbf{Step 2:} Again assuming \eqref{extra}, let us  show absolute continuity of $\mu^\theta$ and the $L^p$ bound \eqref{eq:1d_lp_bound_hor} for $\theta\in (0,1)$. We shall proceed by a displacement convexity argument which is reminiscent of McCann's seminal work \cite{mccannconvexity}. Let us recall that $F_{\mu^0}$ is Lipschitz with locally Lipschitz inverse $Q^-$ so that ${F_{\mu^0}}_\# \mu^0=\LL$ and  
\[\mu^\theta=((1-\theta) Q^-+ \theta Q^+)_\# \LL= ((1-\theta) Q^-+ \theta Q^+)_\# ({F_{\mu^0}}_\# \mu^0)  =((1-\theta) \id+ \theta T)_\#\mu^0,\]
where $T:=Q^+ \circ F_{\mu^0}$ is the monotone transport from $\mu^0$ to $\mu^1$ so that $\mu^\theta$ is the \emph{displacement interpolation} between $\mu^0$ and $\mu^1$ as defined by McCann in \cite{mccannconvexity} (in the more general and involved multi-dimensional setting). Since the (locally Lipschitz) map $(1-\theta) \id + \theta T$ has a Lipschitz inverse and $\mu^0$ is absolutely continuous, $\mu^\theta$ is absolutely continuous, we then denote by $f_{\mu^\theta}$ its density.  Recalling that the qdf of $\mu^\theta$, $Q_\theta= (1-\theta) Q^-+ Q^+$ is locally Lipschitz, differentiable with a strictly positive derivative a.e.~and we have the change of variable formula:
\[f_{\mu^\theta}(Q_\theta)  Q_\theta'=1 \mbox{ a.e.}\]
Let $V:\R \to \R$ with $V(0)=0$ be convex, then the function $\alpha>0 \mapsto \Phi(\alpha) := \alpha V(\alpha^{-1})$ is convex as well and then we have
\[\begin{split} \int_{\R} V(f_{\mu^\theta}(x))\d x&= \int_{0}^1 V \left(\frac{1}{{Q_{\theta}}'(t)} \right) {Q_{\theta}}'(t)  \d t =\int_0^1 \Phi ((1-\theta) Q_{\mu^0}'(t)+ \theta Q_{\mu^1}'(t))) \d t\\
&\leq (1-\theta) \int_0^1  \Phi (Q_{\mu^0}'(t)) \d t + \theta \int_0^1  \Phi (Q_{\mu^1}'(t)) \d t\\
&= (1-\theta) \int_\R  V(f_{\mu^0}(x)) \d x + \theta \int_\R  V (f_{\mu^1}(x)) \d x.
\end{split}\]
Taking $V(\alpha)=\vert \alpha\vert^p$, recalling \eqref{quanth1} we in particular get
\[\int_{\R} (f_{\mu^\theta}(x))^p\d x \leq (1-\theta) \int_{\R} (f_{\mu^0}(x))^p\d x + \theta\int_{\R} (f_{\mu^1}(x))^p\d x \leq \Vert \max_{1\leq i \leq N}  f_i \Vert_{L^p(\R)}^p,\]
which gives  \eqref{eq:1d_lp_bound_hor}.

\smallskip 
\textbf{Step 3}: general case by Lemma \ref{vhselarelip}.  To get rid of the extra assumption \eqref{extra}, let $g$ be the density of a standard Gaussian measure and for $\eps>0$ set
\[f_i^\eps := \frac{   \min ((1-\eps)  f_i +\eps g, \eps^{-1})} {\int_\R   \min( (1-\eps)  f_i +\eps g, \eps^{-1}) }.\]
Applying the previous steps to $\mu_\eps^\theta:=\HMedl(\theta, f_1^\eps, \dots, f_N^\eps)$, we get
\[ \min_{1 \leq i \leq N} f_i^\eps \le  \min(f_{\mu_\eps^0}, f_{\mu_\eps^1}) \leq  \max(f_{\mu_\eps^0}, f_{\mu_\eps^1}) \leq \max_{1 \leq i \leq N} f_i^\eps,\]
and for every $\theta\in [0,1]$, 
\[ \| f_{\mu_\eps^\theta}\|_{L^p(\R)} \leq \| \max_{1 \leq i \leq N} f_i^\eps\|_{L^p(\R)}.\]

 Since $f_i^\eps$ converges to $f_i$ in $L^p$ and  $\mu_\eps^\theta$  converges to $\mu^\theta$ in Wasserstein distance  thanks to Lemma \ref{vhselarelip}, we can pass to the limit $\eps\to 0^+$ in these bounds, obtaining \eqref{quanth1} and  \eqref{eq:1d_lp_bound_hor}.
\end{proof}

	\begin{remark}
	For Wasserstein barycenters (in any dimension), the fact that one sample measure with positive weight is $L^p$ implies that  the barycenter is $L^p$ as well (see \cite{agueh_barycenters_2011}). For Wasserstein medians in one dimension, we really need  all sample measures to be $L^p$ to find an $L^p$ median. To see this, recall that the median of $\nu_1:=\delta_{x}$ with weight $2/3$ and any probability $\nu_2\in \cP_1(\R)$  (with the smoothest density one can think of) with weight $1/3$ is  $\delta_{x}$.  Note also that due to the fact that $\Ml^{\pm}$ are Lipschitz but nonsmooth, vertical and horizontal median selections of sample measures with smooth (or Sobolev) densities do not have a continuous density in general.  
	\end{remark}

\section{Multi-marginal and dual formulations}\label{sec:dual_multimarginal}
 
\subsection{Multi-marginal formulation} 
Given the proper metric space $(\cX, d)$, $\bl \in \Delta_N$ and $\bnu=(\nu_1, \dots, \nu_N)\in \cP_1(\cX)^N$, the Wasserstein median problem  \eqref{eq:Medianprob} is, like the Wasserstein barycenter problem, a special instance of the \emph{matching for teams} problem \cite{carlier_matching_2010} and, as such, admits linear reformulations which take the form of multi-marginal optimal transport problems. Let us now recall this reformulation in the Wasserstein median context. For $\bx:=(x, x_1, \dots, x_N)\in \cX^{N+1}$, let us define:
\begin{equation}
f_{\bl} (x, x_1, \dots, x_N):=\sum_{i=1}^N \lambda_i d(x_i, x), \; c_{\bl}(x_1, \dots, x_N):=\min_{y\in \cX} f_{\bl} (y, x_1, \dots, x_N),
\end{equation}
and the projections:
\[\pi_0(\bx)=x, \; \pi_j(\bx)=x_j, \; 1\leq j \leq N, \; \pi_{0,j} (\bx)=(x, x_j), \; \pi_{1, \dots, N} (\bx)=(x_1, \dots, x_N).\]
We denote by $\Pi(\nu_1,\dots,\nu_N)$ the set of Borel probability measures on $\cX^N$ having $\nu_i$ as $i$-th marginal and the linear multi-marginal problems
\begin{equation}\label{eq:MultMargwMed}
\inf \left\{  \int_{\cX^{N+1}} f_{\bl} \d \theta\; : \; \theta \in \cP_1(X^{N+1}),  \; {{\pi_{1,\dots, N}}_{\#} \theta}\in  \Pi(\nu_1,\dots,\nu_N) \right\},
\end{equation}
and 
\begin{equation}\label{eq:MultMargInfForm}
\inf_{\gamma \in \Pi(\nu_1,\dots,\nu_N)}   \int_{\cX^{N}} c_{\bl} \d \gamma.
\end{equation}
Since $(\cX, d)$ is  Polish, it follows from the disintegration theorem (see paragraph 5.3 in \cite{ambrosio_gradient_2005}) that if $\theta$ is admissible for \eqref{eq:MultMargwMed} it can be disintegrated with respect to its marginal 
$\gamma:={{\pi_{1,\dots, N}}_{\#} \theta}$ as 
\[\theta= \theta^{x_1, \dots, x_N} \otimes \gamma,\]
for a Borel family of conditional probability measures $\theta^{x_1, \dots, x_N}$ on $\cX$. For fixed  marginal  $\gamma:={{\pi_{1,\dots, N}}_{\#} \theta}$, minimizing with respect to the conditional probability $\theta^{x_1, \dots, x_N}$ the integral of $f_{\bl}$  obviously amounts to choose it supported on $\Ml(x_1, \dots, x_N)$ so that it is easy to see that \eqref{eq:MultMargwMed} and \eqref{eq:MultMargInfForm} are equivalent in the sense that they have the same value and that $\theta$ solves \eqref{eq:MultMargwMed} if and only if $\gamma:={\pi_{1,\dots, N}}_{\#} \theta$ solves  \eqref{eq:MultMargInfForm} and $\theta$ is supported by the set of $(x, x_1, \dots, x_N)$ such that $x\in \Ml(x_1, \dots, x_N)$. The fact that $\Pi(\nu_1, \dots, \nu_N)$ is tight and the properness of $(\cX,d)$ ensure that the infimum in both \eqref{eq:MultMargwMed} and \eqref{eq:MultMargInfForm} is attained.  The connection with the Wasserstein median problem  \eqref{eq:Medianprob} and its solutions $\Medl(\bnu)$ is summarized by:

\begin{theorem}\label{multimarginal} 
	The following  hold:
		\begin{enumerate}
		\item $\min\eqref{eq:Medianprob} = \min\eqref{eq:MultMargwMed} = \min\eqref{eq:MultMargInfForm}$.\label{item:multimarginal_1}
	\item If $\nu\in \Medl(\bnu)$ i.e.~$\nu$ solves $\eqref{eq:Medianprob}$, then, there exists $\theta$  solving $\eqref{eq:MultMargwMed}$ such that $\nu={\pi_0}_\# \theta$ and, conversely, if $\theta$ solves \eqref{eq:MultMargwMed}, then ${\pi_0}_{\#}\theta \in \Medl(\bnu)$. \label{item:multimarginal_2}
		\item If $\theta$  solves $\eqref{eq:MultMargwMed}$ and $\nu={\pi_0}_\# \theta$, then for every $j$ such that $\lambda_j>0$, $\gamma_j :={\pi_{0,j}}_\# \theta$ is an optimal transport plan between $\nu$ and $\nu_j$. \label{item:multimarginal_3}
		\item \label{item:median_push} If $\mathsf{m}_{\lambda}:\cX^N\to\cX$ is a Borel selection of $\Ml$  and  $\gamma$ solves \eqref{eq:MultMargInfForm}, then   ${(\mathsf{m}_{\boldsymbol{\lambda}})}_\#\gamma\in \Medl(\bnu)$.\label{item:multimarginal_4}
		\item  $\nu\in \Medl(\bnu)$ if and only if there exists $\gamma$ solving \eqref{eq:MultMargInfForm} and a Borel family of probability measures $\theta^{x_1, \dots, x_N}$  such that  $\theta^{x_1, \dots, x_N}$  is supported on $\Ml(x_1,\dots, x_N)$ for $\gamma$-a.e.~$(x_1, \dots, x_N)$ and $\nu= {\pi_0}_\# (\theta^{x_1, \dots, x_N}\otimes \gamma)$. \label{item:multimarginal_5}
		
	\end{enumerate}
\end{theorem}

The proof of similar results can be found in \cite{carlier_matching_2010} and therefore omitted. Even though \cite{carlier_matching_2010} consider  a compact setting (with general costs), the same proof easily adapts to the present setting of a proper metric space with the distance as cost. Note that, as in \cite{carlier_matching_2010}, one can deduce from a Wasserstein median $\nu \in \Medl(\bnu)$ a solution of \eqref{eq:MultMargwMed} with first marginal $\nu$ as follows: let  $\gamma_i$ be an optimal plan between $\nu$ and $\nu_i$ disintegrated with respect to $\nu$ as $\gamma_i =\nu \otimes \gamma_i^x$ and define $\theta$ by gluing i.e.:
\[\int_{\cX} \phi(x, x_1, \dots, x_N) \d \theta(x, x_1, \ldots, x_N):=\int_{\cX} \Big( \int_{\cX^N} \phi(x, x_1, \dots, x_N)  \d \gamma_1^x(x_1) \dots  \d \gamma_N^x(x_N)  \Big) \d  \nu(x)\]
for every $\phi \in C_b(\cX^{N+1})$. Then, $\theta$ solves \eqref{eq:MultMargwMed} and by construction ${\pi_0}_\# \theta=\nu$. Note that pushing forward by a selection of $\Ml$ a solution of the multi-marginal \eqref{eq:MultMargInfForm} as in \ref{item:multimarginal_4} above corresponds to special medians for which, using the notation of \ref{item:multimarginal_5}, $\theta^{x_1,\dots, x_N}=\delta_{\mathsf{m}_{\lambda}(x_1, \ldots, x_N)}$ is a Dirac mass. Since $\Ml$ is in general not single-valued, not all medians are of this form. Consider for instance $\mathcal{X}=[-1,1]$ equipped with the usual Euclidean distance and let $\nu_1=\delta_{-1/2}$ and $\nu_2=\delta_{1/2}$ with uniform weights. Then $\Medl(\nu_1,\nu_2)$ is the set of all probability measures supported on $[-1/2,1/2]$ whereas \ref{item:multimarginal_4} only selects Dirac masses.

Let us now give an application  of Theorem \ref{multimarginal}:
\begin{cor}[\textbf{Moment bounds}]\label{pMoments}
Let $\cX=\mathbb{R}^d$ be equipped with the Euclidean distance. If all the sample measures  are supported on a closed convex subset $\mathcal{K}\subset \mathbb{R}^d$, then every Wasserstein median $\nu\in \Medl(\bnu)$ is supported on $\mathcal{K}$ as well. Moreover, if $V:\R^d \to \R_+$ is quasiconvex (i.e.~$\{V \leq t\}$ is convex for every $t\ge 0$) then
	\[ \int_{\R^d} V \d \nu \leq \sum_{i=1}^N \int_{\R^d}  V \d \nu_i.\]
	In particular, for any $p\in (0, +\infty)$ we have the following bound on the $p$-moments of $\nu$:
	\begin{equation}\label{PMomentsBound}
		\int_{\mathbb{R}^d} |x|^p \d\nu\leq \sum_{i=1}^{N}\int_{\mathbb{R}^d} |x|^p \d \nu_i.
	\end{equation}
\end{cor}
\begin{proof}
We know from point \ref{item:multimarginal_4} of Theorem \ref{multimarginal}, that there exists $\gamma\in \Pi(\nu_1, \dots \nu_N)$ and a family of probability $\theta^{x_1, \dots, x_N}$ supported by $\Ml(x_1, \dots, x_N)$ such that for every continuous and bounded (or more generally Borel) function $f:\R^d \to \R$ one has
\[\int_{\R^d} f(x) \d \nu(x) =\int_{(\R^d)^N} \Big(  \int_{\R^d}  f(x) \d \theta^{x_1, \dots, x_N}(x) \Big) \d \gamma(x_1, \dots, x_N),\]
but since $\theta^{x_1, \dots, x_N}$  is supported by $\Ml(x_1, \dots, x_N)\subset \mathrm{co} \{x_1, \ldots, x_N\}$ (as we have seen in Example \ref{Torricelli}), if all the $\nu_i$'s are supported by the closed convex set $\mathcal{K}$ then so is $\theta^{x_1, \dots, x_N}$ for $\gamma$-a.e.~$(x_1, \dots, x_N)$ and then $\nu(\mathcal{K})=1$. Likewise, if $V$ is nonnegative and quasiconvex, then for $\theta^{x_1, \dots, x_N}$ a.e.~$x$ we have
\[V(x) \leq \max_{1\leq i \leq N } V(x_i) \le \sum_{i=1}^N V(x_i),\]
and integrating this inequality with respect to $\theta^{x_1, \dots, x_N}$ first and then with respect to $\gamma$ in $\Pi(\nu_1, \dots, \nu_N)$ we obtain the announced moment bounds. 
\end{proof}

\paragraph{A counterexample to linear $L^{\infty}$ density bounds in several dimensions.} We have seen in Theorems \ref{MainTheoremOneDimensionalVertical} and \ref{MainTheoremOneDimensionalHorizontal}  that  when $\cX=\R$, and the sample measures have densities uniformly bounded by some $M$,  vertical and horizontal median selections enable to find Wasserstein medians with a density which is bounded by the same bound $M$. In other words, in dimension one,  it is possible to have a linear control on the $L^\infty$ norm of some well-chosen Wasserstein median in terms of $L^\infty$ bounds of the sample measures. The situation seems to be  more intricate in higher dimensions. The following example shows that a linear $L^\infty$-bound cannot hold in two dimensions.

\begin{example}\label{ex:counterex-reg}

For $0<\epsilon<1$ let $\nu_1$ be a uniform measure supported on the rectangle $[-1-\ell, -1]\times[-\tfrac{\epsilon}{2}, \tfrac{\epsilon}{2}]$ and let $\nu_2$, $\nu_3$ and $\nu_4$ be obtained by successive rotations by $90^\circ$ of $\nu_1$ as in Figure \ref{fig:counterex}. Consider uniform weights $\lambda_i=\tfrac{1}{4}$, for $i=1, \dots, 4$, and let $\nu \in \Medl(\nu_1, \dots, \nu_4)$. We know from Theorem \ref{multimarginal} that one can write $\nu:={\pi_0}_\# \theta$ where ${\pi_i}_\# \theta=\nu_i$ for $i=1, \ldots, 4$ and $x$ is a geometric median of $(x_1,\dots,x_4)$ for $\theta$-a.e.~$(x,x_1,\dots,x_4)$. Now note that, with this construction, four points $x_i \in \spt  \; \nu_i$ always form a convex quadrilateral, and as shown in Theorem 1 in \cite{plastria_four-point_2006},  their unique median is the intersection of the two segments $[x_1,x_3]$ and $[x_2,x_4]$. In particular such geometric medians belong to the square $[-\frac{\epsilon}{2},\frac{\epsilon}{2}]^2,$  which therefore supports any $\nu \in \Medl(\nu_1, \dots, \nu_4)$. This shows that the $L^\infty$ norm of $\nu$ is at least $\eps^{-2}$: it cannot be bounded from above uniformly in $\eps$  by a multiple of  $ \max_{i=1,\dots, 4} \Vert \nu_i \Vert_{L^\infty}=\ell^{-1}\eps^{-1}$.

 \begin{figure}[!t]
    \centering
    \begin{tikzpicture}[scale=0.7]
    \filldraw[fill=gray!20] (0, 4) rectangle (1, 1);
    \filldraw[fill=gray!20] (-4, 0) rectangle (-1, -1);
    \filldraw[fill=gray!20] (2, 0) rectangle (5, -1);
    \filldraw[fill=gray!20] (0, -2) rectangle (1, -5);
    \filldraw[fill=gray!20] (0, -2) rectangle (1, -5);

    \filldraw[fill=green!20] (0, 0) rectangle (1, -1);

    \draw [dashed] (0, 1) -- (0, -2);
    \draw [dashed] (1, 1) -- (1, -2);
    \draw [dashed] (-1, 0) -- (2, 0);
    \draw [dashed] (-1, -1) -- (2, -1);

    \draw [decorate, decoration = {brace, raise=5pt, amplitude=5pt}] (5.2, 0) --  (5.2, -1);
    \draw [decorate, decoration = {brace, raise=5pt, amplitude=5pt}] (2, 0.2) --  (5, 0.2);

    \node[text width=1cm] at (4.1, 1.1) {$\ell$};
    \node[text width=1cm] at (6.6, -0.5) {$\epsilon$};

    \node[text width=1cm] at (-2, -0.5) {$\nu_1$};
    \node[text width=1cm] at (4, -0.5) {$\nu_3$};
    \node[text width=1cm] at (0.95, 2.5) {$\nu_2$};
    \node[text width=1cm] at (0.95, -3.5) {$\nu_4$};

    \matrix [above left] at (11.5, -4.5) {
      \node[fill=gray!20, label=right:{Sample measures}] {};\\
      \node[fill=green!20, label=right:{Maximal support of Wasserstein medians}] {};\\
    };
   
\end{tikzpicture}
    \caption{Counterexample to linear $L^\infty$ bounds in dimension two. The support of the four given uniformly distributed measures are indicated in gray. The support of any Wasserstein median is contained in the green area. Confer Example \ref{ex:counterex-reg} for further details.}
    \label{fig:counterex}
\end{figure}
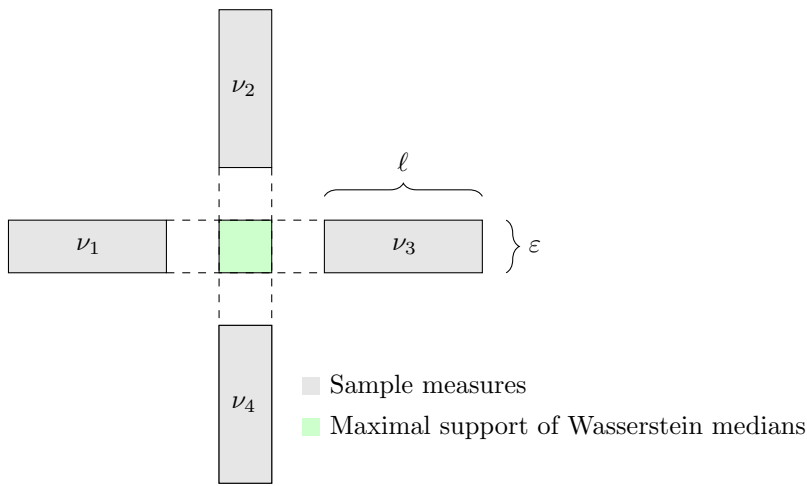

\end{example}

\subsection{Dual Formulation} To introduce a dual formulation \`a la Kantorovich of \eqref{eq:Medianprob}, we fix a point $x_0 \in \cX$ and define the spaces
\begin{align*}
    Y_0 := \left\{ f \in C(\cX)\; : \; \lim_{d(x,x_0) \rightarrow \infty} \frac{f(x)}{1 + d(x,x_0)} = 0\right\}, \quad
    Y_b := \left\{ f \in C(\cX)\; : \; \sup_{x \in \cX} \frac{\vert f(x)\vert}{1 + d(x,x_0)} < \infty\right\}.
\end{align*}
Note that these spaces are independent of the choice of  $x_0$ and that the dual of $Y_0$ may be identified with the space of signed measures
with finite first moment
\begin{equation*}
    (Y_0)^* =\left\{\mu \in \cM(\cX)\; : \; (1+d(x,x_0))\mu \in \cM(\cX)  \right\}.
\end{equation*}
We will also assume here without loss of generality that all the weights $\lambda_i$ are strictly positive  in the Wasserstein median problem \eqref{eq:Medianprob} and define for $\lambda>0$:
\[\Lip_{\lambda}(\cX):=\{ v : \cX \to \R, \; \vert v(x)-v(y)\vert \le \lambda d(x,y), \; \ \text{for all} \  (x,y)\in \cX^2\}.\]
Setting $c_i:=\lambda_i d$, the $c_i$-transform of a function $u:\cX \to \R$, denoted $u^{c_i}$, is by definition given by
\[u^{c_i}(x):=\inf_{y \in \cX} \{\lambda_i d(x,y)-u(y) \}, \; \text{for all} \ x\in \cX,\]
note that, by the triangle inequality  $u^{c_i}$ is either everywhere $-\infty$ or a $\lambda_i$-Lipschitz function. It is also a classical fact (see, e.g.~Proposition 3.1 in \cite{santambrogio_optimal_2015}) that $u\in \Lip_{\lambda_i}(\cX)$ if and only if $u^{c_i}=-u$. 

Following \cite{agueh_barycenters_2011}, let us now consider the concave maximization problem
\begin{equation}\label{eq:Predual}
    \quad \sup \bigg\{\sum_{i=1}^{N}\int_{\cX} u_i^{c_i} d\nu_i\; : \; u_i \in Y_0, \ \sum_{i=1}^{N} u_i=0\bigg\}, 
\end{equation}
and its relaxed version 
\begin{equation}\label{eq:PredualRelax}
    \quad \sup \ \bigg\{\sum_{i=1}^{N}\int_{\cX} u_i^{c_i} \d \nu_i \; : \; u_i \in Y_b, \ \sum_{i=1}^{N} u_i=0 \bigg\}. 
\end{equation}
By definition of the $c_i$-transform, it is easy to check the weak duality relation
\begin{equation*}
    \min\eqref{eq:Medianprob} \geq \sup \eqref{eq:PredualRelax} \geq \sup \eqref{eq:Predual}.
\end{equation*}

Using convex duality by proceeding exactly  as in  the proof of Propositions 2.2 and 2.3 in \cite{agueh_barycenters_2011} for the Wasserstein barycenter case, one can show that \eqref{eq:Medianprob} is the dual of \eqref{eq:Predual} and that strong duality holds i.e.:  $\min\eqref{eq:Medianprob} = \sup \eqref{eq:PredualRelax} =\sup \eqref{eq:Predual}$. It will be convenient in the sequel to consider yet another formulation of  \eqref{eq:PredualRelax}:
\begin{equation}\label{eq:DualLip}
	\sup \biggl\{ \sum_{i=1}^{N} \int_{\cX} u_i \d \nu_i  \; : \;  \ u_i \in \Lip_{\lambda_i}(\cX), \; i=1, \dots, N,  \;  \sum_{i=1}^{N}u_i\leq 0 \biggr\}. 
	\end{equation}

\begin{prop}[\textbf{Lipschitz formulation of the dual problem}]\label{Lipformdual}  Let $(\nu_1, \dots, \nu_N)\in \mathcal{P}_1(\cX)^N$ and $\bl:=(\lambda_1, \dots, \lambda_N)\in \Delta_N$ with each $\lambda_i$ strictly positive. Then we have
\begin{equation}\label{dualKantomedian}
    \min \eqref{eq:Medianprob} = \sup \eqref{eq:PredualRelax}=  \max \eqref{eq:DualLip},
\end{equation}
where we have written max \eqref{eq:DualLip} to emphasize that the supremum in \eqref{eq:DualLip} is attained.

\end{prop}
\begin{proof}
Recall that  $\min \eqref{eq:Medianprob} = \sup \eqref{eq:PredualRelax}$.  
\smallskip

{\textbf{Step 1}:}  $\sup\eqref{eq:PredualRelax}\geq\sup\eqref{eq:DualLip}$. Let $(u_1,\dots,u_N)$ be admissible for \eqref{eq:DualLip}, take $\boldsymbol{\psi}=(\psi_1,\dots,\psi_N)$, with $\psi_i = -u_i$ for all $i=1,\dots, N-1$ and $\psi_N = u_1 + \dots + u_N$. Since Lipschitz functions belong to $Y_b$, $\boldsymbol{\psi}$ is admissible for \eqref{eq:PredualRelax} and we have:
	\begin{equation*}
	\sum_{i=1}^{N}\int_{\cX} u_i \d\nu_i =\sum_{i=1}^{N-1}\int_{\cX} (-\psi_i) \d\nu_i+\int_{\cX} u_N \d\nu_N.
	\end{equation*}
	For $i=1, \dots, N-1$, since  $\psi_i \in \Lip_{\lambda_i}(\cX)$, we have $-\psi_i=\psi_i^{c_i}$. Moreover, $\psi_N=u_1+\dots + u_{N-1}\leq -u_N$, hence $\psi_N^{c_N} \geq u_N$, yielding
	\begin{equation*}
	\sum_{i=1}^{N}\int_{\cX}  u_i \d\nu_i \leq \sum_{i=1}^{N}\int_{\cX}  \psi_i^{c_i} \d\nu_i \leq \sup\eqref{eq:PredualRelax} .
	\end{equation*}
	
\smallskip

{\textbf{Step 2}:}  $ \sup\eqref{eq:DualLip} \geq \sup\eqref{eq:PredualRelax}$. Let $\boldsymbol{\psi}=(\psi_1,\dots,\psi_N)$ be admissible for \eqref{eq:PredualRelax}. Consider $\boldsymbol{u}=(u_1,\dots,u_N)=(\psi_1^{c_1},\dots,\psi_N^{c_N})$. By construction, each $u_i$ is  $\lambda_i$-Lipschitz and to see  that $\boldsymbol{u}$ is admissible for \eqref{eq:DualLip} we observe that for every $x\in \cX$:
	\begin{equation*}
	\sum_{i=1}^N u_i(x)=\sum_{i=1}^N \psi_i^{c_i}(x)=\sum_{i=1}^N \inf_y \{\lambda_i d(x,y)-\psi_i(y) \}\leq -\sum_{i=1}^N \psi_i(x)=0,
	\end{equation*}
	and, then,
	\[\sum_{i=1}^N \int_{\cX} \psi_i^{c_i } \d \nu_i =\sum_{i=1}^N \int_{\cX} u_i \d \nu_i  \leq \sup\eqref{eq:DualLip}.\]
	
\smallskip

{\textbf{Step 3}:} the supremum is attained in $\sup\eqref{eq:DualLip}$. We note that both constraints and the objective function in $\eqref{eq:DualLip}$ are unchanged when one replaces $u_i$ by $u_i+ \alpha_i$ where  the $\alpha_i$'s are constant that sum to $0$, we may therefore restrict the maximization in $\sup\eqref{eq:DualLip}$ to the smaller admissible set of potentials $(u_1, \dots, u_N)$ such that
\begin{equation}\label{furtherc1}
 u_i \in \Lip_{\lambda_i}(\cX), \ \sum_{i=1}^N u_i \leq 0, \ \text{and} \ \int_\cX u_i \d \nu_i=0, \quad \text{for} \ i=1, \dots, N-1.
\end{equation} 
Since this set contains $(0, \ldots, 0)$ we can reduce it even further by considering only potentials for which the objective is positive: 
\begin{equation}\label{furtherc2}
\int_\cX u_N \d \nu_N \geq 0.
\end{equation}
If we denote by $K$ the set of potentials that satisfy \eqref{furtherc1} and \eqref{furtherc2}, we observe that if $(u_1, \dots, u_N)\in K$ then for $i=1, \dots, N-1$ and $x\in \cX$, since $u_i$ is $\lambda_i$-Lipschitz, one has
\[  u_i(x) \leq \int_{\cX} u_i \d \nu_i + \lambda_i \int_{\cX} d(x, y) \d \nu_i (y) \le  \lambda_i d(x,x_0) + m_i, \; m_i:=\lambda_i \int_{\cX} d(x_0, y) \d \nu_i(y).\] 
Reasoning in a similar way for $-u_i$, we get bounds with linear growth, namely $\vert u_i \vert \leq  \lambda_i d(\cdot,x_0)+ m_i$ for $i=1, \ldots, N-1$. Since $u_N \leq -\sum_{i=1}^N u_i$ we get a similar upper bound with linear growth  for $u_N$, and, for a lower bound, we use \eqref{furtherc2} which, together with the fact that $u_N$ is $\lambda_N$-Lipschitz gives
\[ u_N \geq -\lambda_N d(\cdot, x_0)-\lambda_N \int_{\cX} d(x_0, y) \d \nu_N(y).\]
Let us now take a maximizing sequence in $K$ for \eqref{eq:DualLip}. The above linear bounds and Ascoli--Arzel\`a's theorem guarantee that this sequence converges locally uniformly to some  $\boldsymbol{u}$, and  again by these linear bounds, the fact that  $\nu_i\in \cP_1(\cX)$ for all $i =1,\dots, N$, and Lebesgue's dominated convergence theorem, one deduces that $\boldsymbol{u}\in K$ and $\boldsymbol{u}$ actually solves \eqref{eq:DualLip}.
\end{proof}

We may derive from the primal-dual relations between \eqref{eq:Medianprob} and \eqref{eq:DualLip} a characterization of Wasserstein medians in terms of Kantorovich potentials

 \begin{theorem}[\textbf{Optimality conditions for Wasserstein medians}] \label{PrimalDualOptimalityConditions}
	Let $\bnu=(\nu_1, \dots, \nu_N) \in \cP_1(\cX)^N$, $\bl=(\lambda_1, \dots, \lambda_N) \in \Delta_N$ with $\lambda_i>0$ and let $\nu \in \cP_1(\cX)$. Then $\nu \in \Medl(\bnu)$ if and only if there exist $\psi_1, \dots, \psi_N$ such that
	\begin{enumerate}
	\item for $i=1, \dots, N$, $\psi_i \in \Lip_1(\cX)$ is a Kantorovich potential between $\nu_i$ and $\nu$, i.e.
	\[W_1(\nu_i, \nu)=\int_{\cX} \psi_i \d \nu_i- \int_{\cX} \psi_i \d \nu,\]
	\item there holds 
		\[\sum_{i=1}^N \lambda_i \psi_i \leq 0 \ \text{on $\cX$, and} \ \sum_{i=1}^N \lambda_i \psi_i = 0 \mbox{ on $\spt (\nu)$}. \]
	\end{enumerate}
	 
 \end{theorem}
\begin{proof}
It follows from the duality result of proposition \ref{Lipformdual} that $\nu \in \Medl(\bnu)$ if and only if there exists $(u_1, \ldots, u_N)$ admissible for \eqref{eq:DualLip} such that
\begin{equation}\label{slackmed}
\sum_{i=1}^N \lambda_i W_1(\nu_i, \nu)=\sum_{i=1}^N \int_{\cX} u_i \d \nu_i
\end{equation}
(in which case $(u_1, \ldots, u_N)$ automatically solves \eqref{eq:DualLip}). Setting $\psi_i= u_i/ {\lambda_i}$ we thus have   $\psi_i \in \Lip_1(\cX)$ and $\sum_{i=1}^N \lambda_i \psi_i \le 0$ on $\cX$. By the Kantorovich--Rubinstein duality formula \eqref{kantorub}, we have
\begin{equation}\label{firstkcweak}
   W_1(\nu_i, \nu) \ge \int_{\cX} \psi_i \d \nu_i- \int_{\cX} \psi_i \d \nu.
\end{equation}
Multiplying by $\lambda_i$ summing and using the fact that $\nu$ is a nonnegative measure and $\sum_{i=1}^N \lambda_i \psi_i \leq 0$ thus yields
\[\begin{split}
\sum_{i=1}^N \lambda_i W_1(\nu_i, \nu)\geq  \sum_{i=1}^N \lambda_i \int_{\cX} \psi_i \d \nu_i - \sum_{i=1}^N \lambda_i  \int_{\cX} \psi_i \d \nu\\
\ge  \sum_{i=1}^N \lambda_i \int_{\cX} \psi_i \d \nu_i=
\sum_{i=1}^N \int_{\cX} u_i \d \nu_i,
\end{split}\]
so that \eqref{slackmed} holds if and only if each inequality \eqref{firstkcweak} is an equality, i.e.~$\psi_i$ is a Kantorovich potential between $\nu_i$ and $\nu$ and 
\[  \int_{\cX} \Big(\sum_{i=1}^N \lambda_i \psi_i\Big) \d \nu=0,\]
i.e.~$\sum_{i=1}^N \lambda_i \psi_i = 0$ on $\spt (\nu)$ since each $\psi_i$ is continuous. 
\end{proof}

\section{Beckmann minimal flow formulation}\label{sec:beckmann}

In this section, we consider the Wasserstein median problem on a  convex compact subset $\Omega$ of  $\R^d$, with non empty interior (which is not really restrictive) equipped with the Euclidean distance. In this setting, we will see that, taking advantage of the so-called Beckmann's minimal flow formulation of Monge's problem, one can derive a system of PDEs that characterize Wasserstein medians. We are given $\bl\in \Delta_N$ with $\lambda_i>0$ for all $i=1,\dots, N$, and $\bnu=(\nu_1, \dots, \nu_N)\in \cP(\Omega)^N$, we know from Corollary \ref{pMoments} that any measure in $\Medl(\bnu)$ is supported on $\Omega$.

\subsection{The Beckmann problem} 

We denote by $\mathcal{M}(\Omega,\mathbb{R}^d)$ the set of vector valued measures on $\Omega$. For such a measure $\sigma$, we denote by $\vert \sigma\vert \in \mathcal{M}_+(\Omega)$ its total variation measure and recall that one can write $\d \sigma = \hat{\sigma} \d \vert \sigma \vert$ for some Borel map $\hat{\sigma}$ such that $ \vert \hat{\sigma} \vert =1$, $\vert \sigma\vert$-a.e.;  for every test-function $\phi \in C(\Omega, \R^d)$, one can therefore write
\[\int_{\Omega} \phi \cdot \; \d \sigma = \int_{\Omega} \phi(x) \cdot \hat{\sigma}(x) \; \d  \vert \sigma \vert(x).\]

Let us denote by $\mathcal{M}_\div(\Omega,\mathbb{R}^d)$ the set of vector valued measures $\sigma$ whose divergence $\nabla \cdot \sigma$ is a finite measure, where $\nabla \cdot \sigma$ is defined in the sense of distributions. Given $i=1, \dots, N$ and $\nu \in \cP(\Omega)$, a vector-valued measure $\sigma_i\in \mathcal{M}_\div(\Omega,\mathbb{R}^d)$ is an admissible flow between $\nu_i$ and $\nu$ if it solves
\[ \nabla \cdot \sigma_i +\nu_i=\nu\]
in the weak sense, i.e.
\[\int_{\Omega} \nabla \phi \cdot \; \d \sigma_i=\int_{\Omega} \phi \; \d (\nu_i-\nu), \; \text{for all} \ \phi \in C^1(\Omega).\]
Beckmann's formulation of the optimal transport problem with distance cost between $\nu_i$ and $\nu$  consists in finding an admissible flow with minimal total variation, it thus reads
\begin{equation}\label{beckmannnuinu}
\inf_{\sigma_i  \in \mathcal{M}_\div(\Omega,\mathbb{R}^d)} \left\{ \vert \sigma_i \vert (\Omega) \; : \; \nabla \cdot \sigma_i +\nu_i=\nu\right\}
\end{equation}
where $\vert \sigma_i \vert (\Omega)$ denotes the total variation of $\sigma_i$. This problem was introduced by Beckmann in the 1950's \cite{Beckmann} and its connection with the optimal transport problem $W_1(\nu, \nu_i)$ is  well-known, as we shall recall now, referring the reader to \cite{santambrogio_optimal_2015} and \cite{ambrosio_lecture_2003} for detailed statements and proofs. First of all, let us recall that the value of \eqref{beckmannnuinu} coincides with the Wasserstein distance $W_1(\nu, \nu_i)$ so recalling the Kantorovich--Rubinstein formula, we have (and we write min and max on purpose to emphasize the existence of solutions):
\begin{equation}\label{summaryduality}
W_1(\nu_i, \nu)=\min_{\sigma_i  \in \mathcal{M}_\div(\Omega,\mathbb{R}^d)} \left\{ \vert \sigma_i \vert (\Omega) \; : \; \nabla \cdot \sigma_i +\nu_i=\nu\right\}=\max_{u_i\in \Lip_1(\Omega)} \int_{\Omega} u_i \; \d (\nu_i-\nu). 
\end{equation}
Following the seminal work of  \cite{bbs1997, bb2001}, the sharp connection between optimal flows, i.e.~solutions of \eqref{beckmannnuinu} and Kantorovich potentials is captured by the Monge--Kantorovich PDE system which we now recall. 

\begin{defi}[\textbf{Monge--Kantorovich PDE}]
A pair $(u_i, \rho_i)\in \Lip_1(\Omega)\times \mathcal{M}_+(\Omega)$ solves the Monge--Kantorovich system between $\nu_i$ and $\nu$:
\begin{equation}\label{mksysnuinu}
\nabla \cdot (\rho_i \nabla_{\rho_i} u_i)+ \nu_i=\nu, \; \vert \nabla_{\rho_i} u_i \vert =1 \mbox{ $\rho_i$-a.e.}
\end{equation}
if there exists $(u_i^\eps)_{\eps>0} \in C^1(\Omega)\cap \Lip_1(\Omega)$ converging uniformly to $u_i$ as $\eps\to 0$, such that $\nabla u_i^\eps$ converges in $L^2(\rho_i)$ to some $\hat{\sigma}_i$ (so that $\vert \hat{\sigma}_i\vert \leq 1$) and
\begin{equation}\label{divsighat}
\nabla \cdot (\rho_i \hat{\sigma}_i)+ \nu_i=\nu, \; \vert \hat{\sigma}_i  \vert =1 \mbox{ $\rho_i$-a.e.}.
\end{equation}
\end{defi}\label{mksystemnuinu}

Assume that $(u_i, \rho_i)\in \Lip_1(\Omega)\times \mathcal{M}_+(\Omega)$ solves the Monge--Kantorovich system between $\nu_i$ and $\nu$, and let  $(u_i^\eps)_{\eps>0} \in C^1(\Omega)\cap \Lip_1(\Omega)$ converge uniformly to $u_i$ as $\eps\to 0$, and be such that $\nabla u_i^\eps$ converges in $L^2(\rho_i)$ to some $\hat{\sigma}_i$ which satisfies \eqref{divsighat}, then using the fact that $\sigma_i:= \rho_i \hat{\sigma}_i$ is admissible for \eqref{beckmannnuinu} we deduce from \eqref{summaryduality} and \eqref{divsighat}:
\[\begin{split}
W_1(\nu_i, \nu) \ge \int_{\Omega} u_i \; \d(\nu_i-\nu) &= \lim_{\eps \to 0}  \int_{\Omega} u_i^\eps \; \d(\nu_i-\nu)= \lim_{\eps \to 0}  \int_{\Omega} \nabla u_i^\eps \cdot \hat{\sigma}_i  \; \d \rho_i\\
&=   \int_{\Omega} \vert  \hat{\sigma}_i \vert^2 \; \d \rho_i = \rho_i (\Omega)= \vert \sigma_i \vert(\Omega) \geq W_1(\nu_i, \nu) 
\end{split}\]
which proves that $u_i$ is a Kantorovich potential and $\sigma_i:= \rho_i \hat{\sigma}_i$ is an optimal flow:
\[W_1(\nu_i, \nu)=  \int_{\Omega} u_i \; \d(\nu_i-\nu)=\vert \sigma_i \vert(\Omega).\]
This also enables one to define unambiguously the $L^2(\rho_i)$-limit of $\nabla v_i^\eps$ for any \emph{any} approximation\footnote{Note that such approximations can easily be performed by first extending $u_i$ to a $1$-Lipschitz function to the whole of $\R^d$ and then mollifying by convolution this extension.} of $u_i$ by $C^1(\Omega)\cap \Lip_1(\Omega)$, indeed if $(v_i^\eps)_{\eps>0}$ is a sequence of such approximations, using again \eqref{divsighat}, we have:
\[\Vert \nabla v_i^\eps- \hat{\sigma}_i \Vert^2_{L^2(\rho_i)} \leq 2 \vert \sigma_i \vert(\Omega)-2 \int_\Omega \nabla v_i^\eps \cdot \hat{\sigma}_i \d \rho_i= 2 W_1(\nu_i, \nu)- 2\int_{\Omega} v_i^\eps \d (\nu_i-\nu) \to 0 \mbox{ as $\eps \to 0$}. \]
In other words, in definition \ref{mksystemnuinu}, the direction $\hat{\sigma}_i \in L^2(\rho_i)$ only depends on $\rho_i$ and $u_i$  and not on the approximation of $u_i$ and it is legitimate to set $\nabla_{\rho_i} u_i=\hat{\sigma}_i$
and to call it the tangential gradient of $u_i$ with respect to $\rho_i$ (and justify a posteriori the notation $\nabla_{\rho_i} u_i$). We have seen that solutions of the Monge--Kantorovich system yield optimal flows and optimal potentials, but the converse is easy to check. Indeed, let $u_i \in \Lip_1(\Omega)$ and $\sigma_i \in \mathcal{M}_\div(\Omega,\mathbb{R}^d)$ be such that 
\[W_1(\nu_i, \nu)=\int_{\Omega} u_i \; \d(\nu_i-\nu)=\vert \sigma_i \vert(\Omega)\]
setting $\rho_i:=\vert \sigma_i\vert$ and $\hat{\sigma}_i$ such that $\vert \hat{\sigma}_i\vert=1$ $\rho_i$-a.e.~and $\d \sigma_i = \hat{\sigma}_i \d \rho_i$, then \eqref{divsighat} holds and if $(u_i^\eps)_{\eps>0}$ is a sequence of $C^1 \cap \Lip_1$ approximations of $u_i$ then 
\[\Vert \nabla u_i^\eps- \hat{\sigma}_i \Vert^2_{L^2(\rho_i)} \leq 2 \vert \sigma_i \vert(\Omega)-2 \int_\Omega \nabla u_i^\eps \cdot \hat{\sigma}_i \d \rho_i= 2 W_1(\nu_i, \nu)- 2\int_{\Omega} u_i^\eps \d (\nu_i-\nu) \to 0 \mbox{ as $\eps \to 0$}\]
so that $(u_i, \rho_i)$ solves  the Monge--Kantorovich system \eqref{mksysnuinu} which therefore fully characterizes the primal-dual extremality relations in \eqref{summaryduality}.

Note  that if $\rho_i$ is absolutely continuous with respect to the Lebesgue measure $\rho_i \in L^1(\Omega)$, then whenever  $u_i \in \Lip_1(\Omega)$, $\rho_i \nabla u_i$ belongs to $L^1(\Omega)$ so $\nabla \cdot (\rho_i \nabla u_i)$ is well defined in the sense of distributions  and  \eqref{mksysnuinu} simplifies to 
\begin{equation}\label{mksysnuinureg}
\nabla \cdot (\rho_i \nabla u_i)+ \nu_i=\nu, \; \vert \nabla u_i \vert =1 \mbox{ $\rho_i$-a.e.}
\end{equation}

 In Monge--Kantorovich theory, $\rho_i=\vert \sigma_i\vert$ where $\sigma_i$ is an optimal flow, is called the transport density and the study of integral estimates for transport densities has been the object of an intensive stream of research \cite{DPP1, EDPP, DPP2, Santamdensity, Dweik}. In particular, if $\nu_i$ is absolutely continuous  with respect to the Lebesgue measure (and $\nu$ is an arbitrary probability measure) then the solution $\sigma_i$ of \eqref{beckmannnuinu} is unique (Theorem 4.14 and Corollary 4.15 in  \cite{santambrogio_optimal_2015})  and absolutely continuous as well (Theorem 4.16 in \cite{santambrogio_optimal_2015}) so that the transport density $\rho_i$ is in $L^1$ and the Monge--Kantorovich PDE can be understood as in \eqref{mksysnuinureg} without using the notion of tangential gradient. Higher integrability results can be found in Theorem 4.20 in \cite{santambrogio_optimal_2015}.
 
 \smallskip
 
 The connection between optimal flows, transport densities and optimal plans, is also well-known, namely given $\gamma_i \in \Pi(\nu, \nu_i)$ optimal i.e.~such that $W_1(\nu_i, \nu)=\int_{\Omega\times \Omega} \vert x-x_i\vert \d \gamma_i(x, x_i)$, define the vector valued measure $\sigma_{\gamma_i}$ by
\begin{equation}\label{defdesigggam}
\int_\Omega \phi \cdot \d \sigma_{\gamma_i}=\int_{\Omega\times \Omega} \int_0^1 \phi(x+t(x_i-x)) \cdot (x_i-x) \; \d t  \;  \d \gamma_i(x, x_i), \; \text{for all} \ \phi \in C(\Omega, \R^d).
\end{equation} 
Then, $\nabla \cdot \sigma_{\gamma_i}+\nu_i=\nu$ and $\sigma_{\gamma_i}$ is an optimal flow i.e.~solves \eqref{beckmannnuinu}, moreover (see Theorem 4.13  in \cite{santambrogio_optimal_2015}), any $\sigma_i$ solving \eqref{beckmannnuinu} is of the form $\sigma_{\gamma_i}$ for some optimal plan $\gamma_i$. We also refer to  in \cite{santambrogio_optimal_2015} and  \cite{ambrosio_lecture_2003} for more on the subject and in particular connections between optimal flows and the directions of the so-called transport rays.

	\begin{figure}[!t]
    \centering
	\begin{subfigure}{0.45\textwidth}
	\centering
	\includegraphics[width=1\textwidth]{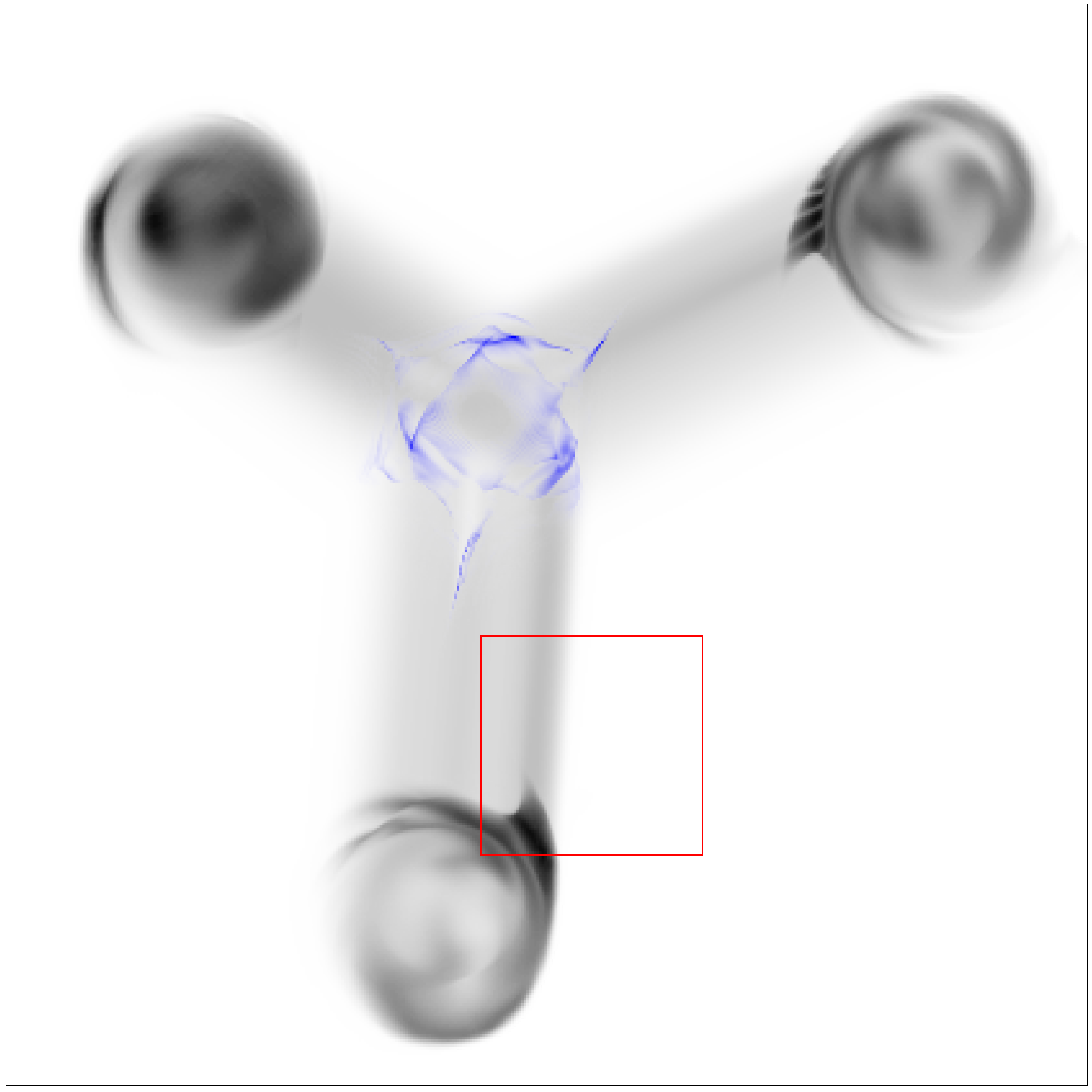}
	\caption{Superimposition of Wasserstein median, transport densities, and sample measures.}
	\end{subfigure}
	\hspace{0.5cm}
	\begin{subfigure}{0.2\textwidth}
		\centering
		\includegraphics[width=1\linewidth]{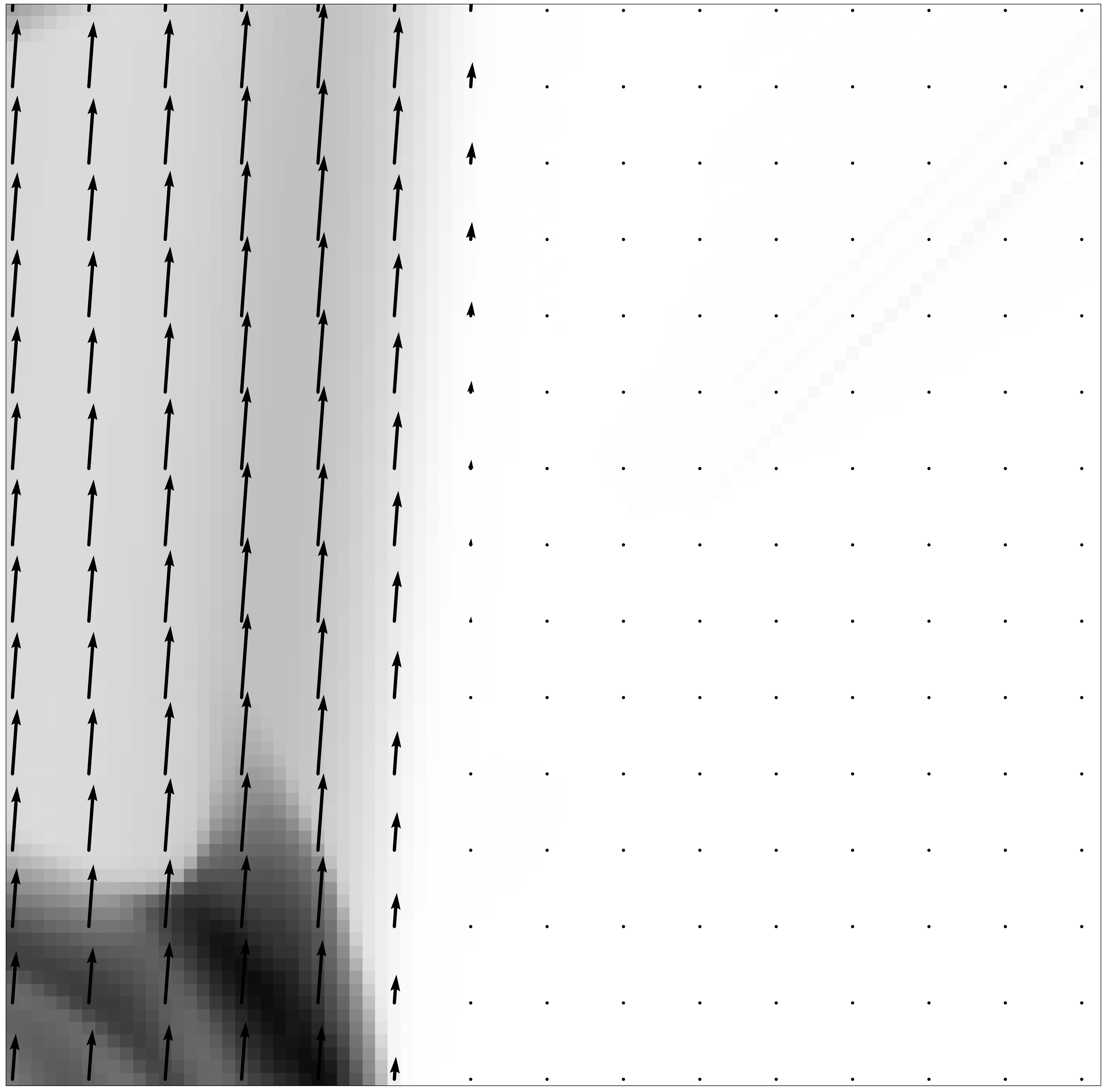}
	\caption{Zoom of the red square in (a), the arrows represent the transport flow in the corresponding point.}
	\end{subfigure}
	\caption{An approximate Wasserstein median (blue) of three sample measures (black) and the three approximately optimal transport densities (in gray) computed via Douglas--Rachford, with step-size $\tau = 10^{-1}$ and relaxation parameters $\theta_k = 1$ for all $k\in \mathbb{N}$, on a $420\times 420$ grid. The figure shows the results after $10000$ iteration, with a residual of $7 \cdot 10^{-8}$.}
	\label{fig:beckmann}
    \end{figure}

\subsection{A system of PDEs for  Wasserstein medians}

We now rewrite the Wasserstein median problem \eqref{eq:Medianprob} in terms of a multi-flow minimization:
\begin{equation}\label{eq:Beckmannform}
\inf_{(\sigma_1, \dots, \sigma_N, \nu) \in \mathcal{M}_\div(\Omb,\mathbb{R}^d)^N\times\mathcal{P}(\Omega)}  \bigg\{\sum_{j=1}^N \lambda_j  \vert \sigma_j \vert (\Omega): \nabla \cdot \sigma_j+\nu_j = \nu, \; j=1, \dots, N \bigg\},
\end{equation}
and observe that $\nu$ solves \eqref{eq:Medianprob} if and only if there exist $(\sigma_1, \dots, \sigma_N) \in \mathcal{M}_\div(\Omb,\mathbb{R}^d)^N$ such that $(\sigma_1, \dots, \sigma_N, \nu)$ solves \eqref{eq:Beckmannform}. Since we have assumed $\lambda_i>0$ we can perform the change of unknown $u_i \to u_i/\lambda_i$ in \eqref{eq:DualLip} and rewrite it as 
\begin{equation}\label{eq:Dual1Lip}
	\sup \biggl\{ \sum_{i=1}^{N} \lambda_i\int_{\Omega} u_i \d \nu_i  \; : \;   u_i \in \Lip_{1}(\Omega), \; i=1, \dots, N,  \;  \sum_{i=1}^{N}\lambda_i u_i\leq 0 \biggr\}. 
	\end{equation}
We may deduce from what we have recalled in the previous paragraph, a characterization of Wasserstein medians as well as optimal flows in \eqref{eq:Beckmannform} and optimal potentials in \eqref{eq:Dual1Lip} by a system of PDEs of Monge--Kantorovich type. Let us emphasize that a similar system of PDEs was derived in \cite{igbida} in a slightly different matching setting where there are two sample measures but also an additional capacity constraint.  Note that if  a median $\nu \in \Medl(\nu_1, \dots, \nu_N)$ was known, the problem of finding the corresponding optimal flows would  be decoupled into $N$ Monge--Kantorovich PDEs in the sense of definition \ref{mksystemnuinu}, but to determine $\nu$, we should  take into account the obstacle constraint $\sum_{i=1}^N \lambda_i u_i \leq 0$ from \eqref{eq:Dual1Lip} and the optimality condition from Theorem \ref{PrimalDualOptimalityConditions} that requires $\sum_{i=1}^N \lambda_i u_i $ to vanish on the support of $\nu$. All this can be summarized as:

\begin{theorem}[\textbf{A Monge--Kantorovich system of PDEs for medians}]
Let $\nu\in \cP(\Omega)$ then $\nu \in \Medl(\nu_1, \dots, \nu_N)$ if and only if there exist $(u_1, \dots, u_N)\in \Lip_1(\Omega)^N$ and $(\rho_1, \dots, \rho_N)\in \mathcal{M}_+(\Omega)^N$ such that, for  $i=1, \dots, N$
\begin{equation}\label{systemmkpdemedian}
\nabla \cdot (\rho_i \nabla_{\rho_i} u_i)+ \nu_i=\nu, \; \vert \nabla_{\rho_i} u_i \vert =1 \mbox{ $\rho_i$-a.e.,}
\end{equation}
coupled with the obstacle conditions
\begin{equation}\label{couplingobstacle}
 \sum_{i=1}^N \lambda_i u_i \leq 0 \; \mbox{ on $\Omega$}, \;   \sum_{i=1}^N \lambda_i u_i= 0 \;  \mbox{ on $\spt(\nu)$}
\end{equation}
Moreover in this case $(u_1,\dots, u_N)$ solves \eqref{eq:Dual1Lip} and $(\rho_1 \nabla_{\rho_1} u_1, \dots, \rho_N \nabla_{\rho_N} u_N, \nu)$ solves \eqref{eq:Beckmannform}. 

\end{theorem}\label{MKsystemmedian}

\begin{remark}[\textbf{$d=1$, $\Omega$ is an interval}] In dimension $1$, one can integrate the equation $\nabla \cdot \sigma_i +\nu_i=\nu$ and in this case, \eqref{eq:Beckmannform} appears as the vertical formulation of the median problem \eqref{eq:one_dim_v}. One can therefore interpret \eqref{eq:Beckmannform} in higher dimensions as a multidimensional extension of \eqref{eq:one_dim_v}.
\end{remark}

\begin{remark}[\textbf{Case of absolutely continuous sample measures}] If $\nu_i$ is in $L^1(\Omega)$, then the corresponding optimal flow $\sigma_i$ and transport density  $\rho_i$ are also in $L^1(\Omega)$ (even though medians need not be absolutely continuous) and one can replace the tangential gradient $\nabla_{\rho_i} u_i$ by $\nabla u_i$ in the Monge--Kantorovich PDE \eqref{systemmkpdemedian}.
\end{remark}

\begin{remark}[\textbf{Connection with the multi-marginal formulation}] If $\theta$ solves the multimarginal problem \eqref{eq:MultMargwMed}, then we know that $\nu:={\pi_0}_\#\theta$ is a median and we can recover the corresponding flows as in \eqref{defdesigggam} i.e.~by defining:
\[\int_\Omega \phi \cdot \d \sigma^\theta_{i}=\int_{\Omega^{N+1}} \int_0^1 \phi(x+t(x_i-x)) \cdot (x_i-x) \; \d t  \; \d  \theta(x, x_1, \dots, x_N), \; \text{for all} \ \phi \in C(\Omega, \R^d),\]
with this construction $(\sigma^\theta_1, \dots, \sigma^\theta_N, {\pi_0}_\#\theta)$ is a solution of \eqref{eq:Beckmannform}. In fact, invoking Theorem 4.13 in \cite{santambrogio_optimal_2015}, any solution of \eqref{eq:Beckmannform} can be obtained in this way from an optimal multi marginal plan $\theta$.
\end{remark}

\subsection{Approximation by a system of $p$-Laplace equations}

We shall now see how to approximate a median, as well as dual potentials and Beckmann flows by a  single system of $p$-Laplace equations (with $p$ large as in the seminal work of Evans and Gangbo \cite{EvansGangbo}, also see \cite{MazonRossiToledo} for a similar strategy for a matching problem involving two sample measures). Given $\epsilon>0$, we are given an exponent $\peps \geq 2d$, and assume these  exponents satisfy
\begin{equation}\label{pepsgrand}
\lim_{\epsilon \to 0^+} \peps=+\infty.
\end{equation}
We then consider the functional, $\Jeps$ defined for $u=(u_1, \ldots, u_N) \in W^{1, \peps}(\Omega)^N$ by
\[J_\epsilon(u):=\frac{1}{\peps} \sum_{i=1}^N \int_{\Omega} \vert \nabla u_i\vert^{\peps}+ \frac{1}{2 \epsilon} \int_{\Omega} \Big(\sum_{j=1}^N \lambda_j u_j\Big)_+^2 -\sum_{i=1}^N \lambda_i \int_\Omega u_i \d \nu_i,\]
observing that $\Jeps(u)=\Jeps(u+\alpha)$ if $\alpha_i$'s are constants that sum to $0$, we can add the normalizing constraint
\begin{equation}\label{normalizu}
\int_{\Omega} u_i =0, \; i=1, \ldots, N-1.
\end{equation}
With this normalization at hand we can prove the following.
\begin{prop}
Let $\epsilon > 0,$ $\peps > d$,  then 
$$\inf_{u \in W^{1, \peps}(\Omega)^N} J_\epsilon(u)$$
admits a unique solution which satisfies the normalization \eqref{normalizu}.
\end{prop}
\begin{proof}

\textbf{Existence.} First note that  for $i=1,\dots,N-1$, ${u_i \in W^{1,\peps}(\Omega)}$ with $\int_\Omega u_i=0$, using  successively Poincaré--Wirtinger's, Morrey's and Young's inequalities, we have 
\begin{align*}
&\int_{\Omega} \vert \nabla u_i\vert^{\peps} - \lambda_i \int_\Omega u_i \d \nu_i \\
\geq~ & \frac{C_\epsilon}{2} \| u_i \|^{\peps}_{W^{1,\peps}(\Omega)} - \lambda_i \| u_i \|_{L^\infty(\Omega)} \\
\geq~ & \frac{C_\epsilon}{4} \| u_i \|^{\peps}_{W^{1,\peps}(\Omega)} +C'_\epsilon \| u_i \|^{\peps}_{L^\infty(\Omega)} - \frac{\delta}{\peps} \| u_i \|_{L^\infty(\Omega)}^{\peps} - \frac{1}{\delta^{\frac{q}{\peps}}q}(\lambda_i)^q,
\end{align*}
where $C_\epsilon$, $C'_\epsilon >0$ are constants (independent of $u_i$), $\delta > 0$ and $q = \frac{\peps}{\peps-1}$ the conjugate exponent. 

To treat the $N$-th component, let $u_N \in W^{1,\peps}(\Omega)$ and define $a_N := \fint_\Omega u_N \d x$, then similarly as before
\begin{align*}
&\int_{\Omega} \vert \nabla u_N\vert^{\peps} - \lambda_N \int_\Omega u_N \d\nu_N \\
\geq~ & \frac{C_\epsilon}{2} \| u_N - a_N \|^{\peps}_{W^{1,\peps}(\Omega)} - \lambda_N \| u_N - a_N \|_{L^\infty(\Omega)} - \lambda_N a_N \\
\geq~ & \frac{C_\epsilon}{4} \| u_N - a_N \|^{\peps}_{W^{1,\peps}(\Omega)} +C'_\epsilon \| u_N - a_N \|^{\peps}_{L^\infty(\Omega)} - \frac{\delta}{\peps} \| u_N - a_N\|_{L^\infty(\Omega)}^{\peps} - \frac{1}{\delta^{\frac{q}{\peps}}q}(\lambda_N)^q - \lambda_N a_N.
\end{align*}
By choosing $\delta >0$ small enough, we obtain altogether
\begin{equation}\label{eq:lowerbdJ}
J_\epsilon(u) \geq \frac{C_\epsilon}{4}\sum_{i=1}^{N-1} \| u_i \|^{\peps}_{W^{1,\peps}(\Omega)} + \frac{C_\epsilon}{4} \| u_N - a_N \|^{\peps}_{W^{1,\peps}(\Omega)} + C - \lambda_N a_N + \frac{1}{2 \epsilon} \int_{\Omega} \Big(\sum_{j=1}^N \lambda_j u_j\Big)_+^2,
\end{equation}
where $C$ is a constant only depending on $\peps,$ $\lambda_i$ ($i=1,\dots,N$) and $C'_\epsilon$.

Now let $(u^n)_{n \in \N}=(u_1^n,\dots,u_N^n)_{n \in \N} \in \left( W^{1,\peps}(\Omega)^N \right)^\N$ be a minimizing sequence of $J_\epsilon$ satifying our normalization. In order to conclude that $(u^n)_{n \in \N}$ is bounded in $W^{1,\peps}(\Omega)$, it is enough to find an upper bound on  $a_N^n = \fint_\Omega u_N^n \d x $. Assume by contradiction, that (up to a not relabeled subsequence) $a_N^n \rightarrow + \infty$ as $n \rightarrow \infty$, then, by \eqref{eq:lowerbdJ} there are constants $K,~ \tilde C_\epsilon >0$ (independent of $n$) such that for $i=1,\dots, N-1$
\begin{equation}
\left(\frac{K + \lambda_N a_N^n}{\tilde C_\epsilon}\right)^\frac{1}{\peps} \geq \| u_i^n \|_{L^\infty(\Omega)}, \text{ and } \left(\frac{K + \lambda_N a_N^n}{\tilde C_\epsilon}\right)^\frac{1}{\peps} \geq \| u_N^n - a_N^n\|_{L^\infty(\Omega)},
\end{equation}
for all $n \in \N$.
But then, denoting $K^n_\epsilon := \left(\frac{K + \lambda_N a_N^n}{\tilde C_\epsilon}\right)^\frac{1}{\peps}$
\begin{align*}
&\frac{1}{2 \epsilon} \int_{\Omega} \Big(\sum_{j=1}^N \lambda_j u_j\Big)_+^2 - \lambda_N a^n_N  \\
\geq ~ & \frac{1}{2 \epsilon} \int_{\Omega} \left(\lambda_N a^n_N - K^n_\epsilon \right)_+^2 - \lambda_N a^n_N \\
\geq ~ & \frac{1}{2 \epsilon} \int_{\Omega} \left(\lambda_N a^n_N(1 - o(1)) \right)_+^2 - \lambda_N a^n_N \rightarrow + \infty \quad \text{as } n \rightarrow \infty,
\end{align*}
contradicting $(u^n)_{n \in \N}$ being a minimizing sequence.
This implies that $(a_N^n)_{n \in \N}$ is bounded hence $(u_N^n)_{n \in \N}$ is bounded in $W^{1,\peps}(\Omega)$. Since $(u_i^n)_{n \in \N}$ is bounded in $W^{1,\peps}(\Omega)$ for $i=1, \dots, N$, it has a subsequence that converges weakly in $W^{1,\peps}(\Omega)$, by the weak lower semi continuity of $J_\eps$,  the weak limit of this subsequence is indeed a minimizer of $J_\eps$.

\smallskip

\textbf{Uniqueness.} Let $u, \bar{u}$ be minimizers of $J_\epsilon$. Then by strict convexity of $|\cdot|^{\peps}$ and $\cdot^2$ we have
\begin{align*}
\nabla u_i &= \nabla \bar{u_i} \quad \cL^d\text{-a.e.~for } i = 1,\dots,N, \\
\Big(\sum_{j=1}^N \lambda_j u_j\Big)_+^2 &= \Big(\sum_{j=1}^N \lambda_j \bar u_j\Big)_+^2 \quad \cL^d\text{-a.e.}
\end{align*}
By the normalization \eqref{normalizu} we then get $u_i = \bar u_i$ for $i=1,\dots,N-1$, and there is $c_N \in \R$ such that $u_N = \bar u_N + c_N$. But then
\begin{align*}
0 = J_\epsilon(u) - J_\epsilon(\bar{u}) =\lambda_N \int_\Omega u_N \d \nu_N - \lambda_N \int_\Omega (u_N - c_N)\d \nu_N = \lambda_N c_N,
\end{align*}
which is only possible if $c_N = 0$.
\end{proof}
The unique minimizer of $J_\eps$ under the normalization \eqref{normalizu},  $u^\epsilon=(u_1^\epsilon, \ldots, u_N^\epsilon)$, is characterized by the system of PDEs
\begin{equation}\label{plaplacesystem}
-\nabla \cdot \Big(\vert \nabla u_i^\epsilon \vert^{\peps-2} \nabla u_i^\epsilon \Big) + \lambda_i \Big(\frac{ \sum_{j=1}^N \lambda_j \ueps_j}{\epsilon } \Big)_+=\lambda_i \nu_i, \; i=1, \ldots, N
\end{equation}
with Neumann boundary conditions, in the weak sense which means that, for every $i$ and every $\varphi \in W^{1, \peps}(\Omega)$, one has
\[\int_{\Omega} \vert \nabla u_i^\epsilon \vert^{\peps-2} \nabla u_i^\epsilon \cdot \nabla \varphi+   \lambda_i \int_{\Omega} \Big(\frac{ \sum_{j=1}^N \lambda_j \ueps_j}{\epsilon } \Big)_+ \varphi =  \lambda_i  \int_{\Omega} \varphi \d \nu_i,\]
of course, supplemented by the normalization \eqref{normalizu}.  To shorten notations and for further use, let us define
\begin{equation}\label{defapproxeps}
\sigeps_i:= \frac{ \vert \nabla u_i^\epsilon \vert^{\peps-2} \nabla u_i^\epsilon}{\lambda_i}, \; \neps:=\frac{1}{\epsilon} \Big( \sum_{j=1}^N \lambda_j \ueps_j\Big)_+.
\end{equation}
So that the optimality system \eqref{plaplacesystem} can be rewritten as
\begin{equation}\label{fluxeqeps}
-\nabla \cdot \sigeps_i + \neps=\nu_i, \; i=1, \ldots, N.
\end{equation}
In particular (testing the $N$-th equation against a constant) $\neps$ which is a nonnegative continuous (at least 1/2-H\"older when $\peps\geq 2d$) function, is a probability density on $\Omega$.

Then we have, the following convergence result:

\begin{prop}
Up to extracting a vanishing  (not explicitly written) sequence $\varepsilon_n \to 0$ as $n\to \infty$, one may assume that 
\begin{itemize}

\item $(\ueps)_{\epsilon>0}$ converges  uniformly  to some $u=(u_1, \ldots, u_N)$  which is a vector of optimal dual potentials, i.e.~solves \eqref{eq:Dual1Lip},

\item for each $i$, $(\sigeps_i)_{\epsilon>0}$ converges weakly $*$ to some vector-valued measure $\sigma_i$,  $(\neps)_{\epsilon>0}$  converges weakly $*$ to some probability measure $\nu$ and $(\sigma_1, \ldots, \sigma_N, \nu)$  solves the Beckmann problem \eqref{eq:Beckmannform}. In particular, $\nu$  is a Wasserstein median.

\end{itemize}

\end{prop}

\begin{proof}

{\textbf{Step 1: bounds on $\ueps$}.} Multiplying \eqref{plaplacesystem} by $\ueps_i$  first yields
\begin{equation}\label{estim111}
 \Vert \nabla \ueps_i\Vert_{L^{\peps}}^{\peps} + \lambda_i \int_{\Omega} \ueps_i \neps=\lambda_i \int_\Omega \ueps_i \d \nu_i, \; i=1, \ldots, N.
\end{equation}
Summing over $i$ we thus get
\begin{equation}\label{estim112}
\sum_{i=1}^N  \Vert \nabla \ueps_i\Vert_{L^{\peps}}^{\peps} + \frac{1}{\epsilon} \int_{\Omega}   \Big( \sum_{j=1}^N \lambda_j \ueps_j\Big)^2_+   =\sum_{i=1}^N\lambda_i \int_\Omega \ueps_i \d \nu_i.
\end{equation}
By Morrey's and H\"older's inequalities, $\peps \geq 2d$ and the fact that $\ueps_i$ has zero mean for $i=1, \ldots, N-1$, we have for  positive constant $C$  and $C'$ depending only on $\Omega$ (but possibly changing from one line to another)
\[ \Vert \ueps_i\Vert_{\infty} \leq C \Vert \nabla \ueps_i \Vert_{L^{2d}} \leq C \vert \Omega \vert^{\frac{1}{2d}-\frac{1}{\peps}} \Vert \nabla \ueps_i \Vert_{L^{\peps}} \leq C'  \Vert \nabla \ueps_i \Vert_{L^{\peps}},\]
which together with \eqref{estim111} and the fact that both $\neps$ and $\nu_i$ are probability measures gives 
\begin{equation}\label{estim113}
\max_{i=1, \ldots, N-1}  \Vert \ueps_i\Vert_{\infty} \leq C, \;  \max_{i=1, \ldots, N-1} \Vert \nabla \ueps_i\Vert_{L^{\peps}}^{\peps}  \leq C.
\end{equation}
Let us now get similar bounds on $\ueps_N$, using \eqref{estim111} with $i=N$ and using the fact that $\neps$ and $\nu_i$ are probability measures and then again Morrey's inequality (applied to $u_N - \fint_\Omega u_N$), we get 
\[ \Vert \nabla \ueps_N\Vert_{L^{\peps}}^{\peps} = \lambda_N \int (\ueps_N-\min_{\Omb} \ueps_N)(\nu_i -\neps) \leq \lambda_N \mathrm{osc} _{\Omb} ( \ueps_N) \leq C \Vert \nabla \ueps_N\Vert_{L^{\peps}},\]
which gives
\[ \Vert \nabla \ueps_N\Vert_{L^{\peps}}^{\peps} \leq C, \;  \mathrm{osc} _{\Omb} ( \ueps_N) \leq C.\]

With \eqref{estim113} and \eqref{estim112} and the bound on $\mathrm{osc} _{\Omb} ( \ueps_N)$, we thus get taking $C' \geq \sum_{i=1}^{N-1}\lambda_i \ueps_i$ 
\[ 0 \leq \frac{1}{\epsilon} \int_{\Omega}(\lambda_N \max_{\Omb} \ueps_N-C')_+^2 \leq C + \lambda_N \int \ueps_N \nu_N\leq C+ \lambda_N \max_{\Omb} \ueps_N\]
from which one readily deduces that $ \max_{\Omb} \ueps_N$ is bounded uniformly in $\epsilon$,  hence $(\ueps_N)_{\epsilon>0}$ is bounded in $L^\infty$ because of the bound on $\mathrm{osc} _{\Omb} ( \ueps_N)$. Finally, we have shown that 
\begin{equation}\label{estim114}
\max_{i=1, \ldots, N}  \Vert \ueps_i\Vert_{\infty} \leq C, \;  \max_{i=1, \ldots, N} \Vert \nabla \ueps_i\Vert_{L^{\peps}}^{\peps}  \leq C,
\end{equation}
which implies also $C^{0,\frac{1}{2}}$ bounds so extracting a vanishing  (not explicitly written) sequence $\varepsilon_n \to 0$ as $n\to \infty$, thanks to Ascoli--Arzel\'a's theorem, one may assume that $(u^{\epsilon})_{\epsilon>0}$ converges uniformly to some $u$ with $u\in W^{1,q}(\Omega)$ for every $q\in (1, +\infty)$. And since $(\nabla u^{\epsilon})_{\epsilon>0}$ is bounded in every $L^q$, we may also assume that  for every $q\in (1, +\infty)$,  $(\nabla u^{\epsilon})_{\epsilon>0}$ converges weakly to $\nabla u$ in $L^q(\Omega)$. Of course, we may also assume that $(\nu^{\epsilon})_{\epsilon>0}$ converges weakly $*$ to some probability measure $\nu$ and that the (bounded in $L^1$, thanks to \eqref{estim114} and the definition of   $\sigma_i^{\epsilon}$) sequence  $(\sigma_i^{\epsilon})_{\epsilon>0}$ converges weakly $*$ to some vector-valued measure $\sigma_i$.

\smallskip

{\textbf{Step 2: $u$ satisfies the constraints of the dual}.} By \eqref{estim112} and \eqref{estim114}, we have
\[ \int_{\Omega}   \Big( \sum_{j=1}^N \lambda_j \ueps_j\Big)^2_+ \leq  C\epsilon.\]
so that, letting $\epsilon \to 0^+$, we get
\[\sum_{j=1}^N \lambda_j u_j \leq 0.\]
Let us now prove that each $u_i$ is $1$-Lipschitz as a consequence of \eqref{estim114} and \eqref{pepsgrand}. First fix $q$ and let $\epsilon$ be small enough so that $\peps  \geq q$, then 
\[ \Vert \nabla \ueps_i \Vert_{L^q} \leq \vert \Omega \vert^{\frac{1}{q}-\frac{1}{\peps}} C^{\frac{1}{\peps}}.\]
So letting $\epsilon\to 0$, we get with  \eqref{pepsgrand}
\[\Vert \nabla u_i \Vert_{L^q} \leq  \vert \Omega \vert^{\frac{1}{q}}, \quad \text{for all} \ q\in (1,+\infty).\]
So letting now $q\to +\infty$ we obtain
\[\Vert \nabla u_i \Vert_{L^\infty} \leq  1\]
which implies that each $u_i$ is $1$-Lipschitz by convexity of $\Omega$.

\smallskip

{\textbf{Step 3: optimality of the limits}.} We already know that $\nu$ is a probability measure. Passing to the limit in \eqref{fluxeqeps}, we get 
\[-\nabla \cdot \sigma_i + \nu=\nu_i, \; i=1, \ldots, N,\]
which is the constraint in Beckmann problem \eqref{eq:Beckmannform}. Since $u$ is admissible in the dual, to conclude, by weak duality, it is enough to show that 
\begin{equation}\label{goal111}
\sum_{i=1}^N \lambda_i \vert \sigma_i \vert (\Omega) \leq \sum_{i=1}^N  \lambda_i \int_{\Omega} u_i \d \nu_i.
\end{equation}
First observe that \eqref{estim112} entails
\[\sum_{i=1}^N  \lambda_i \int \ueps_i \d \nu_i  \geq \sum_{i=1}^N  \int_{\Omega} \vert \nabla  \ueps_i\vert^{\peps}
 = \sum_{i=1}^N  \int_{\Omega} \vert \lambda_i \sigeps_i\vert^{\frac{\peps} {\peps-1} } .\]
Note then that, by H\"older's inequality we have 
\[ \int_{\Omega}   \vert  \sigeps_i\vert^{\frac{\peps} {\peps-1} } \geq  \Vert \sigeps_i\Vert_{L^1}^{\frac{\peps} {\peps-1} } \vert \Omega\vert^{-\frac{1}{\peps-1}} \]
so that 
\[ \liminf_{\epsilon\to 0^+}  \int_{\Omega}   \vert  \sigeps_i\vert^{\frac{\peps} {\peps-1} }\geq  \liminf_{\epsilon\to 0^+}   \Vert \sigeps_i\Vert_{L^1}   \Vert \sigeps_i\Vert_{L^1}^{\frac{1} {\peps-1} }    \geq \liminf_{\epsilon\to 0^+} \Vert \sigeps_i\Vert_{L^1} \geq  \vert \sigma_i \vert (\Omega)\]
where the second inequality is obtained by distinguishing the (obvious) case where  (after a suitable extraction) $(\sigeps_i)_{\epsilon>0}$ converges strongly to $0$ in $L^1$ and the case where   $\Vert \sigeps_i\Vert_{L^1}$ remains bounded away from $0$ and the last inequality follows from the weak $*$ convergence of $(\sigeps_i)_{\epsilon>0}$ to $\sigma_i$. We thus get
\[\sum_{i=1}^N  \lambda_i \int_{\Omega} u_i \nu_i = \liminf_{\epsilon\to 0^+} \sum_{i=1}^N  \lambda_i \int_{\Omega} \ueps_i \nu_i\geq  \liminf_{\epsilon\to 0^+} \sum_{i=1}^N  \lambda_i^{^{\frac{\peps} {\peps-1} }}  \int_{\Omega}   \vert  \sigeps_i\vert^{\frac{\peps} {\peps-1} }
\geq \sum_{i=1}^N \lambda_i \vert \sigma_i \vert(\Omega),
\]
which proves \eqref{goal111} and ends the proof. 
\end{proof}

\section{Numerics}\label{sec:numerics}

 In this section, we briefly mention the numerical methods we employed to generate the figures in the paper and present a new one based on a Douglas--Rachford scheme for the multi-flow formulation \eqref{eq:Beckmannform}. All the experiments are performed in Python on a Intel(R) Core(TM) i5-5200U CPU @ 2.20GHz and 8 Gb of RAM and are available for reproducibility at \url{https://github.com/TraDE-OPT/wasserstein-medians}.

  \subsection{Sorting, Linear Programming, Sinkhorn}
	
  Recall from Section \ref{sec:one_dim}, that in the one dimensional case, the Wasserstein median problem admits an almost-closed form solution, which can be computed directly with simple sorting procedures. We implemented these well-known schemes to generate Figure \ref{fig:one_dimensional_wm}. Here we rather focus on the case $\cX \subset \mathbb{R}^2$, which is more relevant e.g. for imaging.

  Wasserstein median problems on a fixed grid of size $n = p^2$ for a sample of size $N$ can be tackled either via Linear Programming methods, taking advantage of the minimum-cost flow nature of the problem \cite{lp2}, or via Sinkhorn-like methods on an entropy-regularized finite dimensional variant of \eqref{eq:MultMargwMed} \cite{cuturi2013SinkhornDistancesLightspeed},  see also \cite{bccmnp15, cuturi_fast_2014}. The latter represents the most popular approach. We employed the Sinkhorn method to generate Figure \ref{fig:robustness}. Despite their well-known advantages, Sinkhorn algorithm and entropic regularization methods can lead to severe computational issues, such as blurred outputs, important numerical instabilities, and  memory issues to store the so-called \emph{kernel matrix} \cite{sm_19}. It is worth mentioning that several efforts have been made to develop Sinkhorn-like methods that address these limitations, including log-space tricks for stability \cite{bccmnp15}, de-biased variants for blurring artifacts \cite{deb_bar_cuturi}, and truncation strategies for memory and speed improvements \cite{sm_19}.

	In the next paragraph,  we present a new approach which targets \eqref{eq:Beckmannform} and benefits from low memory requirements, fast convergence behaviour, and produces non-blurred approximate medians. Note, however, that this is approach, well-suited for Wasserstein medians, cannot be easily generalized to approximate Wasserstein barycenters.
	
	\subsection{Douglas--Rachford on the Beckmann formulation}
	
	Given a square domain $\Omega$, and $N\geq 2$ measures $(\nu_1, \dots, \nu_N) \in \mathcal{P}(\Omega)^N$, consider the Beckmann minimal flow formulation of the Wasserstein median problem  \eqref{eq:Beckmannform}. To discretize \eqref{eq:Beckmannform}, we introduce the square grid $\mathcal{G}_h:=\{hi \; : \; i=1,\dots, p\}^2$ with step-length $h:=1/p$, and the discrete spaces $\mathcal{M}_h:=\{\mu: \mathcal{G} \to \mathbb{R}\}$ and $\bm{\mathcal{S}}_h:=\{\boldsymbol{\sigma}: \mathcal{G}\to \mathbb{R}^2\}$. Note that $\mathcal{M}_h$ and $\bm{\mathcal{S}}_h$ are finite dimensional vector spaces which can be identified with $\mathbb{R}^n$ and $\mathbb{R}^{n\times 2}$, respectively, where $n :=p^2$. Thus, we often treat elements in $\mathcal{M}_h$ and $\bm{\mathcal{S}}_h$ as vectors. We consider the usual discretization of the gradient $\nabla_h: \mathcal{M}_h\to\bm{\mathcal{S}}_h$ defined via forward differences with homogeneous Neumann boundary conditions as in \cite[Section 6.1]{Chambolle2011}. The discrete divergence operator, which we denote by $\div_h = -\nabla_h^*$, is the opposite adjoint of $\nabla_h$, with respect to the scalar products $\langle \cdot, \cdot \rangle_{\mathcal{M}_h}$ and $\langle \cdot, \cdot \rangle_{\bm{\mathcal{S}}_h}$ (i.e.~the usual $\ell^2$ scalar products on $\mathbb{R}^n$ and $\R^{n\times 2}$, respectively). Now, let
	\begin{equation*}
	\mathcal{F}_h:=\bigl\{ (\boldsymbol{\sigma}_1, \dots, \boldsymbol{\sigma}_N, \nu)\in \bm{\mathcal{S}}_h^N\times \mathcal{M}_h \; : \; \div_h  \boldsymbol{\sigma}_k+\nu_k=\nu \ \text{for all} \ k= 1, \dots, N \bigr\},
	\end{equation*}
    where $(\nu_1, \dots, \nu_N)\in \mathcal{M}_h^N$ are suitable (not renamed) discretizations of $\nu_1, \dots, \nu_N$ on the grid $\mathcal{G}_h$. With this notation, let us consider the discretized version of \eqref{eq:Beckmannform}:
	\begin{equation}\label{eq:beckman_discrete_numerics}
	\min_{(\boldsymbol{\sigma}_1, \dots, \boldsymbol{\sigma}_N, \nu) \in \mathcal{F}_h} \ \sum_{k=1}^N \lambda_k \| \boldsymbol{\sigma}_k \|_{1,2} + \mathbb{I}_{\Delta}(\nu)
	\end{equation}
	Where $\Delta$ is the unit simplex, and $\|\cdot \|_{1,2}$ is the $\ell_{1,2}$ norm on $\bm{\mathcal{S}}_h$, also known as \emph{group-Lasso} penalty, which is defined for all $\boldsymbol{\sigma}\in \bm{\mathcal{S}}_h$ by $\|\boldsymbol{\sigma}\|_{1,2}:=\sum_{i=1}^n \|\boldsymbol{\sigma}(x_i)\|$, where $\|\cdot \|$ is the usual $\ell_2$ norm on $\mathbb{R}^n$. 
	
	To solve \eqref{eq:beckman_discrete_numerics}, we apply a Douglas--Rachford method to
	\begin{equation*}
		\min_{(\boldsymbol{\sigma}_1, \dots, \boldsymbol{\sigma}_N, \nu) \in \bm{\mathcal{S}}_h^N\times \mathcal{M}_h} \ \underbrace{\sum_{k=1}^N \lambda_k \| \boldsymbol{\sigma}_k \|_{1,2} + \mathbb{I}_{\Delta}(\nu)}_{:=g_1(\boldsymbol{\sigma}_1, \dots, \boldsymbol{\sigma}_N, \nu)}+\underbrace{\mathbb{I}_{\mathcal{F}_h}(\boldsymbol{\sigma}_1, \dots, \boldsymbol{\sigma}_N, \nu)\vphantom{\sum_{k=1}^n}}_{ :=g_2(\boldsymbol{\sigma}_1, \dots, \boldsymbol{\sigma}_N, \nu)}.
	\end{equation*}
	The Douglas--Rachford method \cite{douglas1956on, drs_mercier_lions} is an instance of the proximal point algorithm \cite{EcksteinBertsekas, bredies2021degenerate}, which can be employed to solve a minimization problem consisting of the sum of two convex lower semicontinuous functions which are accessible through evaluation of their proximity operators. In our case, the proximity operator of $g_1$, which is \emph{separable}, consists in a projection onto the unit simplex, denoted by $P_\Delta$, for the discrete measure $\nu$ and on the application of the proximity operator of the group-Lasso penalty, denoted by $\Shrink_\tau$, where $\tau >0$, on each component $\boldsymbol{\sigma}_i$, which can be computed in closed form \cite{Combettes2005}.
	
	The proximity operator of $g_2$, i.e.~the projection onto the affine subspace $\mathcal{F}_h$, is more delicate. Recall from optimality conditions that, formally, the projection onto the solution set of a linear system of the form $Ax=b$ is given, for all $y$, by $Py = y - A^*\xi$ where $\xi$ is any element that solves $AA^*\xi = Ay-b$. In our case, we have $b = -[\nu_1, \dots, \nu_N]^T$ and the linear operators $A$ and $AA^*$ can be written in block form as
	\begin{equation}\label{eq:numerics_A_and_AA*}
		A :=
		\begin{bmatrix}
		\div_h &        & & -I\\
		              & \ddots & & \vdots \\
		              &        & \div_h & -I
		\end{bmatrix},
		\quad 
		AA^* = 
		\begin{bmatrix}
		-\Delta_h+I &  I     &\cdots & I      \\
		       I    & \ddots &       & \vdots \\
	         \vdots	&        &       & 	I	  \\
	           I    &  I     & \cdots& -\Delta_h+I
		\end{bmatrix},
	\end{equation}
	where $\Delta_h:\mathcal{M}_h\to \mathcal{M}_h$ is the discrete Laplacian operator, namely $\Delta_h = \div_h\nabla_h$. 
	\begin{prop}\label{prop:projection_discrete}
		Let $\boldsymbol{\sigma} =(\boldsymbol{\sigma}_1, \dots, \boldsymbol{\sigma}_N) \in \bm{\mathcal{S}}_h^N$ and $\nu \in \mathcal{M}_h \cap \Delta$, and let $\nabla_h:\mathcal{M}_h\to \bm{\mathcal{S}}_h$ be the discrete gradient operator defined via forward differences with homogeneous Neumann boundary conditions. Then the projection $(\widetilde{\boldsymbol{\sigma}}, \widetilde{\nu})$  of $(\boldsymbol{\sigma}, \nu)$ onto $\mathcal{F}_h$ is given by
		\begin{equation*}
		\widetilde{\boldsymbol{\sigma}_i} := \boldsymbol{\sigma}_i+\nabla_h\xi_i, \quad \widetilde{\nu} := \nu+\xi_1+\cdots+\xi_N,
		\end{equation*}
		where $\displaystyle \xi_i := \xi_i'-(I-\tfrac{1}{N}\Delta_h)^{-1}\bigg(\frac{1}{N}\sum_{j=1}^{N}\xi_j'\bigg)$ and $\xi_i'$ is any solution to
		\begin{equation}\label{eq:def_xi_prime}
			-\Delta_h \xi_i' = \div_h\boldsymbol{\sigma}_i+\nu_i-\nu \quad \text{for all} \ i = 1, \dots, N.
		\end{equation}
	\end{prop}
	\begin{proof}
		First, let $i = 1, \dots, N$, let $\mathbf{1}\in \mathcal{M}_h$ be constantly equal to $1$, and note that, by definition of the scalar products, since $\nu_i, \ \nu \in \Delta$ and $\ker \nabla_h = \spn\{\mathbf{1}\}$, we get
		\begin{equation*}
		\langle \div_h\bm{\sigma}_i+\nu_i-\nu, \mathbf{1}\rangle_{\mathcal{M}_h} = \langle \div_h\bm{\sigma}_i, \mathbf{1}\rangle_{\mathcal{M}_h} =  -\langle \bm{\sigma}_i, \nabla_h \mathbf{1}\rangle_{\bm{\mathcal{S}}_h} = 0.
		\end{equation*}
		Hence, $\div_h\bm{\sigma}_i+\nu_i-\nu\in (\ker \nabla_h)^{\perp} =(\ker \Delta_h)^{\perp} = \Img \Delta_h$ for all $i = 1, \dots, N$, and, thus, \eqref{eq:def_xi_prime} actually admits a solution. From optimality conditions, we only need to show that $A(\widetilde{\boldsymbol{\sigma}}, \widetilde{\nu}) = b$ and that $\boldsymbol{\xi}:=(\xi_1, \dots, \xi_N)$ solves $AA^*\boldsymbol{\xi}=A(\boldsymbol{\sigma}, \nu)-b$ where $A$ and $AA^*$ are defined in \eqref{eq:numerics_A_and_AA*}. Let us start with the latter. Denoting $\bar{\xi}':=\tfrac{1}{N}\sum_{i=1}^N\xi_i'$, we have for all $i = 1, \dots, N$ that
		\begin{align*}
			-\Delta_h \xi_i+\xi_1+\cdots+\xi_N &= -\Delta_h \xi_i'+\Delta_h\left(I-\tfrac{1}{N}\Delta_h\right)^{-1}\bar{\xi}'+N\bar{\xi}'-N\left(I-\tfrac{1}{N}\Delta_h\right)^{-1}\bar{\xi}'\\
			& =\div_h\boldsymbol{\sigma}_i+\nu_i-\nu+N\bar{\xi}'-N\left(I-\tfrac{1}{N}\Delta_h\right)\left(I-\tfrac{1}{N}\Delta_h\right)^{-1}\bar{\xi}'\\
			& = \div_h \boldsymbol{\sigma}_i+\nu_i-\nu.
		\end{align*}
		Hence $AA^*\boldsymbol{\xi}=A(\boldsymbol{\sigma}, \nu)-b$. Regarding $A(\widetilde{\boldsymbol{\sigma}}, \widetilde{\nu}) = b$, we have
		\begin{align*}
			\div_h\widetilde{\boldsymbol{\sigma}}_i +\nu_i &= \div_h \boldsymbol{\sigma}_i+\Delta_h\xi_i +\nu_i = \div_h \boldsymbol{\sigma}_i+\Delta_h \xi_i'-\Delta_h\left(I-\tfrac{1}{N}\Delta_h\right)^{-1}\bar{\xi}'+\nu_i\\
			& = \nu -\Delta_h\left(I-\tfrac{1}{N}\Delta_h\right)^{-1}\bar{\xi}'=\nu+N\bar{\xi}'-N\left(I-\tfrac{1}{N}\Delta_h\right)^{-1}\bar{\xi}'=\widetilde{\nu},
		\end{align*}
		which concludes the proof. 
	\end{proof}
	 Proposition \ref{prop:projection_discrete} allows us to implement a Douglas--Rachford scheme on \eqref{eq:beckman_discrete_numerics}, which we summarize in Algorithm \ref{alg:drs}.
	 \begin{algorithm}[b!]
	 	\caption{Douglas--Rachford for the Wasserstein median problem}
	 	\label{alg:drs}
	 	\KwData{A collection of discrete probability measures $\nu_1, \dots, \nu_N\in \mathcal{M}_h$, a step-size $\tau >0$ and relaxation parameters $(\theta_k)_{k \in \mathbb{N}}$ in $[0,2)$ such that $\sum_{k=0}^\infty\theta_k(2-\theta_k)=+\infty$}
	 	\KwResult{$\nu^* = \lim_{k\to +\infty} \nu^k, \ \bm{\sigma}^*_q = \lim_{k\to +\infty} \bm{\sigma}_q^k$ for $q = 1, \dots, N$ solution to \eqref{eq:beckman_discrete_numerics}}
	 	\textbf{Initialize:} $\boldsymbol{\eta}_1^0, \dots, \boldsymbol{\eta}_N^0 \in \bm{\mathcal{S}}_h$ and $\mu^0 \in \mathcal{M}_h \cap \Delta$\\
	 	
	 	\While{not convergent}{
	 		$\boldsymbol{\sigma}_q^{k+1} = \Shrink_{\tau}\left(\boldsymbol{\eta}_q^{k}\right)$ for all $q =1, \dots, N$\\
	 		$\nu^{k+1} = P_{\Delta}(\mu^k)$\\
	 		\For{$q = 1, \dots, N$}{
	 		\textbf{Solve:} $\displaystyle - \Delta_h \xi_q' = \div_h (2\boldsymbol{\sigma}^{k+1}_q-\boldsymbol{\eta}^{k}_q)+\nu_q-2\nu^{k+1}+\mu^k$\\
	 		$\xi_q =  \xi_q'-\left(I-\frac{1}{N}\Delta_h\right)^{-1}\left(\frac{1}{N}\sum_{j=1}^{N}\xi_j'\right)$
	 		}
 			$\boldsymbol{\eta}_q^{k+1} = (1-\theta_k)\boldsymbol{\eta}_q^{k}+\theta_k\left(\boldsymbol{\sigma}^{k+1}_q+\nabla_h \xi_q\right)$ for all $q=1, \dots, N$\\
 			$\mu^{k+1}=(1-\theta_k)\mu^k+\theta_k\left(\nu^{k+1}+\xi_1+\cdots+\xi_N\right)$
	 		
	 	}
	 \end{algorithm}
 	Note that, in Algorithm \ref{alg:drs}, we are required to solve two \emph{sparse} (elliptic) linear systems, which we tackle with generic sparse linear solvers provided by standard Python libraries. However, one should put adequate care when trying to solve the first Laplacian system. Indeed, if the projection onto the simplex is not computed sufficiently well, the right-hand side can lie out of the range of the Laplacian. For this reason, in our numerical implementation, we smoothed out all possible numerical errors with a further projection of the right-hand side onto the set of discrete measures with a total mass equal to one.
 	
 	The computational cost required to solve the aforementioned linear systems is overall balanced with a very fast iteration-wise convergence behaviour. Remarkably, there is no need to store dense $n\times n$ matrices. This makes the proposed method suitable for highly large-scale instances, see e.g.~Figures \ref{fig:intro2} and \ref{fig:beckmann}.
 	 
 	\paragraph{Convergence.} The Douglas--Rachford splitting method benefits from robust convergence guarantees, without any condition neither on the starting point nor on the step-size $\tau>0$ \cite{drs_mercier_lions, bredies2021degenerate}. In particular, we have that if $(\boldsymbol{\sigma}_q^k)_{k\in \mathbb{N}}, \ (\nu^k)_{k\in \mathbb{N}}, \ (\boldsymbol{\eta}^k)_{k\in \mathbb{N}}$ and $(\mu^k)_{k \in \mathbb{N}}$ are the sequences generated by Algorithm \ref{alg:drs}, then for each $q= 1, \dots, N$, we have $\boldsymbol{\sigma}_q^k \to \boldsymbol{\sigma}_q^*$ and $\nu^k \to \nu^*$ and $(\boldsymbol{\sigma}_1^*, \dots, \boldsymbol{\sigma}_N^*, \nu^*)$ solves \eqref{eq:beckman_discrete_numerics}. As a stopping criterion, we measure the residual $r^k:= \sum_{q=1}^N\|\boldsymbol{\eta}_q^{k+1}-\boldsymbol{\eta}_q^{k}\|_{\bm{\mathcal{S}}_h}^2+\|\mu^{k+1}-\mu^k\|^2_{\mathcal{M}_h}$, which is guaranteed to converge to zero with a $o(k^{-1})$ worst-case rate, and we stop the iterations as soon as the residual drops below a prescribed tolerance. 
 	
 	\paragraph{Comments.} Note that to solve \eqref{eq:beckman_discrete_numerics}, we also implemented the Primal Dual Hybrid Gradient method by Chambolle and Pock, with different step-size selection strategies, such as \emph{backtracking} and \emph{adaptive} schemes \cite{NIPS2015_cd758e8f}, and several different \emph{fixed} step-sizes choices, which, however, always provided very slow behaviours, and therefore, we chose not to discuss it further. Note that for OT-like problems, the Douglas--Rachford splitting method has been employed first in  its dual formulation (ADMM) in  \cite{ppo}, then in \cite{Benamou2015}  and more recently in \cite{bcn2022}. Its extension to the Wasserstein median case proposed in the present paper has been surprisingly overlooked.
 	
    \paragraph{Acknowledgments:} E.C. would like to thank Stefano Gualandi for the helpful discussions during the preparation of this work. G.C. acknowledges the support of the Lagrange Mathematics and Computing Research Center. E.C. has received funding from the European Union’s Framework Programme for Research and Innovation Horizon 2020 (2014-2020) under the Marie Skłodowska-Curie Grant Agreement No. 861137. The Department of Mathematics and Scientific Computing, to which E.C. is affiliated, is a member of NAWI Graz (\href{https://www.nawigraz.at/en}{https://www.nawigraz.at/}). K.E. acknowledges that this project has received funding from the European Union’s Horizon 2020 research and innovation programme under the Marie Skłodowska-Curie grant agreement No. 754362. \includegraphics[width=0.5cm]{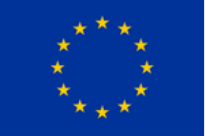}
    K.E. acknowledges funding from the ANR project ANR-20-CE40-0014.

\section*{Appendix}

\paragraph{Proof of \eqref{variationmpmrefined}.} 
Of course, if all the $x_i$'s are equal $I_\pm(\bx)=\{1, \dots, N\}$ and \eqref{variationmpmrefined} is nothing else than \eqref{variationmpm}. We may therefore assume that
\[\Delta :=\min \{ \vert x_i-x_j \vert \; : \; x_i \neq x_j \} >0.\]
Then, setting
\[\delta_i^+:=\max_k \{ y_k-x_k \; : \; x_k=x_i\}, \; \delta_i^-:=\min_k \{ y_k-x_k \; : \; x_k=x_i\},\]
and $\bdel^\pm:=(\delta_1^\pm, \dots , \delta_N^\pm)$ we have $\bx + \bdel^+ \ge \by \ge \bx+ \bdel^-$ and then by monotonicity 
\[ \Ml^\pm(\bx + \bdel^+)  \ge \Ml^{\pm}(\by)  \ge    \Ml^\pm(\bx+ \bdel^-),\]  
but if we choose $\by$ close enough to $\bx$, namely such that
\[ \max_{i, j} \vert \delta^{\pm}_i - \delta^{\pm}_j \vert \leq\frac{\Delta}{2}\]
this, together with the definition of $\Delta$, implies that the components of $\bx$ and $\bx+ \bdel^\pm$ are ordered in the same way, i.e.~$x_j < x_i$ if and only if $x_j + \delta_j^\pm < x_i  + \delta_i^\pm$. Thus, for $i\in I_+(x)$, $x_i =\Ml^+(\bx)$ and
\[ \sum_{j\; : \; x_j  +  \delta_j^- < x_i  + \delta_i^{- } } \lambda_j  =  \sum_{j\; : \; x_j  < x_i } \lambda_j \leq \frac{1}{2} , \]
so that
\[ \Ml^+(\by) \ge \Ml^+(\bx+ \bdel^-) \ge x_i + \bdel^{-}_i  \geq \Ml^+(\bx) + \min_{k\in I_+(\bx)}  (y_k-x_k).\]
In a similar way for $i\in I_-(x)$, $x_i =\Ml^-(\bx)$ and
\[ \sum_{j\; : \; x_j  +  \delta_j^- \leq  x_i  + \delta_i^{- } } \lambda_j  =  \sum_{j\; : \; x_j  \leq  x_i } \lambda_j \geq \frac{1}{2},  \]
so that
\[ \Ml^-(\by) \ge \Ml^-(\bx+ \bdel^-) \ge x_i + \bdel^{-}_i  \geq \Ml^-(\bx) + \min_{k\in I_-(\bx)}  (y_k-x_k).\]
This proves the rightmost inequalites in \eqref{variationmpmrefined}. The proof of the  leftmost inequalities in \eqref{variationmpmrefined} is similar and thus omitted.

\bibliographystyle{plain}
\bibliography{bibli}

\end{document}